\documentclass[11pt]{amsart}
\usepackage{amssymb,amsthm,amsmath}
\usepackage{hyperref}
\usepackage{latexsym}
\usepackage{tikz-cd}
\usepackage{subcaption}
\usepackage{array}
\usepackage{longtable}
\usepackage{afterpage}
\usetikzlibrary{arrows,fit,positioning,calc}
\usepackage[utf8]{inputenc}
\usepackage[margin=1in]{geometry}
\usepackage[shortlabels]{enumitem}
\usepackage{xcolor}
\usepackage{cleveref}
\usepackage{comment}

\let\oldtocsection=\tocsection
\let\oldtocsubsection=\tocsubsection

\renewcommand{\tocsection}[2]{\hspace{0em}\oldtocsection{#1}{#2}}
\renewcommand{\tocsubsection}[2]{\hspace{2em}\oldtocsubsection{#1}{#2}}

\tikzset{every picture/.style={line width=0.75pt}} %set default line width to 0.75pt

\input{math_commands.sty}

\newtheorem{thm}{Theorem}[section]
\newtheorem{prop}[thm]{Proposition}

\newtheorem{lemma}[thm]{Lemma}
\newtheorem{cor}[thm]{Corollary}
\newtheorem{dfn}[thm]{Definition}

\theoremstyle{definition} \newtheorem{ex}[thm]{Example}

\newtheorem{hyp}[thm]{Hypothesis}
\theoremstyle{definition} \newtheorem{rmk}[thm]{Remark}

\title{Clusters, toric ranks, and $2$-ranks of hyperelliptic curves in the wild case}

\author{Leonardo Fiore}
\address{\parbox{\linewidth}{Department of Mathematics ``Federigo Enriques", The University of Milan \newline Via Cesare Saldini, 50, 20133 Milano MI, Italy}}
\email{leonardo@leonardofiore.it}

\author{Jeffrey Yelton} %\orcidlink{0000-0002-8766-8438}}
\address{\parbox{\linewidth}{Department of Mathematics and Computer Science, Wesleyan University \\ 265 Church Street, Middletown, CT 06459-0128 \textbf{(corresponding author)}}}
\email{jyelton@wesleyan.edu}

\begin{document}

\maketitle

\begin{abstract}

    Given a Galois cover $Y \to X$ of smooth projective geometrically connected curves over a complete discrete valuation field $K$ with algebraically closed residue field, we define a semistable model of $Y$ over the ring of integers of a finite extension of $K$ which we call the \emph{relatively stable model} $\Yrst$ of $Y$, and we discuss its properties, focusing on the case when $Y : y^2 = f(x)$ is a hyperelliptic curve viewed as a degree-$2$ cover of the projective line $X := \proj_K^1$.  Over residue characteristic different from $2$, it follows from known results that the toric rank (i.e.\ the number of loops in the graph of components) of the special fiber of $\Yrst$ can be computed directly from the knowledge of the even-cardinality clusters of roots of the defining polynomial $f$.  We instead consider the ``wild" case of residue characteristic $2$ and demonstrate an analog to this result, showing that each even-cardinality cluster of roots of $f$ gives rise to a loop in the graph of components of the special fiber of $\Yrst$ if and only if the depth of the cluster exceeds some threshold, and we provide a computational description of and bounds for that threshold.  As a bonus, our framework also allows us to provide a formula for the $2$-rank of the special fiber of $\Yrst$.

\end{abstract}

\enlargethispage{\baselineskip}

\tableofcontents

\section{Introduction} \label{sec introduction}

Our goal is to investigate hyperelliptic curves over discrete valuation fields in the case of residue characteristic $2$ and certain arithmetic information coming from semistable reductions of these curves.  Given a complete discrete valuation field $K$ of characteristic different from $2$ with algebraically closed residue field, our starting point is to consider a \textit{hyperelliptic curve} $Y$ over $K$; that is, $Y / K$ is a smooth projective curve of positive genus admitting a degree-$2$ morphism onto the projective line $\proj_K^1$.  It is well known that an affine chart for a hyperelliptic curve $Y / K$ of genus $g \geq 1$ is given by an equation of the form 
\begin{equation} \label{eq hyperelliptic p not 2}
    y^2 = f(x) = c\prod_{i = 1}^d (x - a_i),
\end{equation}
where $f(x) \in K[x]$ is a polynomial of degree $d \in \{2g+1, 2g+2\}$ that does not have multiple roots, $c \in K^\times$ is the leading coefficient of $f(x)$, and the elements $a_i \in K^\alg$ are the roots of $f$.  We call $f$ the \emph{defining polynomial} of (this chart of) the hyperelliptic curve $Y$.  The degree-$2$ morphism of $Y$ onto the projective line is given simply by the coordinate function $x$; this morphism is branched precisely at each of the roots of $f$ as well as, in the case that $d = 2g + 1$ (in other words, when $f$ has odd degree), at the point $\infty$.  After applying an appropriate automorphism of the projective line (i.e.\ a suitable change of coordinate) which moves one of the branch points to $\infty$, we obtain an equation of the form in (\ref{eq hyperelliptic p not 2}) with $d = 2g + 1$; we will adhere to this assumption about $f$ throughout most of the paper (see \S\ref{sec models hyperelliptic} for more details).

We are broadly interested in explicitly constructing a semistable model of a given hyperelliptic curve $Y / K$ and understanding the structure of the special fiber of a semistable model of $Y$, specifically its \emph{toric rank}.  We will start by defining a particular semistable model of a hyperelliptic curve: the \textit{relatively stable model} of $Y$ (see \Cref{dfn relatively stable} below).  In this paper, we will not fully construct the relatively stable model, but we will construct enough of it to determine the toric rank of its special fiber.  As this problem is already entirely understood in the case that the residue characteristic is not $2$ and the procedure in that case can be described entirely in terms of the distances between the branch points with respect to the $p$-adic metric on $K$, we will assume that $K$ has mixed characteristic $(0, 2)$ and restrict ourselves to this case (although much of our work is easily adaptable to the case of residue characteristic different from $2$ so that the already well-known results in that situation may be recovered).  The increased complexity of the problem for this case arises from the fact that a hyperelliptic curve comes with a degree-$2$ map to the projective line: the fact that this degree is the same as the residue characteristic implies that we are in a ``wild setting".  Problems involving reduction of curves in the ``wild case", in which one studies semistable models of curves with a degree-$p$ map to the projective line over residue characteristic $p$, have been investigated in a number of works in recent decades (see \S\ref{sec introduction comparison} below), but mainly in the narrow situation where the branch points of the map $Y \to \proj_K^1$ are $p$-adically equidistant (in this case, it is already well known that the toric rank is $0$: see \cite[Lemme 3.2.1]{matignon2003vers}).  In this article, we will drop this equidistance hypothesis and instead focus on the relationship between the combinatorial data of how the branch points are ``clustered" and toric rank of the special fiber of a semistable model.

\subsection{Our main problem} \label{sec introduction main problem}

The questions we are setting out to answer ultimately arise from the following groundbreaking theorem; it was proved first by Deligne and Mumford in \cite{deligne1969irreducibility}, then through independent arguments in \cite{artin1971degenerate} (see also \cite[\S 10.4]{liu2002algebraic} for a detailed explanation of the arguments used in \cite{artin1971degenerate}), and then relatively recently in \cite{arzdorf2012another} using more constructive methods.

\begin{thm} \label{thm semistable reduction}
    Every smooth projective geometrically connected curve $C$ over $K$ achieves \emph{semistable reduction} over a finite extension $K' \supseteq K$, i.e.\ $C$ admits a model $\CC^{\sst}$ over $R'$, where $R'$ is the ring of integers in $K'$ whose special fiber is a reduced curve with at worst nodes as singularities.
\end{thm}

The above result by itself does not tell us how to find a semistable model $\CC^{\sst}$ or exactly how large an extension $K' \supseteq K$ is needed in order to define it.  It moreover does not specify, for a given curve $C / K$, anything about the structure of the special fiber $(\mathcal{C}^{\sst})_s$.  It is therefore natural to ask whether there is any general method by which we may construct a semistable model $\YY^{\sst}$ of a hyperelliptic curve $Y / K$ defined by an equation of the form in (\ref{eq hyperelliptic p not 2}).

A na\"{i}ve attempt to produce a semistable model for $Y$ would be to perform simple changes of variables (if necessary) over a low-degree field extension $K' \supset K$ so that the coefficients appearing in the equation in (\ref{eq hyperelliptic p not 2}) are all integral and then to simply use this equation to define a scheme over the corresponding ring of integers $R'$. More precisely, it is clear that after possibly scaling $x$ and $y$ by appropriate elements of $(K^\alg)^\times$, we may assume that $f$ is monic (i.e. $c=1$), and that the roots $a_i$ are all integral, with $\min_{i,j}v(a_i-a_j)=0$.  In particular, the polynomial $f$ now has integral coefficients, and so we may extend $Y$ to the scheme $\YY / R'$ whose generic fiber is $Y$ and whose special fiber $\YY_s$ is given (over the affine chart $x\neq \infty$ of $\mathbb{P}^1_k$) by the equation 
\begin{equation} \label{eq hyperelliptic p not 2 reduction}
    y^2 = \bar{f}(x) := \prod_{i = 1}^{2g + 1} (x - \bar{a}_i),
\end{equation}
where each element $\bar{a}_i$ is the reduction of $a_i \in \mathcal{O}_{K^\alg}$ in the residue field $k$.  In the situation of residue characteristic $\neq 2$, one can detect the presence of nodes and of non-nodal singularities in the special fiber of $\YY$ directly from the multiplicity of the roots of the reduced polynomial $\bar{f}$, or equivalently, from the data of the valuations of differences between roots of $f$.  By contrast, in our setting of residue characteristic $2$, the model $\YY$ is never semistable regardless of the multiplicities of the roots of $\bar{f}$: in fact, it is elementary to show that the special fiber $\YY_s$ has a non-nodal singularity at each point whose $x$-coordinate is a root of the derivative polynomial $\bar{f}'$.

A more sophisticated approach to constructing semistable models of a hyperelliptic curve $Y$ begins with finding an appropriate model $\XX$ of the projective line $X := \proj_K^1$ and (after possibly replacing the ground field with a finite extension) defining the model $\YY$ to be the normalization of $X$ in the function field of $Y$.  It is well known that there is a one-to-one correspondence between closed discs $D \subset K^\alg$ (with respect to the induced valuation on $K^\alg$) and smooth models of $\proj_K^1$ over finite extensions of $R$, and each model of $\proj_K^1$ over a finite extension of $R$ with reduced special fiber is the compositum of a finite number of smooth models and thus corresponds to a finite collection $\mathcal{D}$ of (closed) discs $D \subset K^\alg$ (see \S\ref{sec models hyperelliptic line} for more details); our ultimate goal will therefore be to find an appropriate collection $\mathcal{D}^\sst$ whose corresponding model $\XX^{(\sst)}$ of the projective line produces a semistable model $\YY^\sst$ of $Y$.

Moreover, we are interested not only in how to construct a semistable model $\YY^\sst$ of a hyperelliptic curve $Y$, but also in how certain characteristics of the defining polynomial may determine the \textit{structure} of the special fiber of such a semistable model.  The special fiber $\SF{\YY^\sst}$ of a semistable model $\YY^\sst$ of a curve $Y / K$ by definition consists of reduced components which meet each other only at nodes.  Each node, viewed as a point in $\YY^\sst$, has a \emph{thickness} (see \S\ref{sec semistable preliminaries contracting}) which is a positive integer.  The structure of the special fiber $\SF{\YY^\sst}$ can be described entirely in terms of the set of its irreducible components, the genus of the normalization of each of these components, the data of which components intersect and at at how many nodes, and the thicknesses of the nodes.  The number of loops in the configuration of components and their intersections (i.e.\ the number of loops in the dual graph of $\SF{\YY^\sst}$) is known as the \emph{toric rank} of $\SF{\YY^\sst}$, and replacing the semistable model $\YY^\sst$ of $Y$ over $R'$ with another semistable model of $Y$ over $R''$ (where $R'$ and $R''$ are the ring of integers of possibly different extensions of $K$) does not affect the toric rank.  This rank are therefore intrinsic to the curve $Y$ itself and particularly interesting to determine (meanwhile, the thicknesses of the nodes change in a predictable manner between semistable models over different extensions of $R$; see the discussion in \S\ref{sec semistable preliminaries contracting} below).

The fact that the reduction of the ``na\"{i}ve model" $\YY$ over residue characteristic $\neq 2$ depends directly on the data of the valuations of differences between roots of $f$, as discussed above, suggests that such combinatorial data may be directly crucial for constructing a semistable model of $Y$ and for understanding the structure of the special fiber of such a semistable model, both in the tame setting and in the setting of residue characteristic $2$ that we are interested in.  The combinatorial notion involved is made precise in \cite{dokchitser2022arithmetic} by defining the \emph{cluster data} associated to a hyperelliptic curve $Y$ over a discrete valuation field $K$: roughly speaking, if $Y$ is defined by an equation of the form in (\ref{eq hyperelliptic p not 2}), its associated cluster data consists of subsets $\mathfrak{s}$ of roots of the defining polynomial $f$ (called \textit{clusters}) which are closer to each other with respect to the discrete valuation of $K$ than they are to the roots of $f$ which are not contained in $\mathfrak{s}$, along with, for each non-singleton cluster $\mathfrak{s}$, the minimum valuation of differences between roots in $\mathfrak{s}$ (called the \textit{depth} of $\mathfrak{s}$).  For precise definitions, see Definition \ref{dfn cluster} below or \cite[Definition 1.1]{dokchitser2022arithmetic}.

In fact, over residue characteristic $\neq 2$, the process of constructing a semistable model $\YY^{\sst}$ of $Y$ as well as the structure of its special fiber is governed entirely by the cluster data associated to $Y$: this can be deduced from the explicit constructions given in \cite[\S4,5]{dokchitser2022arithmetic}.  In particular, under the simplifying assumption that the defining polynomial $f$ has odd degree, one may take the collection $\mathcal{D}^{\sst}$ of closed discs used to construct a semistable model $\YY^\sst$ as above to be the set of all closed discs $D \subset K^\alg$ which minimally cut out some subset of roots of the defining polynomial $f$; in the language of clusters, each of these discs corresponds to a non-singleton cluster of roots of $f$ and is the smallest disc containing that cluster.  Therefore, in this situation, there is a one-to-one correspondence between non-singleton clusters of roots of $f$ and discs in $\mathcal{D}^\sst$.  Moreover, the results cited above show that the thicknesses of the nodes of the special fiber $\SF{\YY^\sst}$ are determined by the so-called \emph{relative depths} of the clusters and that the toric rank of $\SF{\YY^\sst}$ equals the number of even-cardinality clusters which themselves are not the union of even-cardinality clusters.  Our broader goal is to seek some analog of these results in our wild setting, and our particular focus for this paper is to demonstrate an analog of the last result about toric rank.

\subsection{A summary of our main results} \label{sec introduction main results}

In our situation, where the residue characteristic of $K$ is $2$, it is natural to ask whether a semistable model of $Y$ can be constructed by a procedure governed entirely by the associated cluster data as in the setting of residue characteristic $\neq 2$.  In short, the answer is ``no", but we will show that some properties of semistable models can still be determined in certain cases by the cluster data and that there is a method of constructing a semistable model which begins with analyzing the even-cardinality clusters, which alone allows us to compute toric rank.  More precisely, in \S\ref{sec semistable relatively stable}, we define a particularly nice (unique up to unique isomorphism) semistable model $\Yrst$ of a given hyperelliptic curve $Y$ which we call the \emph{relatively stable model} (see Definition \ref{dfn relatively stable} below).  We will define a \emph{valid disc} (\Cref{dfn valid disc} below) to be a closed disc $D \subset K^\alg$ among the collection of discs used the manner discussed above to construct the semistable model $\Yrst$ (excluding such discs which correspond to components of $\SF{\Xrst}$ over which the cover $\SF{\Yrst}\to \SF{\Xrst}$ is inseparable).  In this paper we will describe a criterion for a given closed disc $D \subset K^\alg$ to be valid, and we will demonstrate a method for finding all valid discs which contain an \textit{even} number of roots of $f$ and for using these to compute the toric rank associated to $Y$.

The main results of this paper may be summarized by the following theorem, which is a combination of a (sometimes simplified version of) statements presented and proved in \S\ref{sec depths},\ref{sec toric rank}.

\begin{thm} \label{thm introduction main}
    Assume all of the above set-up for a hyperelliptic curve $Y / K$ of genus $g$ given by an equation of the form $y^2 = f(x) \in K[x]$, where the polynomial $f$ has degree $2g + 1$.  Let $\Yrst/R'$ be the relatively stable model of $Y$, where $R'$ is the ring of integers of an appropriate finite field extension $K'\supseteq K$. Let $\RR \subset K^\alg$ denote the set of roots of $f$. For any cluster of roots $\mathfrak{s} \subsetneq \RR$, we write $\mathfrak{s}'$ for the minimal cluster which properly contains $\mathfrak{s}$.
    
    The clusters of roots in $\mathcal{R}$ and the valid discs associated to $Y$ are related in the following manner.
    \begin{enumerate}[(a)]
        \item Given a valid disc $D \subseteq K^\alg$, the cardinality of $D \cap \mathcal{R}$ is even (and we may have $D \cap \mathcal{R} = \varnothing$).
        \item If a cluster $\mathfrak{s}$ has even cardinality, there are either $0$, $1$, or $2$ valid discs $D \subseteq R'$ satisfying the following property: either we have $D \cap \mathcal{R} = \mathfrak{s}$ or $D$ is the smallest disc containing $\mathfrak{s}'$.
        \item Let $\mathfrak{s}$ be an even-cardinality cluster of \emph{relative depth} 
        \begin{equation*}
            m := \min\{v(a - a') \ | \ a, a' \in \mathfrak{s}\} - \min\{v(a - a') \ | \ a, a' \in \mathfrak{s}'\}
        \end{equation*}
        (see Definition \ref{dfn cluster}), and write $f_0(x) = \prod_{a \in \mathfrak{s}} (x - a)$ and $f_\infty(x) = f(x) / f_0(x)$.  There exists a rational number $B_{f, \mathfrak{s}} \in \qq_{\geq 0}$ which is independent of the relative depth of $\mathfrak{s}$ in the sense of Remark \ref{rmk B independence}, such that 
        \begin{enumerate}[(i)]
            \item if $m > B_{f,\mathfrak{s}}$, the number of valid discs as in part (b) is ``2";
            \item if $m = B_{f,\mathfrak{s}}$, the number of valid discs as in part (b) is ``1"; and 
            \item if $m < B_{f,\mathfrak{s}}$, the number of valid discs as in part (b) is ``0".
        \end{enumerate}
        Moreover, in the case of (i), the $2$ guaranteed valid discs containing $\mathfrak{s}$ each give rise to $1$ component or to $2$ non-intersecting components of $\SF{\Yrst}$.  In the case that each gives rise to a single component of $\SF{\Yrst}$, the resulting pair of components intersects at $2$ nodes, whereas in the case that one of the valid discs gives rise to $2$ (non-intersecting) components $V_1$ and $V_2$ of $\SF{\Yrst}$, the other valid disc gives rise to a single component of $\SF{\Yrst}$ which intersects each of $V_1$ and $V_2$ at a single node.  In either case, each of these nodes has thickness equal to $(m - B_{f,\mathfrak{s}})/v(\pi)$, where $\pi$ is a uniformizer of $K'$.
        \item Given an even-cardinality cluster $\mathfrak{s}$, the bound $B_{f,\mathfrak{s}}$ from part (c) satisfies $B_{f,\mathfrak{s}} \leq 4v(2)$.  If we furthermore assume that $\mathfrak{s}$ and $\mathfrak{s}'$ each have a maximal subcluster of odd cardinality (e.g.\ a maximal subcluster which is a singleton), we have the inequality
        \begin{equation}
            B_{f,\mathfrak{s}} \geq \Big( \frac{2}{|\mathfrak{s}| - 1} + \frac{2}{2g + 1 - |\mathfrak{s}|} \Big) v(2).
        \end{equation}
        \item The toric rank of some (any) semistable model of $Y$ is equal to the number of those even-cardinality clusters satisfying item (i) above which themselves cannot be written as a disjoint union of other even-cardinality clusters satisfying item (i) above.
    \end{enumerate}
\end{thm}

The above result can be viewed as a vast generalization of the results in \cite{yelton2021semistable}, where the second author explicitly constructed semistable models of elliptic curves with a cluster of cardinality $2$ and depth $m$ (as well as elliptic curves with no even-cardinality clusters).  The threshold for $m$ above which there are $1$ or $2$ valid discs containing that cardinality-$2$ cluster which is found in \cite{yelton2021semistable} comes as the following easy corollary to the above theorem; we remark that this corollary can be deduced also from standard formulas for the $j$-invariant of an elliptic curve (specifically, the particular choice of power of $2$ multiplied to the rest of the formula, which influences the valuation of the $j$-invariant in residue characteristic $2$; see for instance the $j$-invariant formula for a Legendre curve as given in \cite[Proposition III.1.7]{silverman2009arithmetic}).

\begin{cor} \label{cor g=1 B=4v(2)}
    Suppose that we are in the $g = 1$ case of the situation in \Cref{thm introduction main} and that $\mathfrak{s}$ is a cluster of cardinality $2$.  Then we have $B_{f,\mathfrak{s}} = 4v(2)$.
\end{cor}

\begin{proof}
    The parent cluster of $\mathfrak{s}$ (i.e., the minimal cluster strictly containing $\mathfrak{s}$) is $\mathfrak{s}'=\RR$, which has cardinality 3. It is clear that both $\mathfrak{s}$ and $\mathfrak{s}'$ have a singleton child cluster (i.e., a maximal subcluster consisting of only one root). Now \Cref{thm introduction main}(d) provides the inequalities $B_{f,\mathfrak{s}}\leq 4v(2)$ and 
    \begin{equation}
        B_{f,\mathfrak{s}} \geq \big( \frac{2}{1} + \frac{2}{1} \big)v(2) = 4v(2).
    \end{equation}
    The equality $B_{f,\mathfrak{s}} = 4v(2)$ follows.
\end{proof}

\subsection{Comparison to other works} \label{sec introduction comparison}

A hyperelliptic curve is a special case of an $n$-cyclic cover of the projective line for some $n \geq 2$.  There have been a number of works discussing semistable models of such curves.  When the degree $n$ is not divisible by the residue characteristic, the process of constructing a semistable model is relatively straightforward and is provided in \cite[\S3]{bouw2015semistable}, \cite[\S4]{bouw2017computing}, \cite[\S4, 5]{dokchitser2022arithmetic} (for hyperelliptic curves, using the language of clusters), and \cite{gehrunger2021reduction} (for hyperelliptic curves, using the language of stable marked curves), as well as in earlier works.

The existing results for the wild case of semistable reduction of superelliptic curves, such as when the defining equation is of the form $y^p = f(x)$ where $p$ is the residue characteristic, have been far more limited.  To the best of our knowledge, investigations into this case began with Coleman, who in \cite{coleman1987computing} outlined an algorithm for changing coordinates in such a way that the defining equation is converted to a form whose reduction over the residue field does not describe a curve which is an inseparable degree-$p$ cover of the line; when $p = 2$, this is more or less equivalent to our notion of \emph{part-square decompositions} which will be introduced in \S\ref{sec models hyperelliptic part-square}.  This idea is further developed by Lehr and Matignon in \cite{matignon2003vers} and later in \cite{lehr2006wild} (among several other works).  Their results apply only to the very particular case of \textit{equidistant geometry}, meaning that the valuations of differences between each pair of distinct roots of the defining polynomial $f$ are all equal, which in the language of clusters means that there are no proper, non-singleton clusters of roots.  The ideas of Lehr and Matignon are generalized in \cite{arzdorf2012another}, which interprets the relevant techniques in terms of rigid analytic geometry.  The wild case is further discussed using such language in \cite[\S4]{bouw2017computing}, in which several examples are computed; the working of these examples is mainly done through clever guessing rather than a direct algorithm, however.

Our work differs from the prior research discussed above in that our major focus is on the relationship between clusters of roots and the structure of the special fiber of a semistable model of a hyperelliptic curve when the residue characteristic is $2$.  To the best of our knowledge, apart from this work, the only other works investigating any specific case in terms of cluster data which has been investigated where equidistant geometry is not assumed are the recent article \cite{dokchitser2023note}, which treats a case involving an even number of roots clustering in pairs, and the recent article \cite{gehrunger2025reduction}, in which the authors independently arrived at some similar results to ours in order to demonstrate a method of constructing the \emph{stable marked model} of a hyperelliptic curve using the associated cluster data, which those authors applied to classify the structures of such models in the case of genus $2$.

\subsection{Outline of the paper} \label{sec introduction outline}

The rest of our paper is organized as follows. First, we establish the algebro-geometric setting that we need in \S\ref{sec semistable models}, which begins with briefly providing the basic background definitions and facts relating to models of curves over local rings, and then proceeds to look more closely at the properties of the special fiber of such a model and how to compare two models of the same curve by considering (-1)-lines and (-2)-curves (see Definitions \ref{MinusLines2} and \ref{dfn minus two curve} below).  All of this set-up allows us in the following section to define a particular ``nice" semistable model of a curve $Y$ which is a Galois cover of another curve $X$, which we call the \textit{relatively stable model} $\Yrst$ of $Y$ (see \Cref{dfn relatively stable} below) and which is the main topic of \S\ref{sec semistable relatively stable}.  Viewing a hyperelliptic curve $Y / K$ as a degree-$2$ (Galois) cover of the projective line $\proj_K^1 =: X$, the relatively stable model of $Y$ is the one directly treated in the rest of this paper.

In \S\ref{sec models hyperelliptic}, we consider models of hyperelliptic curves over discrete valuation rings.  As a hyperelliptic curve is (by definition) a double cover of a projective line, we first look at models of projective lines over discrete valuation rings; the well-known characterization of such models is summarized in \S\ref{sec models hyperelliptic equations}.  Then in the rest of \S\ref{sec models hyperelliptic}, we look at models of hyperelliptic curves from the point of view of algebraic equations which define them.  More precisely, we derive equations which define normalizations of smooth models of the projective line (possibly looking over finite extensions of $K$) in the function field of the hyperelliptic curve $Y$.  We describe how we use \emph{part-square decompositions} (see \Cref{dfn qrho} below) of the defining polynomial of $Y$ to find these normalizations.

We next turn our attention to clusters in \S\ref{sec clusters}, laying out the definitions of \emph{clusters} and \emph{cluster data} as in \cite{dokchitser2022arithmetic} (and other subsequent works) as well as introducing \emph{valid discs} (see \Cref{dfn valid disc} below), which by definition correspond more or less to the smooth models of the projective line comprising the model $\Xrst$ of the projective line of which the relatively stable model $\Yrst$ is the normalization in $K(Y)$.  Our goal now is to describe a relationship between clusters and valid discs in order to prove \Cref{thm introduction main}(a)-(d); this is done in \S\ref{sec depths}, which focuses on developing a method of finding valid discs for any particular hyperelliptic curve $Y$.  The approach of \S\ref{sec depths} is to develop methods of determining the existence and find the depth of a valid disc containing a given cluster of roots.  In the course of developing these methods, given a polynomial $f$, we define lower-degree polynomials $f^{\mathfrak{s}}_+$ and $f^{\mathfrak{s}}_-$ determined by a particular even-cardinality cluster $\mathfrak{s}$ of roots of $f$ such that part-square decompositions of $f^{\mathfrak{s}}_\pm$ can be used to determine the existence and depths of valid discs containing $\mathfrak{s}$.  One of the main findings is that an even-cardinality cluster $\mathfrak{s}$ has $2$ associated valid discs if and only if the \emph{depth} of $\mathfrak{s}$ exceeds a certain ``threshold" $B_{f,\mathfrak{s}} \in \qq$ as in \Cref{thm introduction main}(c).  This threshold is constructed in \S\ref{sec depths threshold}, and in that subsection we also derive an estimate of $B_{f,\mathfrak{s}}$ which proves the inequalities in \Cref{thm introduction main}(d).

Then in \S\ref{sec toric rank}, we proceed to examine the structure of the special fiber of our desired semistable model $\Yrst$ given knowledge of the valid discs containing particular clusters of roots.  In this section, we show that in the situation of \Cref{thm introduction main}(c)(i) above, the guaranteed pair of valid discs, under certain circumstances, produces a loop in the graph of components of the special fiber of $\Yrst$, or in other words, increases the toric rank of the hyperelliptic curve by $1$.  This allows us to present (as \Cref{thm toric rank}) and prove a formula for the toric rank in terms of viable discs, as seen in \Cref{thm introduction main}(e).

In \S\ref{sec 2-rank}, we use our general set-up and results on valid discs containing clusters to derive a formula for the $2$-rank of the special fiber of the relatively stable model $\Yrst$ (\Cref{thm 2-rank}), which comes mainly from those clusters $\mathfrak{s}$ which are contained in a unique valid disc.

We finish the paper by applying our results to some examples and classes of examples in \S\ref{sec examples}.

\subsection{Notations and conventions} \label{sec introduction notation}

Below we outline our notation and conventions for this paper.

Firstly, whenever we use interval notation (e.g.\ $[a, b]$, $(a, b)$, $(a, +\infty)$, etc.), the bounds will always be elements of $\qq \cup \{\pm\infty\}$, and the interval will be understood to consist of all \textit{rational} numbers (rather than all real numbers) between the bounds; i.e.\ we have $[a, b] = [a, b] \cap \qq$; we have $[a, +\infty] = [a, +\infty) = [a, +\infty) \cap \qq$; etc.

\subsubsection{Rings, fields, and valuations}

We will adhere to the following assuptions:
\begin{itemize}
    \item $K$ is a field endowed with a discrete valuation $v : K \to \qq \cup \lbrace +\infty \rbrace$, complete with respect to $v$; when studying hyperelliptic curves over $K$ (i.e., from \S\ref{sec models hyperelliptic} on), we will always assume that $K$ has mixed characteristic $(0, 2)$;
    \item $R = \OO_K = \{z \in K \ | \ v(z) \geq 0\}$ is the ring of integers of $K$;
    \item $k$ is the residue field of $R$ (and of $K$), which we assume to be algebraically closed and of characteristic $2$;
    \item thanks to the completeness of $K$, given any algebraic extension $K' \supseteq K$, the valuation $v: K\to \qq  \cup \lbrace +\infty \rbrace$ extends uniquely to a valuation on $K'$ which we also denote by $v$: this turns $K'$ into a non-archimedean field with residue field $k$, whose ring of integers will be denoted $R':=\OO_{K'}$; when the extension $K'/K$ is finite, $K'$ is actually a complete discretely-valued field, and $R'$ is hence a complete DVR; and 
    \item $K^\alg$ is an algebraic closure of $K$.
\end{itemize}

\subsubsection{Lines, hyperelliptic curves, and models} \label{sec notation he}

Beginning in \S\ref{sec models hyperelliptic}, the symbol $X$ will normally denote the projective line $\mathbb{P}^1_K$, and $x$ will be its standard coordinate.  Similarly, beginning in \S\ref{sec models hyperelliptic}, the symbol $Y$ will in general be used to denote a hyperelliptic curve of any genus $g \ge 1$ over $K$ and ramified over $\infty\in X(K)$ and endowed with a 2-to-1 ramified cover map $Y\to X$; over the affine chart $x\neq \infty$, $Y$ can be described by an equation of the form $y^2=f(x)$, with $f(x)\in K[x]$ a polynomial of odd degree $2g+1$. The set of the $2g+1$ roots of $f(x)$ will be denoted $\RR\subseteq K^\alg$.  We  will use the notation $\Rinfty$ to mean the set of \emph{all} $2g+2$ branch points of $Y\to X$, including $\infty$.

In \S\ref{sec semistable galois covers},\ref{sec semistable relatively stable}, we work with Galois covers in greater generality, and in those sections $Y\to X$ indicates any Galois cover of smooth projective geometrically connected $K$-curves.

For convenience, we list the notation we will use relating to curves and models in Table \ref{table1} below.

\begin{longtable}{ p{3.25cm} p{10.5cm} p{1.25cm} }
\caption{Notation relating to a given curve $C / K$} \label{table1} \\
  Notation & Description & Section \\ \hline\hline
  $\CC / R$, ($\XX / R$, $\YY / R$) & a model of $C$ (or $X$ or $Y$) over the ring of integers of $K$ & \S\ref{sec semistable preliminaries} \\ \hline
  $\CC \leq \CC'$ & the model $\CC'$ \emph{dominates} the model $\CC$ & \S\ref{sec semistable preliminaries} \\ \hline
  $g(C)$ & the genus of $C$ & \S\ref{sec semistable preliminaries} \\ \hline
  $\CC_s / k$ & the special fiber of a model $\CC / R$ & \S\ref{sec semistable preliminaries} \\ \hline
  $a(V), m(V), w(V)$ & several integers attached to a component $V$ of $\CC_s$ & \S\ref{sec semistable preliminaries invariants} \\ \hline
%  $\underline{w}(V)$ & a partition of $w(V)$ for a component $V \in \Irred(\YY_s)$, coming from a $G$-action on $\YY$ & \S\ref{sec relatively stable -2-curves} \\ \hline \vspace{.25em}
  $a(\CC_s), t(\CC_s)$ & abelian and toric ranks of the special fiber of a semistable model & \S\ref{sec semistable preliminaries semistable models} \\ \hline
  $\Ctr(\CC, \CC')$ & the set of points of $\CC_s$ to which the irreducible components of $\CC'_s$ that do not appear in $\CC_s$ are contracted & \S\ref{sec semistable preliminaries contracting} \\ \hline
  $\Irred(\CC_s)$ & set of irreducible components of the special fiber $\CC_s / k$ & \S\ref{sec semistable preliminaries} \\ \hline
  $\Sing(\CC_s)$ & set of singular points of the special fiber $\CC_s / k$ & \S\ref{sec semistable preliminaries invariants} \\ \hline
  $\Cmin$, $\Cst$ & the minimal regular model and the stable model of $C$ & \S\ref{sec semistable preliminaries contracting} \\ \hline
  $\CC^{\mathrm{rst}}$ & the relatively stable model of $C$, given in \Cref{dfn relatively stable} & \S\ref{sec semistable relatively stable} \\ \hline
  $\Xmini$ & the quotient $\Ymini / G$ given a $G$-Galois cover $Y \to X$ & \S\ref{sec semistable galois covers} \\ \hline
  $\Xrst$ & the quotient $\Yrst / G$ given a $G$-Galois cover $Y \to X$ & \S\ref{sec semistable relatively stable} \\ \hline
  $\Gamma(\CC_s)$ & the dual graph of the special fiber & \S\ref{sec semistable preliminaries semistable models}
\end{longtable}

\subsubsection{Polynomials, discs, and clusters}

Let $h \in K^\alg[z]$ be a polynomial; we denote its degree by $\deg(h)$.  We extend the valuation $v : K^\alg \to \qq \cup \lbrace +\infty\rbrace$ to the \emph{Gauss valuation} $v: K^\alg[z]\to \qq \cup \lbrace +\infty\rbrace$; that is, for any polynomial $h(z) := \sum_{i = 0}^{\deg(h)} H_i z^i \in K^\alg[z]$, we set 
$$v(h) = v\left(\sum_{i = 0}^{\deg(h)} H_i z^i\right) = \min_{0 \leq i \leq {\deg(h)}} \{v(H_i)\}.$$

\begin{dfn} \label{dfn normalized reduction}
    A \emph{normalized reduction} of a nonzero polynomial $h(z) \in K^\alg[z]$ is the reduction in $k[z]$ of $\gamma^{-1}h$, where $\gamma \in (K^\alg)^{\times}$ is some scalar satisfying $v(\gamma) = v(h)$.
\end{dfn}

\begin{rmk} \label{rmk normalized reduction}

Clearly a normalized reduction of a polynomial $h(z)$ is a nonzero polynomial in $k[x]$ and is unique up to scaling; thus, the degrees of the terms appearing in the normalized reduction (which is what we will be chiefly interested in for our purposes) do not depend on the particular choice of $\gamma \in K^{\times}$ in \Cref{dfn normalized reduction}.

\end{rmk}

By a \emph{disc} (of $K^\alg$), we mean any subset of $K^\alg$ of the form $D_{\alpha,b}:=\{x\in K^\alg: v(x-\alpha)\ge b\}$ for some \emph{center} $\alpha\in K^\alg$ and \emph{depth} $b\in \qq$ (thus, our discs are always closed discs).  In this article, all depths (of clusters and of discs) will be rational numbers so that there will always exist an element of $K^\alg$ whose valuation is equal to any given depth.  Note that the depth of a disc is essentially minus a logarithm of its radius under the $p$-adic metric, and so a \textit{greater} depth corresponds to a \textit{smaller} disc.

For convenience, we list the special notation for this paper that we will use relating to polynomials, discs, and clusters in Table \ref{table2} below.

%\afterpage{
\begin{longtable}[htb]{ p{3cm} p{10cm} p{1.25cm} }
\caption{Notation relating to polynomials, discs and clusters} \label{table2} \\
Notation & Description & Section \\
\hline\hline
$x_{\alpha, \beta}$ & the coordinate obtained from $x$ through translation by $\alpha \in K^\alg$ and scaling by $\beta \in (K^\alg)^\times$, i.e.\ $x_{\alpha, \beta} = \beta^{-1} (x - \alpha)$ & \S\ref{sec models hyperelliptic line} \\ \hline
$h_{\alpha, \beta}$ & the polynomial obtained from $h$ such that $h_{\alpha, \beta}(x_{\alpha, \beta}) = h(x)$, i.e.\ $h_{\alpha, \beta}(z) = h(\beta z + \alpha)$ & \S\ref{sec models hyperelliptic line} \\ \hline
$\mathcal{X}_{\alpha, \beta}$ & the model of $X = \proj_K^1$ with coordinate $x_{\alpha, \beta}$ & \S\ref{sec models hyperelliptic line} \\ \hline
$D_{\alpha, b}$, $D_{\mathfrak{s}, b}$ & the disc of depth $b$ centered at $\alpha$ or at a point in $\mathfrak{s}$ & \S\ref{sec models hyperelliptic line} \\ \hline
$\mathcal{X}_D$ & the model of $X = \proj_K^1$ corresponding to a disc $D$ & \S\ref{sec models hyperelliptic line} \\ \hline
$\mathcal{X}_{\mathcal{D}}$ & the minimal model of $X = \proj_K^1$ dominating $\mathcal{X}_D$ for all discs $D$ in a collection $\mathcal{D}$ & \S\ref{sec models hyperelliptic line} \\ \hline
$\ell(\mathcal{X}_D, P)$ & given in \Cref{dfn ell ramification index} & \S\ref{sec models hyperelliptic separable} \\ \hline
$\underline{v}_h(D)$ & the valuation of the polynomial $h_{\alpha, \beta}$ for any $\alpha$ and $\beta$ such that $D = D_{\alpha, v(\beta)}$ & \S\ref{sec depths piecewise-linear} \\ \hline
$t_{q, \rho}$ & a rational number associated to a part-square decomposition $h = q^2 + \rho$ & \S\ref{sec models hyperelliptic part-square} \\ \hline
$\underline{t}_{q, \rho}(D)$ & the difference $\underline{v}_\rho(D) - \underline{v}_f(D)$ given a part-square decomposition $f = q^2 + \rho$ & \S\ref{sec depths piecewise-linear} \\ \hline
$\mathfrak{t}^{\mathfrak{s}}(D)$, $\mathfrak{t}^{\mathcal{R}}(D)$ & given in \Cref{dfn t fun}, applied to a cluster $\mathfrak{s}$ or to $\mathcal{R}$ & \S\ref{sec depths piecewise-linear} \\ \hline
$\mathfrak{s}'$ & the parent cluster of a cluster $\mathfrak{s}$, given in \Cref{dfn cluster} & \S\ref{sec clusters} \\ \hline
$d_\pm(\mathfrak{s}), \delta(\mathfrak{s})$ & rational numbers relating to the depth of a cluster $\mathfrak{s}$, given in \Cref{dfn cluster} & \S\ref{sec clusters} \\ \hline
$I(\mathfrak{s})$ & a subset of $\qq$ associated to a cluster $\mathfrak{s}$, given in \Cref{dfn cluster} & \S\ref{sec clusters} \\ \hline
$J(\mathfrak{s})$ & a certain sub-interval of $I(\mathfrak{s})$ containing rational numbers $b$ such that $\mathfrak{t}^{\mathcal{R}}(D_{\alpha, b}) = 2v(2)$ & \S\ref{sec depths construction valid discs} \\ \hline
$b_\pm(\mathfrak{s})$ & the endpoints of the interval $J(\mathfrak{s})$ & \S\ref{sec depths construction valid discs} \\ \hline
$\partial^\pm \mathfrak{t}^\mathcal{R}(b)$ & given in \Cref{lemma ell and t function} & \S\ref{sec depths construction valid discs} \\ \hline
$\lambda_\pm(\mathfrak{s})$ & the slopes $\mp\partial^\pm \mathfrak{t}^{\mathcal{R}}(b_\pm)$ & \S\ref{sec depths construction valid discs} \\ \hline
$f^{\mathfrak{s}}$, $f^{\mathcal{R} \smallsetminus \mathfrak{s}}$ & given by the formulas in (\ref{eq factorization}) & \S\ref{sec depths separating factorizing} \\ \hline
$f_\pm^{\mathfrak{s}}$ & given by the formulas in (\ref{eq standard form}) & \S\ref{sec depths separating std form} \\ \hline
$\mathfrak{t}_\pm^{\mathfrak{s}}(b)$ & given by the formulas in (\ref{eq mathfrak t pm}) & \S\ref{sec depths separating factorizing} \\ \hline
$b_0(\mathfrak{t}_\pm^{\mathfrak{s}})$ & the least value of $b$ at which $\mathfrak{t}_\pm^{\mathfrak{s}}$ attains $2v(2)$ & \S\ref{sec depths separating reconstructing invariants} \\ \hline
$B_{f, \mathfrak{s}}$ & the ``threshold depth" given in \Cref{prop depth threshold} & \S\ref{sec depths threshold}
\end{longtable}
%}

\subsection{Acknowledgements} \label{sec introduction acknowledgements}

The authors would like to thank Fabrizio Andreatta for proposing that the first author, as work for his Masters thesis, join the early stages of the research project of the second author, as well as for providing guidance and helpful discussions to the first author throughout his research work in the Masters program.  The authors are also grateful to an anonymous referee for a number of helpful suggestions to improve the exposition.

\section{The relatively stable model of a Galois covering} \label{sec semistable models}

The main purpose of this section is to define the relatively stable model of a Galois covering of curves and to establish some useful properties that will allow us to construct them.  We begin by recalling a number of background results on models of curves.

\subsection{Preliminaries on models of curves} \label{sec semistable preliminaries}

In this subsection, we briefly recall and develop definitions and results on semistable models of general curves over discretely-valued fields, which will later be applied to Galois coverings and hyperelliptic curves.  Our main reference for these background results is \cite{liu2002algebraic}. In this section, $C$ is a smooth, geometrically connected, projective curve over a complete discretely-valued field $K$, whose ring of integers is denoted $R$, and whose residue field $k$ is assumed to be algebraically closed (see \S\ref{sec introduction notation}).

A \emph{model} of $C$ (over $R$) is a normal, flat, projective $R$-scheme $\CC$ whose generic fiber is identified with $C$. The models of $C$ form a preordered set $\Models(C)$, the order relation being given by \emph{dominance}: given two models $\CC$ and $\CC'$ of $C$, we will write $\CC\le \CC'$ to mean that $\CC'$ dominates $\CC$, i.e.\ that the identity map $C\to C$ extends to a birational morphism $\CC'\to \CC$.

Suppose we are given a finite extension $K'/K$ (which, under our assumptions, will necessarily be totally ramified, since the residue field $k$ of $K$ is algebraically closed); let $e_{K'/K}$ be the ramification index (which coincides, in our setting, with the degree of the extension), and let $R'\supseteq R$ denote the ring of integers of $K'$. We will freely say a \emph{model of $C$ over $R'$} to mean a model of $C':=C\otimes_K K'$ over $R'$. Given a model $\CC$ of $C$ over $R$, it is possible to construct a corresponding model $\CC'$ of $C$ over $R'$, which is defined as the normalization of the base-change $\CC\otimes_R R'$; in the case that we will be interested in, namely when our model $\CC$ is a semistable model of $C$, the model has reduced special fiber and thus the scheme $\CC\otimes_R R'$ is already normal (for example, by Serre's criterion for normality).  It follows in this situation that we have $\CC'=\CC\otimes_R R'$ and that the special fibers $\CC'_s$ and $\CC_s$ are canonically isomorphic.

The special fiber $\CC_s$ of a model $\CC$ of $C$ is (geometrically) connected and consists of a number of irreducible components $V_1, \ldots, V_n$; these components are projective, possibly singular curves over the residue field $k$, each one appearing in $\CC_s$ with a certain multiplicity (which is defined as the length of the local ring of $\CC_s$ at the generic point of the component). We denote by $\Irred(\CC_s)=\{ V_1,\ldots, V_n\}$ the set of such components. It follows from the properness and flatness of the morphism $\CC \to \Spec(R)$ that the genus $g(C)$ of the smooth $K$-curve $C$ coincides with the arithmetic genus of the $k$-curve $\CC_s$.

\subsubsection{Semistable models}
\label{sec semistable preliminaries semistable models}

A model $\CC$ of $C$ is said to be \emph{semistable} if its special fiber is reduced and its singularities (if there are any) are all nodes (i.e.\ ordinary double points). More generally, we say that a model $\CC$ of $C$ is \emph{semistable} at a point $P\in \CC_s$ if $\CC_s$ is reduced at $P$ and if $P$ is either a smooth point or a node of $\CC_s$; in the latter case, the completed local ring at $P$ has the form $R[[t_1,t_2]]/(t_1 t_2 - a)$ for some $a\in R$ with $v(a)>0$. The integer $v(a)/v(\pi)\ge 1$, where $\pi$ is a uniformizer of $R$, is known as the \emph{thickness} of the node. A semistable model is regular precisely when all of its nodes have thickness equal to 1.  We remark that while regularity of a model $\CC$ is generally not preserved over an extension of $R$, the property of semistability is preserved: whenever $\CC / R$ is semistable, $\CC' / R'$ is semistable too; however, the thickness of each node of $\CC$ gets multiplied by the ramification index $e_{K'/K}$ in $\CC'$.

To describe the combinatorics of a semistable model $\CC$ of a curve $C$, one can form the \emph{dual graph} $\Gamma(\CC_s)$ of its special fiber, whose set of vertices is $\Irred(\CC_s)$ and whose edges correspond to the nodes connecting them.

When a semistable model exists, we say that $C$ has \emph{semistable reduction} over $R$. By Theorem \ref{thm semistable reduction} above, any curve $C$ is guaranteed to have semistable reduction after replacing $R$ with a large enough finite extension.

Given a semistable model $\CC$ of $C$, the (generalized) Jacobian $\Pic^0(\CC_s)$ of the (possibly singular) $k$-curve $\CC_s$ is an extension of an abelian variety $A$ by a torus $T$. The ranks of $A$ and $T$ are known respectively as the \emph{abelian rank} and the \emph{toric rank}, and we denote them respectively by $a(\CC_s)$ and $t(\CC_s)$; they are non-negative integers adding up to the genus $g(C)$ (see, for example, \cite[\S 7.5]{liu2002algebraic} or \cite[Chapters 8 and 9]{bosch2012neron}).  The following facts about these ranks are well known.
\begin{enumerate}[(a)]
    \item For all models $\CC$, the abelian rank $a(\CC_s)$ coincides with the sum $a(\CC_s)=\sum_{V\in \Irred(\CC_s)} a(V)$, where $a(V):=g(\widetilde{V})$ is the genus of the normalization $\widetilde{V}$ of $V$.
    \item If $\CC$ is a semistable model, the toric rank $t(\CC_s)$ can be computed as $\mathrm{rank}_\zz H^1(\Gamma(\CC_s), \zz)$ (see \cite[Example 9.2.8]{bosch2012neron} for a proof); this is the number of edges minus the number of vertices plus $1$.
\end{enumerate}

It is also well known that the abelian and toric ranks do not depend on the choice of semistable model of a curve $C$.

We now define the notions of (-1)-line and (-2)-line in the context of semistable models.

\begin{dfn}
    \label{MinusLines2}
    If  $\CC$ is a model, $V\in \Irred(\CC_s)$, and $\CC$ is semistable at the points of $V$, then $V$ is said to be a (-1)-line (resp.\ a (-2)-line) if it is a line (i.e. $V\cong \mathbb{P}^1_k$) and the number of nodes of $\CC_s$ lying on it is equal to 1 (resp.\ 2).
\end{dfn}

\begin{rmk}
    It is possible to show that, if $\CC$ is regular model that is semistable at the points of a component $V\in \Irred(\CC_s)$, then the definition above is consistent with the definition involving self-intersection numbers commonly given in the more general context of regular models: this follows, for example, from the formula for self-intersection numbers given in \cite[Proposition 9.1.21(b)]{liu2002algebraic}.
\end{rmk}

\subsubsection{Contracting components of special fibers} \label{sec semistable preliminaries contracting}

When $\CC$ and $\CC'$ are two models such that $\CC\le \CC'$, the image of a component $V'$ of $\CC'_s$ through the birational morphism $\CC'\to \CC$ is either a component $V$ of $\CC_s$ or a single point $P$ of $\CC_s$; in this second case, we say that the birational morphism $\CC' \to \CC$ \emph{contracts} $V'$. The rule $V'\mapsto V$ defines a one-to-one correspondence between the irreducible components of $\CC'_s$ that are not contracted by $\CC'\to \CC$ and the irreducible components of $\CC_s$; we say that $V$ is the \emph{image} of $V'$ in $\CC_s$, and $V'$ the \emph{strict transform} of $V$ in $\CC'_s$.  The birational morphism $\CC'\to \CC$ is an isomorphism precisely over the open subscheme $\CC\setminus \lbrace P_1, \ldots, P_n \rbrace$, where the $P_i$'s are the points of the special fiber of $\CC$ to which some $V'\in \Irred(\CC'_s)\setminus \Irred(\CC_s)$ contracts.

Suppose that $\CC$ and $\CC'$ are two models of $C$. We can compare them by looking at the components of their special fibers. To this aim, let us make the auxiliary choice of a model $\CC''$  dominating them both, so that we can think of $\Irred(\CC_s)$ and $\Irred(\CC'_s)$ as two subsets of a common larger set, namely $\Irred(\CC''_s)$. We denote by $\Ctr(\CC,\CC')\subseteq \CC_s(k)$ the set of points $P\in \CC_s$ such that there exists an irreducible component $V''\in \Irred(\CC''_s)$ that is the strict transform of some component of $V'\in \Irred(\CC'_s)$ and contracts to $P$.
It is clear that the formation of $\Ctr(\CC,\CC')$ does not depend on the choice of $\CC''$.
Roughly speaking, this is the way that one should think of $\Ctr(\CC,\CC')$: any given component $V\in \Irred(\CC'_s)$ is either also present in the special fiber of $\CC$ (i.e., $V\in \Irred(\CC_s)$) or it is not, in which case it is \emph{contracted} to some point $P_V\in \CC_s(k)$. The set $\Ctr(\CC,\CC')$ is simply the set of all such $P_V$'s, as $V$ varies in $\Irred(\CC'_s)\setminus \Irred(\CC_s)$. 
We clearly have that $\Ctr(\CC,\CC')=\varnothing$ (i.e., all the irreducible components of $\CC'_s$ are also present in $\CC_s$) if and only if $\CC$ dominates $\CC'$.

Given a model $\CC'$ of $C$ and any proper subset $\{ V_1, \ldots, V_n\}\subsetneq \Irred(\CC'_s)$, it is always possible to form a model $\CC$ of $C$ dominated by $\CC'$ such that the birational morphism $\CC'\to \CC$ contracts precisely the components $V_1, \ldots, V_n\in \Irred(\CC'_s)$; as a consequence, we have  $\Irred(\CC_s)=\Irred(\CC'_s)\setminus \{ V_1, \ldots, V_n\}$.

Given a finite number of models $\CC_1, \ldots, \CC_n$, one can form a minimal model $\CC$ dominating them all: it is enough to take any model $\CC'$ dominating them all, and then contract each $V\in \Irred(\CC'_s)$ that is not the strict transform of an irreducible component of $\SF{\CC_i}$ for some $i$.  It is clear that $\Irred(\CC_s)$ coincides with the (not necessarily disjoint) union $\bigcup_i \Irred((\CC_i)_s)$.

Contracting (-1)- and (-2)-lines does not ever disrupt semistability: more precisely, if $\CC'$ is a model that is semistable at the points of some components $V_1, \ldots, V_n$ of $\CC_s$, and the $V_i$'s happen to all be (-1) and (-2)-lines, then the model $\CC$ that is obtained from $\CC$ by contracting all the $V_i$'s is semistable at the points where the $V_i$'s contract. Desingularizing is also an operation that preserves semistability: if $\CC$ is semistable at a point $P\in \CC_s$, and $\CC'$ is its minimal desingularization (i.e. the regular model dominating $\CC$ which is minimal with respect to dominance), then $\CC'$ is still semistable at all points lying above $P$, and moreover, if $P$ is a node of thickness $t$, the inverse image of $\CC'_s$ at $P$ consist of a chain of $t$ nodes of thickness 1, joined by $t-1$ (-2)-lines (see \cite[Corollary 10.3.25]{liu2002algebraic}).  More generally, if $\CC$ is a model that is semistable at a point $P$ and if $\CC'$ is any model dominating $\CC$ but dominated by its minimal desingularization, then $\CC'$ is semistable at the points above $P$.

If a curve $C$ has positive genus, then, among all of its regular models, there is a minimum one (with respect to dominance).  This model is called the \emph{minimal regular model}, and we denote it by $\Cmin$; it can be characterized as the unique regular model of $C$ whose special fiber does not contain (-1)-lines.

If $C$ has semistable reduction and positive genus, its minimal regular model is semistable (by \cite[Theorem 10.3.34]{liu2002algebraic}). If $C$ has semistable reduction and genus at least $2$, then the set of its semistable models has a minimum (with respect to dominance), which is called the \emph{stable model} of $C$; it is denoted by $\Cst$ and can be characterized as the unique semistable model of $C$ whose special fiber contains neither (-1)-lines nor (-2)-lines. The stable model $\Cst$ can be obtained from $\Cmin$ by contracting all the (-2)-lines appearing in its special fiber.

\subsubsection{Invariants attached to a component of the special fiber}
\label{sec semistable preliminaries invariants}

In this section we look more closely at the special fibers of models (over $R$) of a smooth projective geometrically connected $K$-curve $C$. In \S\ref{sec semistable preliminaries invariants}, in particular, we define a number of invariants attached to each component of the special fiber of a model, while in \S\ref{sec special fibers criterion}, we use them to state a criterion that allows us to identify those models of $C$ that are part of the minimal regular model when $C$ has semistable reduction.

Given a model $\CC$ of $C$ and a component $V\in \Irred(\CC_s)$, we consider several invariants attached to $V$, listed as follows:
\begin{enumerate}
    \item $m(V)$ denotes the multiplicity of $V$ in $\CC_s$;
    \item $a(V)$ denotes the \emph{abelian rank} of $V$, i.e.\ the genus of the normalization $\widetilde{V}$ of $V$;
    \item $w(V)$ is defined only when $m(V)=1$, and it denotes the number of singular points of $\CC_s$ that belong to $V$, each one counted as many times as the number of branches of $V$ at that point; in other words, if $\tilde{V}$ is the normalization of $V$, then $w(V)$ is the number of points of $\tilde{V}$ that lie over $V\cap \Sing(\CC_s)$, where $\Sing(\CC_s)$ is the set of singular points of $\CC_s$.
\end{enumerate}

We now show that, under appropriate assumptions, the integers $m$, $a$, and $w$ are left invariant when the model is changed.
\begin{lemma}
    \label{lemma inv m a w}
    Let $\CC'$ be another model of $C$ which dominates $\CC$, and let $V'$ denote the strict transform of $V$ in $\CC'_s$. Then we have $m(V')=m(V)$ and $a(V') = a(V)$. Moreover, if $\CC'$ is dominated by the minimal desingularization of $\CC$, we also have $w(V')=w(V)$.
    \begin{proof}
        For $m$ and $a$, the lemma immediately follows from the consideration that $\CC'_s\to \CC_s$ is an isomorphism away from a finite set of points of $\CC_s$. We will now prove the result for $w$.
        
        Let $Q$ be a point of $V'$ which lies over some $P\in V$. We claim that $\CC'_s$ is smooth (resp.\ singular) at $Q$ if and only if $\CC_s$ is smooth (resp.\ singular) at $P$. This is obvious whenever $\CC'\to \CC$ is an isomorphism above $P$. If $\CC'\to \CC$ is not an isomorphism above $P$, the claim follows from the two following observations.  Firstly, since we are assuming that $\CC'$ is dominated by the minimal desingularization of $\CC$, it must be the case that $\CC$ is not regular at $P$ and consequently that $\CC_s$ is singular at $P$.  Secondly, the fiber of $\CC'\to \CC$ above $P$ is pure of dimension 1, and it consists of those components $E_i$ of $\CC'_s$ that contract to $P$; the point $Q$ will thus belong not only to $V'$, but also to one of the $E_i$'s, so that $\CC'_s$ will certainly be singular at $Q$. This completes the proof of the claim.
        
        Now, since $\CC'\to \CC$ restricts to a birational morphism $V'\to V$ of $k$-curves, the set of branches of $V$ at a point $P\in \CC_s$ equals the set of branches of $V'$ at the points of $\CC'_s$ lying above $P$. If we combine this consideration with the claim we have just proved, we have that $V'\to V$ induces a bijection between the set of the branches of $V$ at the singular points of $\CC_s$ and the set of the branches of $V'$ at the singular points of $\CC'_s$. The equality $w(V')=w(V)$ follows.
    \end{proof}
\end{lemma}

We now describe how the invariants we have defined allow us to detect (-1)-lines and (-2)-lines.
\begin{lemma}
    \label{lemma12lines}
    Let $\CC$ be any model of $C$, and let $V$ be an irreducible component of $\CC_s$. Then,
    \begin{enumerate}
        \item[(a)] if $\CC$ is regular and $V$ is a (-1)-line of multiplicity 1, then $a(V)=0$ and $w(V)=1$;
        \item[(b)] if $\CC$ is regular and $V$ is a (-2)-line of multiplicity 1, then $a(V)=0$ and $w(V)\in \lbrace 1, 2\rbrace$;
        \item[(c)] if $\CC$ is semistable at the points of $V$, then $V$ is a (-1)-line if and only if $a(V)=0$ and $w(V)=1$; and 
        \item[(d)] if $\CC$ is semistable at the points of $V$, then $V$ is a (-2)-line if and only if $a(V)=0$ and $w(V)=2$, where the reverse implication only holds when $g(C)\geq 2$.
    \end{enumerate}
    \begin{proof}
        If $V$ is a component of multiplicity 1 in the special fiber $\CC_s$ of a regular model $\CC$, then it follows from the intersection theory of regular models (see \cite[Chapter 9]{liu2002algebraic}) that its self-intersection number of $V$ is equal to minus the number of points at which $V$ intersects the other components of $\CC_s$, each counted with a certain (positive) multiplicity. Once this has been observed, parts (a) and (b) follow immediately from the usual definition of (-1)-lines and (-2)-lines for regular models involving self-intersection numbers.
        
        Suppose now that $V$ is a component of the special fiber $\CC_s$ of a model $\CC$ that is semistable at the points of $V$ (which, in particular, implies that $m(V)=1$). From the definition of the invariant $w$ and the structure of semistable models, it is clear that $w(V)$ equals the sum $2w_{\text{self}}(V) + w_{\text{other}}(V)$, where $w_{\text{self}}(V)$ is the number of self-intersections of $V$, while $w_{\text{other}}(V)$ is the number of intersections of $V$ with other components of $\CC_s$; moreover, we have $w_{\text{self}}(V)=0$ if and only if $V$ is smooth.
        But the line $\mathbb{P}^1_k$ is the unique smooth $k$-curve with abelian rank 0, so the component $V$ is a line if and only if $a(V)=0$ and $w_{\text{self}}(V)=0$; according to \Cref{MinusLines2}, the component $V$ is thus a (-1)-line or a (-2)-line if and only if $a(V)=0$, $w_{\text{self}}(V)=0$, and $w_{\text{other}}(V)$ equals $1$ or 2 respectively.
        
        From the considerations above, both implications of (c), as well the forward implication of (d), immediately follow. To prove the reverse implication of (d), one has to exclude the possibility that $a(V)=0$, $w_{\text{self}}(V)=1$, and $w_{\text{other}}(V)=0$. But if this were the case, the unique irreducible component of $\CC_s$ would be $V$ (because $w_{\text{other}}(V)=0$, but $\CC_s$ is connected), and the special fiber $\CC_s$ would consequently be a reduced $k$-curve having arithmetic genus equal to that of $V$, which is $a(V)+w_{\text{self}}(V)=1$. Since the arithmetic genus of $\CC_s$ coincides with $g(C)$, we would get $g(C)=1$; we therefore get the reverse implication of (d) as long as $g(C) \neq 1$.
    \end{proof}
\end{lemma}

Inspired by the above lemma, we make the following definition.
\begin{dfn} \label{dfn minus two curve}
    Given a model $\CC$ and an irreducible component $V$ of $\CC_s$ such that $\CC$ is semistable at the points of $V$, the component $V$ is said to be a \emph{(-2)-curve} of $\CC_s$ if $m(V)=1$, $a(V)=0$, and $w(V)=2$.
\end{dfn}
\begin{rmk}
    \label{Minus2LinesCurves}
    \Cref{lemma12lines} ensures that, if $V$ is a component of $\CC_s$ and $\CC$ is semistable at the points of $V$, then, when $V$ is (-2)-line, it is a (-2)-curve, and the converse also holds provided that $g(C)\neq 1$. If $g(C)=1$, the proof of \Cref{lemma12lines} shows that $V$ may be a (-2)-curve without being a (-2)-line, and this happens precisely when $V$ is the unique component of $\CC_s$ and it is a $k$-curve of abelian rank 0 intersecting itself once (which is to say, a projective line with two points identified).
\end{rmk}

\begin{rmk}
    It is easy to see from combining \Cref{lemma inv m a w} with \Cref{lemma12lines} that the properties of being a (-1)-line or a (-2)-curve are preserved and reflected under desingularization in the semistable case; this is the reason why the notion of a (-2)-curve (rather than a (-2)-line) will turn out to be more convenient for us.
\end{rmk}

\subsubsection{A criterion for being part of the minimal regular model}
\label{sec special fibers criterion}
As initial evidence of the usefulness of the invariants $a$, $m$, $w$ introduced before, we provide a criterion for a model $\CC$ to be part of the minimal regular model $\Cmin$ in the case that $C$ has semistable reduction.
\begin{prop}
    \label{prop part of min}
    Assume that $g(C)\ge 1$ and that $C$ has semistable reduction. Let $\CC$ be any model. Then $\CC\le \Cmin$ if and only if for each component $V$ of $\CC_s$, we have
    \begin{enumerate}[(i)]
        \item $m(V)=1$, and 
        \item $a(V)\ge 1$ or $w(V)\ge 2$.
    \end{enumerate}
    \begin{proof}
        First assume that we have $\CC\le \Cmin$.  Then the invariants $a$, $m$, and $w$ of a vertical component $V$ of $\CC$ must be equal to those of its strict transform $V^{\mathrm{min}}$ in $\Cmin$, thanks to \Cref{lemma inv m a w} (more generally, they remain the same in any model lying between $\CC\le \Cmin$). Since $\Cmin$ is semistable, its special fiber is reduced; thus, we get $m(V)=1$. Let us now assume that $a(V)=0$. If it were the case that $w(V)=0$, then $\CC_s=V$ would be a line; since the arithmetic genus of $\CC_s$ coincides with $g(C)$, this contradicts the condition that $g(C) \geq 1$. If we had $w(V)=1$, then, via \Cref{lemma12lines}(c), $V^{\mathrm{min}}$ would be a (-1)-line, which is impossible, since the minimal regular model does not contain (-1)-lines. Thus, the quantity $w(V)$ is necessarily $\geq 2$.
        
        Now assume that for each component $V$ of $\CC_s$, the conditions (i) and (ii) given in the statement hold.  Let $\CC'$ be the minimal desingularization of $\CC$. Assume by way of contradiction that $\CC'_s$ contains a (-1)-line. Since the desingularization $\CC'$ is minimal, such a (-1)-line must necessarily be the strict transform $V'$ of some component $V\in \Irred(\CC_s)$. By \Cref{lemma inv m a w}, the quantities $a(V')$, $m(V')$ and $w(V')$ are equal to $a(V)$, $m(V)$ and $w(V)$ respectively. Thus, from the condition $m(V)=1$, we deduce $m(V')=1$; since $V'$ is a (-1)-line of multiplicity 1,  Lemma \ref{lemma12lines}(a) ensures that $a(V') = 0$ and $w(V') = 1$, hence $a(V)=0$ and $w(V)=1$. But this contradicts our hypothesis, so we conclude that $\CC'_s$ cannot contain a (-1)-line.  It follows that $\CC'=\Cmin$, and we get $\CC\le \Cmin$ as desired.
    \end{proof}
\end{prop}

\subsection{Models of Galois covers} \label{sec semistable galois covers}

For the rest of this section, we specialize to the case in which the Galois cover $Y\to X$ that we considered in \S\ref{sec semistable galois covers} is the degree-$2$ map from a hyperelliptic curve $Y$ of genus $g\ge 1$ to the projective line $X := \proj_K^1$ (see \S\ref{sec notation he}).  Let $G:=\Aut_{X}(Y)$ denote the Galois group, which is the group of order $2$.  The ideas and results in the remainder of this section can be adapted much more generally to any Galois cover of smooth projective geometrically connected curves, but for our purposes it suffices to restrict ourselves to the particularly simple case of a hyperelliptic curve viewed as a $2$-cyclic cover of the projective line.

We say that a model $\YY$ of $Y$ \emph{comes from} a model $\XX$ of $X$ (or that $\YY$ is the model of $Y$ \emph{corresponding to} $\XX$) if $\YY$ is the normalization of $\XX$ in the function field $K(Y)$; in this case, the Galois group $G$ acts on $\YY$ and $\XX$ can be recovered as the quotient $\YY/G$, and moreover, the set of irreducible components $\Irred(\YY_s)$ is a $G$-set, and we have $\Irred(\XX_s)=\Irred(\YY_s)/G$.

The minimal regular model $\Ymini$ is always acted upon by $G$; we use the notation $\Xmini=\Ymini/G$ to denote the model of $X$ from which the minimal regular model $\Ymini$ comes.

\subsubsection{Vertical and horizontal (-2)-curves} \label{sec relatively stable -2-curves}

In our context of $Y \to X$ being a $2$-cover, we may distinguish two different kinds of (-2)-curves among components of $\Irred(\YY_s)$, keeping in mind that by definition a (-2)-curve has a total of exactly $2$ branches passing through singular points of $\YY_s$ (they may both pass through the same node of $\YY_s$ or each pass through a different node of $\YY_s$).

\begin{dfn} \label{defVerticalHorizontalMinusTwoLines}
    Given a (-2)-curve $V\in \Irred(\YY_s)$ such that $\YY$ is semistable at the points of $V$, we say that it is \emph{vertical} or \emph{horizontal} depending on whether the generator of the Galois group $G$ acts by transposing or fixing the $2$ branches of $V$ passing through singular points of $\YY_s$.
\end{dfn}

\begin{rmk} \label{rmkVerticalHorizontalMinusTwoLines}

The terminology \emph{vertical} and \emph{horizontal} in the above definition comes from considering that the definition immediately that, if a (-2)-curve $V$ meets the rest of $\YY_s$ at $2$ nodes, those nodes lie over the same (resp. distinct) nodes of $\XX_s$ if the curve $V$ is vertical (resp. horizontal).  This is illustrated in \Cref{fig VerticalHorizontalMinusTwoLines} below.

\end{rmk}

\begin{figure}[t]

\begin{subfigure}[t]{.4\textwidth}
\centering
\includegraphics[scale=.4]{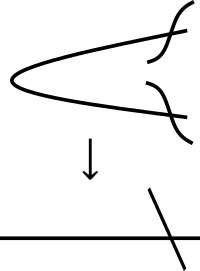}
\end{subfigure}
~
\begin{subfigure}[t]{.4\textwidth}
\centering
\includegraphics[scale=.36]{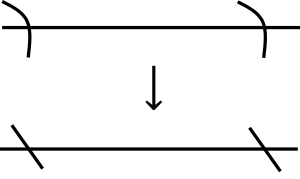}
\end{subfigure}

\caption{A component $V$ of $\YY_s$ and the other components of $\YY_s$ it intersects with, mapping to its image in $\XX_s$, where on the left (resp. right) $V$ is a vertical (resp. horizontal) (-2)-curve.}

\label{fig VerticalHorizontalMinusTwoLines}
\end{figure}

The property of being a horizontal or vertical (-2)-curve is preserved and reflected under desingularization of semistable models.
\begin{prop}
    \label{Minus12LinesStrictTransformRelative}
    Let $\YY'$ be another model acted upon by $G$ which dominates $\YY$ but is dominated by the minimal desingularization of $\YY$. Given a component $V$ of $\YY_s$ such that $\YY$ is semistable at the points of $V$, if $V'$ denotes the strict transform of $V$ in $\YY'_s$, we have that $\YY'$ is semistable at the points of $V'$; moreover, the transform $V'$ is a (-1)-line (resp.\ a horizontal (-2)-curve, resp.\ a vertical (-2)-curve) if and only if $V$ is.
    \begin{proof}
        The fact that $\YY'$ is semistable at the points of $V'$ has been discussed in \S\ref{sec semistable preliminaries contracting}.  Then the result follows from \Cref{lemma12lines}(c)(d) combined with \Cref{lemma inv m a w} and the fact that the birational map $\YY' \to \YY$ is $G$-equivariant, so that as we replace $\YY$ with $\YY'$ and $V$ with its strict transform, the set $\mathfrak{B}$ is preserved not only as a set (as in the proof of \Cref{lemma inv m a w}) but as a $G$-set.
    \end{proof}
\end{prop}

\subsubsection{Vanishing and persistent nodes} \label{sec relatively stable vanishing persistent}

Given a model $\YY$ of $Y$ corresponding to some model $\XX$ of $X$, we can ask ourselves how the properties of $\XX$ and $\YY$ are related to each other.  Writing $f: \YY\to \XX=\YY/G$ for the hyperelliptic quotient map, we present an important result concerning the semistability of $\XX$ and $\YY$.
\begin{prop}
    \label{prop vanishing persistent}
    In the setting above, we have that $\XX$ is semistable at $f(Q)$ whenever $\YY$ is semistable at some $Q\in \YY_s$. More precisely, we have the following.
    \begin{enumerate}[(a)]
        \item If $Q$ is a smooth point of $\YY_s$, then $f(Q)$ is a smooth point of $\XX_s$;
        \item If $Q$ is a node of thickness $t$ of $\YY_s$, we have two possibilities:
        \begin{enumerate}[(i)]
            \item if the Galois group $G$ fixes $Q$ and flips the two branches of ${\YY}_s$ passing through $Q$, then $f(Q)$ is a smooth point of ${\XX}_s$;
            \item if the stabilizer $G_Q \le G$ of $Q$ does not flip the two branches of ${\YY}_s$ passing through $Q$, then $f(Q)$ is a node of $\XX_s$, and its thickness is $t|G_Q|$.
        \end{enumerate}
    \end{enumerate}
    In particular, if the model $\YY$ is semistable, then so is the model $\XX$.
    \begin{proof}
        The proof consists of an explicit local study of the quotient map $f: \YY\to \XX$, which can be found in \cite[Proposition 10.3.48]{liu2002algebraic}.
    \end{proof}
\end{prop}
 
\begin{dfn}
    \label{dfn vanishing persistent node}
    A node $Q$ of $\YY$ is said to be \emph{vanishing} or \emph{persistent} with respect to the hyperelliptic cover $Y\to X$ depending on whether it falls under case (i) or (ii) of \Cref{prop vanishing persistent}(b), i.e.\ depending on whether it lies above a smooth point or a node of $\XX_s$.
\end{dfn}

We have seen in \S\ref{sec semistable preliminaries contracting} that, if $Q$ is a node of thickness $t$ of $\YY$, its inverse image in the special fiber of the minimal desingularization of $\YY$ consists of a chain of $t$ nodes of thickness 1, connected by $t-1$ (-2)-curves. More generally, if $\YY'$ is any model acted upon by $G$ that dominates $\YY$ but is dominated by its minimal desingularization, then $\YY'$ is semistable at the points lying above $Q$, and the inverse image of $Q$ in $\YY'_s$ consists of a chain of $m$ nodes $Q_1, \ldots, Q_m$, whose thicknesses add up to $t$, and $m-1$ (-2)-curves $L_1, \ldots, L_{m-1}$ connecting them. It is an interesting question to ask whether the $Q_i$'s are persistent or vanishing, and whether the $L_i$'s are horizontal or vertical.
\begin{prop}
    \label{prop galois cover desingularizing}
    In the setting above, we have the following:
    \begin{enumerate}
            \item[(a)] if $Q\in \YY_s$ is a persistent node, then the $Q_i$'s also are, and the $L_i$'s are all horizontal;
             \item[(b)] if $Q\in \YY_s$ is vanishing and $m$ is odd, then $G_Q$ permutes $L_i$ and $L_{m-i}$; the $L_i$'s are all horizontal, while the $Q_i$'s are all persistent, apart from the middle one $Q_{(m+1)/2}$ which is vanishing; and 
            \item[(c)] if $Q\in \YY_s$ is vanishing and $m$ is even, then $G_Q$ permutes each $L_i$ with $L_{m-i}$; the $L_i$'s are all horizontal, apart from the middle one $L_{m/2}$ which is vertical, while the $Q_i$'s are all persistent.
    \end{enumerate}
    \begin{proof}
        Since the $Q_i$'s and the $L_i$'s have image $Q$ in $\YY_s$, if $Q$ is not fixed by $G$, then neither are the lines $L_i$ and the nodes $P_i$ (from which it follows from definitions that the $Q_i$'s are persistent nodes and the $L_i$'s are horizontal), whereas if $Q$ is fixed by $G$, then $G$ acts on the sets $\lbrace L_i\rbrace_{1 \leq i \leq m - 1}$ and $\lbrace Q_i\rbrace_{1 \leq i \leq m}$.  We denote by $\Lambda_-$ and $\Lambda_+$ the strict transforms in $\YY'_s$ of the two branches of $\YY_s$ passing through $Q$, so that the node $Q_1$ connects $\Lambda_-$ with $L_1$ and the node $Q_m$ connects $L_{m-1}$ with $\Lambda_+$.
           
        We may now assume that $G$ fixes $Q$.  Suppose first that $Q$ is persistent so that $G$ also fixes the branches $\Lambda_+$ and $\Lambda_-$. Since $Q_1$ is the unique point of $\Lambda_-$ lying above $Q$, the point $Q_1$ is also fixed by $G$; since $g$ fixes $Q_1$ and $\Lambda_-$, it must also fix the only other branch of $\YY_s$ passing through $Q_1$, and thus it fixes $L_1$. Since $G$ fixes $Q_1$ and $L_1$, it must also fix the only other node that lies on $L_1$, namely $Q_2$.  Iterating the argument, one gets that the Galois group $G$ fixes each of the $L_i$'s and the $Q_i$'s, which again implies that the $Q_i$'s are persistent and the $L_i$'s are horizontal, proving part (a). 
            
        Now suppose that $Q$ is vanishing so that $G$ flips $\Lambda_+$ and $\Lambda_-$.  Since $Q_1$ is a point of $\Lambda_-$, the nontrivial element $g \in G$ must send it a point of $\Lambda_+$ lying above $Q$, which then must be $Q_m$. Then as we have $g\cdot \Lambda_-=\Lambda_+$, one deduces that $g\cdot L_1 = L_{m-1}$, and so on.  From this kind of iterative argument, parts (b) and (c) follow (how it ends clearly depends on whether $m$ is even or odd).
    \end{proof}
\end{prop}

\subsection{Defining the relatively stable model} \label{sec semistable relatively stable}

We are now ready to define the \emph{relatively stable model} of the hyperelliptic curve $Y$ with respect to the Galois cover $Y\to X$.
\begin{dfn} \label{dfn relatively stable}
    A model of $\YY$ of $Y$ is said to be \emph{relatively stable} with respect to the Galois cover $Y \to X$  if it is semistable, it is acted upon by $G$, and its special fiber does not contain vanishing nodes, (-1)-lines, or horizontal (-2)-curves. If a relatively stable model exists, the curve $Y$ is said to have \emph{relatively stable reduction} with respect to the cover $Y\to X$.
\end{dfn}

\begin{rmk}
    If $\YY$ is relatively stable and $\XX=\YY/G$, since $\YY$ cannot contain vanishing nodes, we have $\Sing(\YY_s)=f^{-1}(\Sing(\XX_s))$, where $f: \YY\to \XX$ is the cover map and $\Sing(\XX_s)$ and $\Sing(\YY_s)$ are the respective sets of nodes of the semistable models $\XX$ and $\YY$. 
\end{rmk}

A relatively stable model $\Yrst$ can only exist if the curve $Y$ has semistable reduction; moreover, since $\Yrst$ is semistable and contains no (-1)-lines, it is clear that $\Yrst \leq \Ymini$. It is also clear from the definition that the property of being relatively stable is preserved and reflected under arbitrary extensions of $R$.

\begin{prop}
    The relatively stable model, if it exists, is unique.
    \begin{proof}
        Let $\YY_1$ and $\YY_2$ two relatively stable models, and let $\YY_3$ be the minimal model dominating them both; by minimality, each vertical component of $\SF{\YY_3}$ is the strict transform of a component of $\SF{\YY_1}$ or of a component of $\SF{\YY_2}$.
        Since $\YY_1$ and $\YY_2$ are $\le \Ymini$, we also have $\YY_3\le \Ymini$; moreover, the model $\YY_3$ is semistable since it lies between the semistable model $\YY_1$ and its minimal desingularization $\Ymini$. Since the special fibers of $\YY_1$ and $\YY_2$ only contain vertical (-2)-curves, the same is true for $\YY_3$, thanks to \Cref{Minus12LinesStrictTransformRelative}. Suppose by way of contradiction that $\YY_3 \gneq \YY_1$: this means that some node $Q$ of $\SF{\YY_1}$ is replaced, in $\SF{\YY_3}$, by a chain of $m$ nodes and $m-1$ (-2)-curves (with $m>1$). But since $\SF{\YY_3}$ does not contain horizontal (-2)-curves, \Cref{prop galois cover desingularizing} forces $Q$ to be vanishing of thickness 2, which is a contradiction, since $\SF{\YY_1}$ does not contain vanishing nodes.
        Hence, we have $\YY_1=\YY_3$, which is to say that $\YY_1\ge \YY_2$; now we get $\YY_1= \YY_2$ by symmetry.
    \end{proof}
\end{prop}

\begin{prop} \label{prop rst exists}
    The relatively stable model exists if and only if $Y$ has semistable reduction and $\Ymini_s$ contains no vanishing nodes.  If it does exist, the relatively stable model can be formed by contracting all horizontal (-2)-curves of the minimal regular model $\Ymini$.
    
    \begin{proof}
        Suppose the curve $Y$ has semistable reduction and that $\Ymini$ contains no vanishing nodes.  As the projective line $X$ has genus $0$, the graph of components $\Gamma(\Xmini_s)$ is a tree and thus there exists a component $L$ of $\Xmini_s$ which meets the rest of $\Xmini_s$ at only $1$ node $P$.  Let $V$ be a component of $\Ymini_s$ lying over $L$.  If $a(V) = 0$, then the component $V$ must meet the rest of $\Ymini_s$ at more than one node as there are no (-1)-lines in $\Ymini$; let us name two of these nodes $Q_1$ and $Q_2$.  Since there are no vanishing nodes in $\Ymini_s$, the nodes $Q_1$ and $Q_2$ each lie over a node of $\Xmini_s$ which lies in $L$ and is thus necessarily the node $P$.  Then the Galois group must transpose the nodes $Q_1$ and $Q_2$ and thus also the branches of $V$ going through them, implying that $V$ is a vertical (-2)-curve.  From this we conclude that not all components of $\Ymini_s$ are horizontal (-2)-curves; we are consequently allowed to form a model, which we denote by $\Yrst$, by contracting all horizontal (-2)-curves of $\Ymini$, and this model will still be semistable.
        
        It is clear that $G$ acts on $\Yrst$. Suppose that  $\SF{\Yrst}$ contains a horizontal (-2)-curve.  Then its strict transform in $\SF{\Ymini}$ is still a horizontal (-2)-curve in light of  \Cref{Minus12LinesStrictTransformRelative}, which contradicts the fact that all (-2)-curves of $\SF{\Ymini}$, by construction, get contracted in $\SF{\Yrst}$.
        
        Suppose now that $\SF{\Yrst}$ contains a vanishing node.  Then, in light of \Cref{prop galois cover desingularizing}, its preimage in $\SF{\Ymini}$ must contain a vanishing node or a vertical (-2)-curve, which is impossible since $\SF{\Ymini}$ does not contain vanishing nodes by assumption, and its vertical (-2)-curves do not get contracted in $\SF{\Yrst}$ by construction. Finally, since we have $\Yrst\le \Ymini$, the special fiber $\SF{\Yrst}$ contains no (-1)-line. Henceforth, the model $\Yrst$ is actually relatively stable.
    
        Let us now prove the converse implication. Suppose that the relatively stable model $\Yrst$ exists. By definition, its special fiber $\SF{\Yrst}$ only contains persistent nodes; hence, by \Cref{prop galois cover desingularizing}, the special fiber $\SF{\Ymini}$ is obtained from $\SF{\Yrst}$ by replacing each of its node with an appropriate chain of horizontal (-2)-curves and persistent nodes and thus retains the property of not containing vanishing nodes.
    \end{proof}
\end{prop}

From now on, we use the symbol $\Yrst$ to denote the relatively stable model of $Y$ (whenever it exists) and write $\Xrst := \Yrst/G$ for the model of $X$ from which it comes in the sense of \S\ref{sec semistable galois covers}.

\Cref{prop rst exists} establishes a criterion to determine whether $Y$ has relatively stable reduction or not by looking at its minimal regular model. Hereafter we propose a refined version of such a criterion, which will be \Cref{BigCriterionJeffReduction} below.

\begin{lemma}
    \label{LemmaMinusOneLinesVanishing}
    Given a regular or a semistable model $\YY$ of $Y$, no (-1)-line of $\YY_s$ can pass through a vanishing node of $\YY_s$.
    \begin{proof}
        Suppose that $Q$ is a vanishing node, and let $L$ be a (-1)-line of $\YY_s$ passing through it. By definition, the Galois group $G$ must fix $Q$ and flip the two branches that pass through it, and thus the nontrivial Galois element will send $L$ to another (-1)-line of $\YY_s$ passing through $Q$.  One sees directly from \Cref{MinusLines2} that since the special fiber $\YY_s$ contains two intersecting (-1)-lines, it must consist only of those (-1)-lines, contradicting our assumption that $Y$ has positive genus.
    \end{proof}
\end{lemma}

\begin{prop}
    \label{BigCriterionJeffReduction}
    With the above set-up, the following are equivalent:
    \begin{itemize}
        \item[(a)] $Y$ has relatively stable reduction;
        \item[(b)] $Y$ has semistable reduction, and the vanishing nodes of all models $\YY$ of $Y$ acted upon by $G$ all have even thickness;
        \item[(c)] $Y$ has semistable reduction, and the vanishing nodes of some semistable model $\YY$ of $Y$ acted upon by $G$ all have even thickness;
        \item[(d)] $Y$ has semistable reduction, and some semistable model $\YY$ of $Y$ acted upon by $G$ does not contain any vanishing node.
    \end{itemize}
    \begin{proof}
        Let us prove (a)$\implies$(b). We will proceed by way of contradiction: we assume that there exists a model $\YY$ of $Y$ acted upon by $G$ which has a vanishing node $Q$ of odd thickness $t$; we need to prove that $\Ymini$ contains a vanishing node (see \Cref{prop rst exists}). We have that the minimal desingularization of $\YY$ is still semistable, and it also contains a vanishing node of odd thickness (see \Cref{prop galois cover desingularizing}), hence we lose no generality if we assume that $\YY$ is regular. Let us further assume, without loss of generality, that $\YY$ is minimal (with respect to dominance) among the regular semistable models of $Y$ acted upon by $G$ and carrying a vanishing node $Q$. If $\YY=\Ymini$, then we are done. If instead we have $\YY\gneq \Ymini$, then $\SF{\YY}$ contains a $G$-orbit of (-1)-lines $\lbrace gL: g\in G \rbrace$; let $\YY_1$ be the semistable model that is obtained by contracting it. Since neither of the $gL$ can pass through the vanishing node $Q$ by \Cref{LemmaMinusOneLinesVanishing}, the birational map $f: \YY\to \YY_1$ is an isomorphism above $Q_1:=f(Q)$; in particular, the node $Q_1$ is still a vanishing node of odd thickness of the model $\YY_1$, which, by construction, is again regular, semistable and acted upon by $G$; this violates the minimality of $\YY$.
        
        The implication (b)$\implies$(c) is obvious; let us therefore proceed to proving (c)$\implies$(d). Let $\YY$ be some semistable model of $Y$ whose vanishing nodes all have even thickness; by \Cref{prop galois cover desingularizing}, the minimal desingularization $\YY'$ of $\YY$ will not contain a vanishing node, whence (d) follows. 
        
        Let us finally prove (d)$\implies$(a). If $\YY$ is a model acted upon by $G$ which contains no vanishing nodes, its minimal desingularization has the same property by \Cref{prop galois cover desingularizing}; hence we can assume without losing generality that $\YY$ is regular, and that it is moreover minimal (with respect to dominance) in the set of the regular semistable models of $Y$ acted upon by $G$ that do not contain vanishing nodes. If $\YY=\Ymini$, we are done by \Cref{prop rst exists}. If instead we have $\YY\gneq \Ymini$, then the special fiber $\YY_s$ contains a $G$-orbit of (-1)-lines $\lbrace gL: g\in G \rbrace$; if $\YY_1$ is the model we obtain by contracting the lines in this orbit, the birational map $\YY\to \YY_1$ is clearly an isomorphism above all nodes of $\YY_1$.  The model $\YY_1$ will also not contain any vanishing nodes, and by construction it is still semistable and regular; this violates the minimality of $\YY$.
    \end{proof}
\end{prop}

\begin{cor}
    The curve $Y$ always has relatively stable reduction over a large enough finite extension of $R$.
    \begin{proof}
        After possibly extending $R$, we may assume that $Y$ has semistable reduction over $R$ by \Cref{thm semistable reduction}. Let $\YY$ be any semistable model of $Y$ acted upon by $G$. If the model $\YY$ does not contain a vanishing node of odd thickness, then the curve $Y$ has relatively stable reduction over $R$ by \Cref{BigCriterionJeffReduction}. Otherwise, let $R'$ be any extension of $R$ with even ramification index $e$. If we base-change $\YY$ to $R'$, we still have a semistable model, and the thicknesses of the nodes all get multiplied by $e$. Hence, all nodes of $\YY_{R'}$ have even thickness, and $Y$ has relatively stable reduction over $R'$ thanks to \Cref{BigCriterionJeffReduction}.
    \end{proof}
\end{cor}

The following result is the analog of \Cref{prop part of min} for the relatively stable model (instead of the minimal regular one).

\begin{prop} \label{prop part of rst}
    With the above set-up, assume that $Y$ has relatively stable reduction; let $\XX$ be any model of $X$; and let $\YY$ be the corresponding model of $Y$. Then, 
    $\XX\le \Xrst$ if and only if, for all components $V$ of $\YY_s$, we have $m(V)=1$ and one of the following holds:
        \begin{enumerate}[(i)]
            \item $a(V)\ge 1$;
            \item $a(V)=0$ and $w(V)\ge 3$; or
            \item $V$ is a vertical (-2)-curve.
        \end{enumerate}
        \begin{proof}
            Suppose first that we have $\XX\le \Xrst$.  Given a component $V$ of $\YY_s$, its strict transform $V^{\mathrm{rst}}$ in $\Yrst$ will have the same the same invariants $m$, $a$, and $w$ as $V$ by \Cref{lemma inv m a w} and will be a vertical (-2)-curve if and only if $V$ is by \Cref{Minus12LinesStrictTransformRelative}. Now, since $\SF{\Yrst}$ is reduced, we have $m(V)=1$. Since $\SF{\Yrst}$ does not contain (-1)-lines or horizontal (-2)-curves, we deduce from \Cref{lemma12lines}(c) and \Cref{dfn minus two curve} (applied to the model $\Yrst$) that one of the three conditions (i), (ii) and (iii) above must occur.
            
            Now assume that $m(V) = 1$ and that either (i), (ii), or (iii) holds.  By \Cref{prop part of min}, we deduce that $\YY\le \Ymini$. Let $V$ be any component of $\CC_s$, and let $V^{\mathrm{min}}$ be the strict transform of $V$ in $\SF{\Ymini}$, which has the same invariants $m$, $a$, and $w$ as $V$ by \Cref{lemma inv m a w} and is a vertical (resp. horizontal) (-2)-curve if and only if $V$ is a vertical (resp. horizontal) curve by \Cref{Minus12LinesStrictTransformRelative}.  In particular, the transform $V^{\mathrm{min}}$ is not a horizontal (-2)-curve, so all horizontal (-2)-curves of $\SF{\Ymini}$ get contracted in $\YY_s$, which, in light of \Cref{prop rst exists}, is equivalent to saying that $\YY\le \Yrst$.
        \end{proof}
\end{prop}

We end this subsection by pointing out some simple properties of the relatively stable model.
\begin{prop}
    \label{prop smoothness components rst}
    Suppose that $Y$ has relatively stable reduction. Then all of the irreducible components of $\SF{\Yrst}$ are smooth $k$-curves.
    \begin{proof}
        Choose a component $L \in \Irred(\SF{\Xrst})$, which is necessarily a line as $X$ is the projective line (see the well-known results of \S\ref{sec models hyperelliptic line} below).  Let $Q$ be a node of $\SF{\Yrst}$ lying over a point $P\in L$. Since $\Yrst$ does not contain vanishing nodes by definition, we have that $P$ is a node of $\Xrst$; moreover, since the line $L$ is smooth at $P$, we have that $P$ must connect $L$ with another line $L'\in \Irred(\SF{\Xrst})$. We deduce from this that $Q$ connects two distinct components $V$ and $V'$ of $\SF{\Yrst}$, one lying over $L$, and the other lying over $L'$. We conclude that, if $V_1$ and $V_2$ are two (possibly coinciding) irreducible components of $\SF{\Yrst}$ lying over $L$, they cannot be connected by a node; hence, the inverse image of $L$ in $\SF{\Yrst}$ is a smooth $k$-curve.
    \end{proof}
\end{prop}

\begin{rmk} \label{rmk rst isomorphism failure}
    It is possible (but not necessary for the results in this paper) to demonstrate the following useful property of the relatively stable model.  Assume that $Y$ has relatively stable reduction, let $\XX$ be any model of $X$, and let $\YY$ be the corresponding model of $Y$. Assume that $\YY$ has reduced special fiber. Then we have $\Ctr(\XX,\Xrst)=f(\Sing(\YY_s))$, where $f: \YY\to \XX$ is the $G$-covering map and $\Sing(\YY_s)$ is the set of all singular points of $\YY_s$.  In other words, when $\XX$ is a smooth model of $X$, the (other) components of $\Xrst$ are contracted precisely to the points of $\XX_s$ above which $\YY_s$ is singular.
\end{rmk}

\section{Models of hyperelliptic curves} \label{sec models hyperelliptic}

Our main aim is now to more explicitly understand the relatively stable model $\Yrst$ of the hyperelliptic curve $Y$. After recalling some basic general facts about hyperelliptic curves in \S\ref{sec models hyperelliptic equations}, we describe the semistable models of the line $X$ in \S\ref{sec models hyperelliptic line}. After introducing the notion of a \emph{part-square decomposition} of a polynomial in \S\ref{sec models hyperelliptic part-square}, we exploit it in \S\ref{sec models hyperelliptic forming} to describe the model of the hyperelliptic curve $Y$ corresponding to a given smooth model of the line $X$ (i.e.\ to a given disc $D$). The special fiber of such models of $Y$ will be more thoroughly studied in \S\ref{sec models hyperelliptic separable}, in which we provide a criterion (given in \Cref{thm part of rst separable}) to decide whether a disc $D$ belongs to the collection $\mathfrak{D}$ that defines the semistable model of the line $\Xrst$ from which $\Yrst$ comes.

\subsection{Equations for hyperelliptic curves} \label{sec models hyperelliptic equations}

We let $F$ be any field and write $X := \proj_F^1$ for the projective line over $F$.  In this subsection, we review basic facts and definitions relating to hyperelliptic curves over $F$ which can be found in \cite[\S7.4.3]{liu2002algebraic}.

\begin{dfn}
	A \emph{hyperelliptic curve} over $F$ is a smooth curve $Y / F$ of genus $g \geq 1$ along with a separable (branched) covering morphism $Y \to X$ of degree 2, which we call the \emph{hyperelliptic map}.
\end{dfn}

It is possible through repeated applications of the Riemann-Roch Theorem to show the well-known fact that the affine chart $x\neq \infty$ of any hyperelliptic curve $Y / F$ can be described by an equation of the form 
\begin{equation} \label{eq hyperelliptic general affine1}
y^2 + q(x)y = r(x),
\end{equation}
where $\deg(q) \leq g + 1$ and $\deg(r) \leq 2g + 2$ and the hyperelliptic map is given by the coordinate $x : Y \to X$.  If $F$ has characteristic different from $2$, a suitable change of the coordinate $y$ allows us to convert this equation into the simpler form $y^2 = r(x) + \frac{1}{4}q^2(x)$, from which it is clear that the smoothness condition implies that the polynomial $f(x) := r(x) + \frac{1}{4}q^2(x)$ must be separable.
Over the complementary affine chart of $\proj_F^1$ where $x \neq 0$, the hyperelliptic curve $Y$ can be described by the equation
\begin{equation} \label{eq hyperelliptic general affine2}
    \check{y}^2 + \check{q}(\check{x})\check{y} = \check{r}(\check{x}),
\end{equation}
where $\check{x} = x^{-1}$, $\check{y} = x^{-(g+1)}y$, $\check{q}(\check{x}) = x^{-(g+1)}q(x)$, and $\check{r}(\check{x}) = x^{-(2g+2)}r(x)$.  Note that the polynomial $\check{q}(z) \in F[z]$ (resp.\ $\check{r}(z) \in F[z]$) differs from the polynomial $q(z) \in F[z]$ (resp.\ $r(z) \in F[z]$) only in that each power $z^i$ which appears in the polynomial is replaced by $z^{g+1-i}$ (resp.\ $z^{2g+2-i}$) while the coefficients remain the same.

If $F$ has characteristic different from $2$ (i.e.\ if we consider \emph{tame} hyperelliptic curves), the Riemann-Hurwitz formula ensures that the ramification locus of $Y_{\bar{F}}\to X_{\bar{F}}$ consists of $2g+2$ points of $Y_{\bar{F}}$, lying over $2g+2$ distinct branch points of $X_{\overline{F}}$.  In fact, the branch locus determines a hyperelliptic curve almost completely, as we see from the following well-known result.

\begin{prop} \label{prop hyperelliptic}
	Given a field $F$ of characteristic different from $2$, and letting $X$ be the projective line $\mathbb{P}^1_F$, the following data are equivalent:
	\begin{enumerate}[(i),nolistsep]
		\item a hyperelliptic curve $Y$ of genus $g$ having rational branch locus, endowed with a distinguished hyperelliptic map $Y \to X$; 
		\item a separable polynomial $f(x) \in F[x]$ of degree $2g + 1$ or $2g + 2$ all of whose roots lie in $F$, modulo multiplication by a scalar in $(F^{\times})^2$; and 
		\item a cardinality-$(2g+2)$ subset $\mathcal{B} \subset X(F)$ together with an element $c \in F^\times/(F^\times)^2$.
	\end{enumerate}
	Moreover, in (ii) above, the polynomial $f$ will have degree $2g + 1$ (resp.\ $2g + 2$) if in the context of (iii) above the coordinate of $X$ is chosen such that $\infty$ is (resp.\ is not) an element of $\mathcal{B}$.
	
\end{prop}

\begin{rmk} \label{rmk assumptions}
    Suppose that in the context of the above proposition, none of branch points of $Y \to X$ is $\infty$.  We may then find an isomorphic hyperelliptic curve over $F$ whose branch points over $X$ include the point $\infty \in X(F)$ by applying an automorphism of the projective line $X$ which moves one of the branch points to $\infty$ (and thus whose defining polynomial $f$, as in statement (ii) of the proposition, has degree $2g + 1$).
\end{rmk}

From now on, we assume that $F$ is the discretely-valued field $K$ satisfying the conditions given in \S\ref{sec introduction notation}.  In light of the remark above, up to possibly replacing $K$ with a finite extension (so that at least one of the branch points of the cover $Y\to X$ is rational), we can and will make the following assumption throughout the rest of the paper.

\begin{hyp} \label{hyp properties of f}
The hyperelliptic curve $Y$ is defined over $K$ by the equation $y^2 = f(x)$, where $x$ is the standard coordinate of $X=\mathbb{P}^1_K$, and $f(x) \in K[x]$ is a polynomial of (odd) degree $2g+1$, where $g$ is the genus of $Y$.
\end{hyp}

Proposition \ref{prop hyperelliptic} allows us to treat the hyperelliptic curve $Y$ essentially as a marked line.  This is peculiar to the hyperelliptic case: if we were to deal with a tame Galois covering of the line of degree greater than $2$, the same branch locus would in general be shared by multiple branched coverings corresponding to various possible monodromy actions.

Our aim will be constructing semistable models of a given hyperelliptic curve $Y \to X$ by normalizing some carefully chosen semistable models of the line $X$ in the quadratic extension $K(X)\subseteq K(Y)$. We will start by analyzing smooth and semistable models of the line $X$ in the next subsection; in the subsequent ones, we will turn our attention to the corresponding models of $Y$.

\subsection{Models of the projective line} \label{sec models hyperelliptic line}

Throughout the rest of this paper, we write $D_{\alpha, b}$ for the disc $\{x \in K^\alg \ | \ v(x - \alpha) \geq b\}$ of center $\alpha \in K^\alg$ and depth $b \in \qq$; given a finite subset $\mathfrak{s} \subset K^\alg$ and some $b \in \qq$ such that $v(\zeta - \zeta') \geq b$ for all $\zeta, \zeta' \in \mathfrak{s}$, we write $D_{\mathfrak{s}, b} = D_{\alpha, b}$ for some (any) $\alpha \in \mathfrak{s}$.

As before, let $X:=\mathbb{P}^1_K$ be the projective line, and let $x$ denote its standard coordinate. Given $\alpha\in K^\alg$ and  $\beta\in (K^\alg)^\times$, one can define a smooth model $\XX_{\alpha,\beta}$ of $X$ over the ring of integers $R'$ of $K':=K(\alpha,\beta)\subseteq K^\alg$) by declaring $\XX_{\alpha,\beta}:=\mathbb{P}^1_{R'}$, with coordinate $x_{\alpha,\beta} := \beta^{-1}(x - \alpha)$, as an $R'$-scheme, and identifying the generic fiber $\XX_\eta$ with $X$ via the linear transformation $x_{\alpha,\beta}=\beta^{-1}(x-\alpha)$. If $(\alpha_1,\beta_1)$ and $(\alpha_2,\beta_2)$ are such that $v(\alpha_1-\alpha_2)\ge v(\beta_1)=v(\beta_2)$, then $\XX_{\alpha_1,\beta_1}$ and $\XX_{\alpha_2,\beta_2}$ are isomorphic as models of $X$, the isomorphism being given by the change of variable $x_{\alpha_2,\beta_2}=u x_{\alpha_1,\beta_1}+ \delta$, where $u$ is the unit $\beta_1(\beta_2)^{-1}$ and $\delta$ is the integral element $\beta_2^{-1}(\alpha_1-\alpha_2)$. In other words, the model $\XX_{\alpha,\beta}$ only depends, up to isomorphism, on the disc $D := D_{\alpha,v(\beta)}$; for this reason, we will often denote $\XX_{\alpha,\beta}$ by $\XX_D$.

If $\mathfrak{D}=\{D_1,\ldots, D_n\}$ is a non-empty, finite collection of discs of $K^\alg$, one can form a corresponding model $\XX_{\mathfrak{D}}$, which is defined as the minimal model dominating all the smooth models $\lbrace \XX_D: D\in \mathfrak{D}\rbrace$. Its special fiber is a reduced $k$-curve of arithmetic genus equal to $0$ i.e.\ it consists of $n$ lines $L_1, \ldots, L_n$ corresponding to the discs $D_i$'s meeting each other at ordinary multiple points, without forming loops (i.e.\ its toric rank is $0$).

The following is a standard fact about models of the projective line.

\begin{prop}
    \label{prop smooth models line discs}
    The construction $D\mapsto \XX_D$ described above defines a bijection between the discs of $K^\alg$ and the smooth models of $X$ defined over finite extensions of $R$ considered up to isomorphism (two models $\XX_1/R'_1$ and $\XX_2/R'_2$ are considered isomorphic if they become so over some common finite extension $R''\supseteq R'_1, R'_2$).

    More generally speaking, the construction $\mathfrak{D}\mapsto \XX_\mathfrak{D}$ described above defines a bijection between the finite non-empty collection of discs $\mathfrak{D}$ of $K^\alg$ and the models of $X$ having reduced special fiber defined over finite extensions of $R$, considered up to isomorphism..

\end{prop}

Given two discs $D=D_{\alpha,b}$ and $D'=D_{\alpha',b'}$, we want to compare the smooth models $\XX_{D}$ and $\XX_{D'}$: to this end, one may verify the following proposition by an immediate computation.

\begin{prop}
    \label{prop relative position smooth models line}
    With the notation above, assume $D\neq D'$, and let $P\in \SF{\XX_D}(k)$ and $P'\in \SF{\XX_{D'}}(k)$ be the points such that $\Ctr(\XX_D,\XX_{D'})=\{ P \}$ and $\Ctr(\XX_{D'},\XX_{D})=\{ P' \}$. Then there are the following three possibilities (illustrated in Figure \ref{fig two models line} below):
    \begin{enumerate}[(a)]
        \item when $D\subsetneq D'$, $P$ is the point $\overline{x_{\alpha,\beta}}=\infty$ and $P'$ is the point $\overline{x_{\alpha',\beta'}}=\overline{(\beta')^{-1}(\alpha-\alpha')}\neq \infty$;
        \item when $D'\subsetneq D$, $P$ is the point $\overline{x_{\alpha,\beta}}=\overline{\beta^{-1}(\alpha'-\alpha)}\neq \infty$, and $P'$ is the point $\overline{x_{\alpha',\beta'}}=\infty$; or 
        \item when $D\cap D'=\varnothing$, $P$ is the point $\overline{x_{\alpha,\beta}}=\infty$ and $P'$ is the point $\overline{x_{\alpha',\beta'}}=\infty$.
    \end{enumerate}
\end{prop}

\afterpage{
\begin{figure}[t]
    \centering

\begin{tikzpicture}[x=0.75pt,y=0.75pt,yscale=-1,xscale=1]
%uncomment if require: \path (0,291); %set diagram left start at 0, and has height of 291
%Straight Lines [id:da9019407953732452] 
\draw [color={rgb, 255:red, 0; green, 0; blue, 0 }  ,draw opacity=1 ]   (100,126) -- (37,201) ;
%Straight Lines [id:da22108323773119343] 
\draw    (19,120.5) -- (101,193) ;
%Straight Lines [id:da721302398750755] 
\draw    (68.5,163.5) .. controls (66.79,161.87) and (66.75,160.21) .. (68.38,158.5) .. controls (70.01,156.79) and (69.97,155.13) .. (68.26,153.5) .. controls (66.55,151.87) and (66.51,150.21) .. (68.14,148.5) .. controls (69.76,146.79) and (69.72,145.13) .. (68.01,143.51) .. controls (66.3,141.88) and (66.26,140.22) .. (67.89,138.51) .. controls (69.52,136.8) and (69.48,135.14) .. (67.77,133.51) .. controls (66.06,131.88) and (66.02,130.22) .. (67.65,128.51) .. controls (69.28,126.8) and (69.24,125.14) .. (67.53,123.51) .. controls (65.82,121.88) and (65.78,120.22) .. (67.41,118.51) .. controls (69.04,116.8) and (69,115.14) .. (67.29,113.51) -- (67.24,111.66) -- (67.05,103.67) ;
\draw [shift={(67,101.67)}, rotate = 88.61] [color={rgb, 255:red, 0; green, 0; blue, 0 }  ][line width=0.75]    (10.93,-3.29) .. controls (6.95,-1.4) and (3.31,-0.3) .. (0,0) .. controls (3.31,0.3) and (6.95,1.4) .. (10.93,3.29)   ;
%Straight Lines [id:da10619539778414422] 
\draw [color={rgb, 255:red, 0; green, 0; blue, 0 }  ,draw opacity=1 ]   (311,132) -- (248,207) ;
%Straight Lines [id:da19040057833004675] 
\draw    (230,126.5) -- (312,199) ;
%Straight Lines [id:da24639727602399242] 
\draw    (279.5,169.5) .. controls (277.75,167.92) and (277.67,166.26) .. (279.25,164.51) .. controls (280.83,162.76) and (280.75,161.1) .. (279,159.51) .. controls (277.25,157.93) and (277.17,156.27) .. (278.75,154.52) .. controls (280.33,152.77) and (280.25,151.11) .. (278.5,149.53) .. controls (276.75,147.94) and (276.67,146.28) .. (278.25,144.53) .. controls (279.83,142.78) and (279.75,141.12) .. (278,139.54) .. controls (276.25,137.95) and (276.17,136.29) .. (277.75,134.54) .. controls (279.33,132.79) and (279.25,131.13) .. (277.5,129.55) .. controls (275.75,127.97) and (275.67,126.31) .. (277.25,124.56) .. controls (278.83,122.81) and (278.75,121.15) .. (277,119.56) .. controls (275.25,117.98) and (275.17,116.32) .. (276.75,114.57) -- (276.5,109.65) -- (276.1,101.66) ;
\draw [shift={(276,99.67)}, rotate = 87.13] [color={rgb, 255:red, 0; green, 0; blue, 0 }  ][line width=0.75]    (10.93,-3.29) .. controls (6.95,-1.4) and (3.31,-0.3) .. (0,0) .. controls (3.31,0.3) and (6.95,1.4) .. (10.93,3.29)   ;
%Straight Lines [id:da09396203159723215] 
\draw [color={rgb, 255:red, 0; green, 0; blue, 0 }  ,draw opacity=1 ]   (531,133) -- (468,208) ;
%Straight Lines [id:da28863141682284854] 
\draw    (450,127.5) -- (532,200) ;
%Straight Lines [id:da16273735566824843] 
\draw    (499.5,170.5) .. controls (497.8,168.87) and (497.76,167.2) .. (499.39,165.5) .. controls (501.02,163.8) and (500.99,162.13) .. (499.29,160.5) .. controls (497.59,158.87) and (497.55,157.2) .. (499.18,155.5) .. controls (500.81,153.8) and (500.78,152.13) .. (499.08,150.5) .. controls (497.38,148.87) and (497.34,147.21) .. (498.97,145.51) .. controls (500.6,143.81) and (500.56,142.14) .. (498.86,140.51) .. controls (497.16,138.88) and (497.13,137.21) .. (498.76,135.51) .. controls (500.39,133.81) and (500.35,132.14) .. (498.65,130.51) .. controls (496.95,128.88) and (496.92,127.21) .. (498.55,125.51) .. controls (500.18,123.81) and (500.14,122.14) .. (498.44,120.51) .. controls (496.74,118.88) and (496.71,117.21) .. (498.34,115.51) .. controls (499.97,113.81) and (499.93,112.14) .. (498.23,110.51) -- (498.21,109.66) -- (498.04,101.67) ;
\draw [shift={(498,99.67)}, rotate = 88.79] [color={rgb, 255:red, 0; green, 0; blue, 0 }  ][line width=0.75]    (10.93,-3.29) .. controls (6.95,-1.4) and (3.31,-0.3) .. (0,0) .. controls (3.31,0.3) and (6.95,1.4) .. (10.93,3.29)   ;
%Straight Lines [id:da8163270566704827] 
\draw    (31.5,132.5) .. controls (32.31,134.71) and (31.6,136.22) .. (29.39,137.03) .. controls (27.18,137.84) and (26.47,139.35) .. (27.27,141.56) -- (25.23,145.94) -- (21.85,153.19) ;
\draw [shift={(21,155)}, rotate = 295.02] [color={rgb, 255:red, 0; green, 0; blue, 0 }  ][line width=0.75]    (10.93,-3.29) .. controls (6.95,-1.4) and (3.31,-0.3) .. (0,0) .. controls (3.31,0.3) and (6.95,1.4) .. (10.93,3.29)   ;
%Straight Lines [id:da9929071003160636] 
\draw    (296.5,148.5) .. controls (298.81,148.99) and (299.71,150.39) .. (299.21,152.7) .. controls (298.71,155.01) and (299.61,156.41) .. (301.92,156.91) .. controls (304.23,157.4) and (305.13,158.8) .. (304.63,161.11) -- (305.58,162.59) -- (309.92,169.32) ;
\draw [shift={(311,171)}, rotate = 237.2] [color={rgb, 255:red, 0; green, 0; blue, 0 }  ][line width=0.75]    (10.93,-3.29) .. controls (6.95,-1.4) and (3.31,-0.3) .. (0,0) .. controls (3.31,0.3) and (6.95,1.4) .. (10.93,3.29)   ;
%Straight Lines [id:da3775360813946179] 
\draw    (499.5,170.5) .. controls (497.99,172.31) and (496.33,172.46) .. (494.52,170.95) .. controls (492.71,169.44) and (491.05,169.6) .. (489.54,171.41) .. controls (488.03,173.22) and (486.37,173.37) .. (484.56,171.86) -- (481.96,172.09) -- (473.99,172.82) ;
\draw [shift={(472,173)}, rotate = 354.81] [color={rgb, 255:red, 0; green, 0; blue, 0 }  ][line width=0.75]    (10.93,-3.29) .. controls (6.95,-1.4) and (3.31,-0.3) .. (0,0) .. controls (3.31,0.3) and (6.95,1.4) .. (10.93,3.29)   ;

% Text Node
\draw (122,202) node    {$L'$};
% Text Node
\draw (23,204) node  [color={rgb, 255:red, 0; green, 0; blue, 0 }  ,opacity=1 ]  {$L$};
% Text Node
\draw (12,159.4) node [anchor=north west][inner sep=0.75pt]    {$\infty $};
% Text Node
\draw (333,208) node    {$L'$};
% Text Node
\draw (234,209) node  [color={rgb, 255:red, 0; green, 0; blue, 0 }  ,opacity=1 ]  {$L$};
% Text Node
\draw (553,209) node    {$L'$};
% Text Node
\draw (454,211) node  [color={rgb, 255:red, 0; green, 0; blue, 0 }  ,opacity=1 ]  {$L$};
% Text Node
\draw (450,168.4) node [anchor=north west][inner sep=0.75pt]    {$\infty $};
% Text Node
\draw (31,234) node [anchor=north west][inner sep=0.75pt]   [align=left] {Case (a): $\displaystyle D\varsubsetneq D'$};
% Text Node
\draw (233,234) node [anchor=north west][inner sep=0.75pt]   [align=left] {Case (b): $\displaystyle D'\varsubsetneq D$};
% Text Node
\draw (450,234) node [anchor=north west][inner sep=0.75pt]   [align=left] {Case (c): $\displaystyle D\cap D'=\emptyset $};
% Text Node
\draw (313,170.4) node [anchor=north west][inner sep=0.75pt]    {$\infty $};
% Text Node
\draw (17,54.4) node [anchor=north west][inner sep=0.75pt]  [font=\normalsize]  {$\overline{x_{\alpha ,\beta }} =\infty $};
% Text Node
\draw (18,71.4) node [anchor=north west][inner sep=0.75pt]  [font=\normalsize]  {$\overline{x_{\alpha ',\beta '}} =\overline{(\beta ')^{-1} (\alpha -\alpha ')} \neq \infty $};
% Text Node
\draw (227,53.4) node [anchor=north west][inner sep=0.75pt]  [font=\normalsize]  {$\overline{x_{\alpha ,\beta }} =\overline{\beta ^{-1} (\alpha '-\alpha )} \neq \infty $};
% Text Node
\draw (228,72.4) node [anchor=north west][inner sep=0.75pt]  [font=\normalsize]  {$\overline{x_{\alpha ',\beta '}} =\infty $};
% Text Node
\draw (453,53.4) node [anchor=north west][inner sep=0.75pt]  [font=\normalsize]  {$\overline{x_{\alpha ,\beta }} =\infty $};
% Text Node
\draw (454,72.4) node [anchor=north west][inner sep=0.75pt]  [font=\normalsize]  {$\overline{x_{\alpha ',\beta '}} =\infty $};

\end{tikzpicture}
    \caption{The special fiber of the minimal model $\XX_{\{D,D'\}}$ dominating $\XX_D$ and $\XX_D'$. Here, $L$ and $L'$ are the lines corresponding to the discs $D$ and $D'$ respectively.} \label{fig two models line}
\end{figure}
}

\begin{prop} \label{prop thickness}
    Suppose that $\XX / R'$ is a semistable model of the line $X$ for some finite extension $R' / R$, such that there are discs $D_{\alpha,b} \subsetneq D_{\alpha',b'} \subset K^\alg$ corresponding to two intersecting components of $\SF{\XX}$.  Then the thickness of the node where they intersect is given by the formula $(b' - b) / v(\pi)$, where $\pi \in K^\alg$ is a uniformizer of $R'$.
\end{prop}

\begin{proof}
    We can clearly replace the center $\alpha'$ with $\alpha \in D_{\alpha,b} \subsetneq D_{\alpha',b'}$; now choosing $\beta, \beta' \in (K^\alg)^{\times}$ be scalars such that $v(\beta) = b$ and $v(\beta') = b'$.  Then, with the notation above, we have coordinates $x_{\alpha,\beta}$ and $x_{\alpha,\beta'}$ corresponding to each of these components of $\SF{\XX}$, and these coordinates are related by the equation $x_{\alpha,\beta} = \beta' \beta^{-1} x_{\alpha,\beta'}$.  Locally around the point of intersection, a defining equation is $x_{\alpha,\beta} x_{\alpha,\beta'}^\vee = \beta' \beta^{-1}$, where $x_{\alpha,\beta'}^{\vee} = x_{\alpha,\beta'}^{-1}$, and so the thickness by definition is equal to $v(\beta' \beta^{-1}) / v(\pi) = (b' - b) / v(\pi)$.
\end{proof}

\subsection{Part-square decompositions} \label{sec models hyperelliptic part-square}

We begin this subsection by defining a \textit{part-square decomposition}, and then we study part-square decompositions with certain properties.

\begin{dfn} \label{dfn qrho}
    Given a nonzero polynomial $h(x) \in K^\alg[z]$, a \emph{part-square decomposition} of $h$ is a way of writing $h = q^2 + \rho$ for some $q(x), \rho(x) \in K^\alg[x]$, with $\deg(q)\le \lceil\frac{1}{2}\deg(h)\rceil$.
\end{dfn}

\begin{rmk}
    \label{rmk degree of rho}
    The definition forces $\deg(\rho)\le \deg(h)$ when $h$ has even degree and $\deg(\rho)\le \deg(h)+1$ when $h$ has odd degree. The definition allows $q$ to be equal to zero.
\end{rmk}

Given a part-square decomposition $h = q^2 + \rho$, we define the rational number $t_{q, \rho}:=v(\rho)-v(h) \in \qqinfty$.

\begin{dfn} \label{dfn good totally odd}

We define the following properties of a part-square decomposition $h = q^2 + \rho$.

\begin{enumerate}[(a)]
    \item The decomposition is said to be \emph{good} either if we have $t_{q,\rho} \geq 2v(2)$ or if we have $t_{q,\rho} < 2v(2)$ and there is no decomposition $h = \tilde{q}^2 + \tilde{\rho}$ such that $t_{\tilde{q}, \tilde{\rho}} > t_{q, \rho}$.
    \item The decomposition is said to be \emph{totally odd} if $\rho$ only consists of odd-degree terms.
\end{enumerate}

\end{dfn}

\begin{rmk} \label{rmk good decompositions}
    The trivial part-square decomposition $h=0^2+h$ has $t_{0,h} = 0$; this immediately implies that all good decompositions $h = q^2 + \rho$ satisfy $t_{q,\rho} \ge 0$.
\end{rmk}

\begin{rmk} \label{rmk same t for good}
    If $h=q^2+\rho=(q')^2+\rho'$ are two good part-square decompositions for the same nonzero polynomial $h$, then \Cref{dfn good totally odd} directly implies the equality $t_{q,\rho} = t_{q',\rho'}$ if either is $< 2v(2)$, so we have $\truncate{t_{q,\rho}}=\truncate{t_{q',\rho'}}$.
\end{rmk}

\begin{prop} \label{prop good decomposition}
    Let $h = q^2 + \rho$ be a part-square decomposition satisfying $t_{q,\rho} < 2v(2)$.  Then we have the following.
    \begin{enumerate}[(a)]
        \item The decomposition $h = q^2 + \rho$ is good if and only if some (any) normalized reduction of $\rho$ is not the square of a polynomial with coefficients in $k$.
        \item Suppose that the decomposition $h = q^2 + \rho$ is good and that $h = \tilde{q}^2 + \tilde{\rho}$ is another good decomposition.  Then given any normalized reductions of $\rho$ and $\tilde{\rho}$ respectively, the odd degrees appearing among terms in these normalized reductions, as well as the derivatives of the normalized reductions, are equal up to scaling.
    \end{enumerate}
    
    \begin{proof}
        We begin by proving part (a).  If $t_{q,\rho}<0$, the decomposition is not good (see \Cref{rmk good decompositions}), and any normalized reduction of $\rho$ is a square, since it is a scalar multiple of a normalized reduction of $q^2$. We now have to prove the two implications when $t_{q,\rho}\ge 0$.
    
        Suppose that $h = q^2 + \rho$ satisfies $0\le t_{q,\rho} < 2v(2)$ but is not good, so that another decomposition $h = \tilde{q}^2(x) + \tilde{\rho}(x)$ with $t_{\tilde{q}, \tilde{\rho}} > t_{q, \rho}$ can be found.
        Let us now consider $q + \tilde{q}$ and $q - \tilde{q}$: their product has valuation $v(q^2-\tilde{q}^2) = v(\tilde{\rho}-\rho) = v(\rho)$, while their difference has valuation
        \begin{equation} \label{eq v(2 tilde q)}
            v(2\tilde{q})=v(2)+\frac{1}{2}v(h-\tilde{\rho}) \geq v(2)+\frac{1}{2}v(h) > \frac{1}{2}v(\rho).
        \end{equation}
        From this, it is immediate to deduce that they must both have valuation equal to $\frac{1}{2}v(\rho)$.  We may now write 
        \begin{equation} \label{eq good decomposition}
        	\rho = \tilde{\rho} + 2\tilde{q}(\tilde{q}-q) -(\tilde{q}-q)^2.
        \end{equation}
        But we observe that the first two summands both have valuation $>v(\rho)$. This implies that the normalized reduction of $\rho$ is a square.
        		
        Conversely, suppose that the decomposition satisfies $0\le t_{q,\rho} < 2v(2)$ and that the normalized reduction of $\rho$ is a square; this is clearly equivalent to saying that we can form a part-square decomposition $\rho = q_1^2 + \rho_1$ of the polynomial $\rho$ that satisfies $t_{q_1,\rho_1}>0$; hence, we have $v(q_1) = \frac{1}{2}v(\rho)$ and $v(\rho_1) > v(\rho)$.
        
        Let us now consider the part-square decomposition $h = \tilde{q}^2 + \tilde{\rho}$, where $\tilde{q}:=q + q_1$ and $\tilde{\rho}=\rho_1 - 2q q_1$. Notice that the assumption $t_{q,\rho}\ge 0$ implies that $v(q)\ge \frac{1}{2}v(h)$; we therefore have $v(2q q_1) \geq v(2) + \frac{1}{2} v(h) + \frac{1}{2}v(\rho) > v(\rho)$. We conclude that $v(\tilde{\rho})=v(\rho_1-2qq_1)>v(\rho)$, i.e.\  $t_{\tilde{q}, \tilde{\rho}}> t_{q, \rho}$; therefore, the original part-square decomposition $h = q^2 + \rho$ was not good.  Thus, both directions of part (a) are proved.
        
        We now turn to part (b) and assume that $h = q^2 + \rho = \tilde{q}^2 + \tilde{\rho}$ are both good decompositions.  Then both (\ref{eq v(2 tilde q)}) and (\ref{eq good decomposition}) are still valid, and the fact that $v(\rho) = v(\tilde{\rho})$ by \Cref{rmk same t for good} implies that $v(2\tilde{q}(\tilde{q}-q) -(q-\tilde{q})^2) \geq v(\rho)$.  Then if $v(q - \tilde{q}) < \frac{1}{2}v(\rho)$, from (\ref{eq good decomposition}) we must have $v(2\tilde{q}(\tilde{q} - q)) = v((q - \tilde{q})^2) < v(\rho)$, which contradicts (\ref{eq v(2 tilde q)}).  We therefore have $v(\tilde{q} - q) \geq \frac{1}{2}v(\rho)$, and (\ref{eq v(2 tilde q)}) now implies $v(2\tilde{q}(\tilde{q} - q)) > v(\rho)$.
        
        Let $\gamma \in K^\alg$ be a scalar with $v(\gamma) = v(\rho) = v(\tilde{\rho})$.  If $v(\tilde{q} - q) > \frac{1}{2}v(\rho)$, then (\ref{eq good decomposition}) shows that $v(\tilde{\rho} - \rho) > v(\rho)$ and so $\gamma^{-1}(\tilde{\rho} - \rho)$ has positive valuation; therefore, the reductions of $\gamma^{-1}\rho$ and $\gamma^{-1}\tilde{\rho}$ are equal, and we are done.  If $v(\tilde{q} - q) = \frac{1}{2}v(\rho)$, then $\gamma^{-1}(\tilde{\rho} - \rho)$ reduces to a square (namely a normalized reduction of $\tilde{q} - q$ squared); the square of a polynomial in $k[x]$ has only even-degree terms, and its derivative vanishes, which shows that the reductions of $\gamma^{-1}\rho$ and $\gamma^{-1}\tilde{\rho}$ have the same odd degrees appearing and have the same derivative.  Thus again we are done, and part (b) is proved.
    \end{proof}
\end{prop}

\begin{cor} \label{cor totally odd is good}
Every totally odd part-square decomposition of a polynomial is good.
    \begin{proof}
    Suppose that the decomposition $h = q^2 + \rho$ is totally odd.  If $t_{q, \rho} \geq 2v(2)$, then we are already done, so assume that $t_{q, \rho} < 2v(2)$.  Then since $\rho$ consists only of odd-degree terms, the same is true of any normalized reduction of $\rho$, which consequently cannot be the square of any polynomial in $k[z]$. Thus \Cref{prop good decomposition} implies that the decomposition is good.
    \end{proof}
\end{cor}

We now want to show that a good part-square decomposition of a polynomial always exists, for which, thanks to Corollary \ref{cor totally odd is good}, it suffices to show that a polynomial always has a totally odd part-square decomposition.

\begin{prop} \label{prop totally odd existence}

Given a nonzero polynomial $h(z) \in K^\alg[z]$, there always exists a totally odd part-square decomposition $h = q^2 + \rho$ with $q(z), \rho(z) \in K^\alg[z]$.

\end{prop}

\begin{proof}
    Let $h_e \in K^\alg[z]$ be the sum of the even-degree terms of $h$, whose roots are easily seen to come in pairs $\pm\sqrt{\alpha_i}$ for some elements $\alpha_i \in K^\alg$.  Writing $c_0 z^m + c_1 z^{m - 1} + \dots + c_m = \sqrt{c}\prod_i (z + \sqrt{\alpha_i})$, where $\sqrt{c}$ is a square root of the leading coefficient $c$ of $h_e$, one verifies straightforwardly that setting $q$ to be the polynomial whose $i$th coefficient is $c_i$ for even $i$ and $\sqrt{-1} c_i$ for odd $i$ (given some fixed square root of $-1$) produces a totally odd decomposition of $h$.
\end{proof}

We note that, if a nonzero polynomial $h\in K^\alg[z]$ is written as a product of factors $h=\prod_{i=1}^N h_i$ with $h_i\in K^\alg[z]$, then, given part-square decompositions $h_i=q_i^2+\rho_i$, with $q_i, \rho_i\in K^\alg[z]$, one can use them to form a part-square decomposition $h = q^2+\rho$, where $q = \prod_{i=1}^N q_i$ and $\rho = h-q^2$.  The following proposition dealing with this situation will be useful for our computations in \S\ref{sec depths}.

\begin{prop}
    \label{prop product part-square}
    In the setting above, for $1 \leq i \leq N$, let $t_i =t_{q_i,\rho_i}$ and $t =t_{q,\rho}$, and assume that we have $t_i \geq 0$.  Then we have $t \geq \min\{t_1,\dots, t_N\}$.  Moreover, we have the following.
    \begin{enumerate}[(a)]
        \item If there is a unique index $i_0$ such that the minimum $\min\{t_1,\ldots t_N\}$ is achieved by $t_{i_0}$, then we have $t= \min\{t_1,\ldots t_N\}$.  Moreover, in this case, the part-square decomposition of $h$ is good if and only if that of $h_{i_0}$ is.
        \item Suppose that $N=2$ and that, for all roots $s_1$ in $K^\alg$ of $h_1$ and for all roots $s_2$ of $h_2$, we have $v(s_1)>0$ but $v(s_2)<0$; assume, moreover, that both decompositions $h_i = q_i^2 + \rho_i$ are good. Then, if $\min\{t_1,t_2\}<2v(2)$, we have $t=\min\{t_1,t_2\}$, and the corresponding decomposition of $h$ is also good.
    \end{enumerate}

    \begin{proof}

    It clearly suffices to prove part (a) for $N = 2$, so let us assume that $N = 2$.  We have 
    \begin{equation} \label{eq product part-square}
            \rho = h-q^2 = h_1h_2 - (q_1 q_2)^2 = (q_1^2+\rho_1)(q_2^2+\rho_2)-(q_1 q_2)^2 = \rho_1 q_2^2 + \rho_2 q_1^2 + \rho_1 \rho_2.
    \end{equation}
    \Cref{rmk good decompositions} tells us that $t_1, t_2 \geq 0$, and therefore we have have $v(\rho_i)\ge v(h_i)$ (with strict inequality when $t_i > 0$) and thus $v(q_i^2)=v(h_i-\rho_i)\ge v(h_i)$ (with equality when $t_i > 0$) for $i = 1, 2$.  We meanwhile have $v(\rho_i) = v(h_i) + t_i$ by definition of the $t_i$'s for $i = 1, 2$, which enables us to estimate the valuations of the three terms on the right-hand side of (\ref{eq product part-square}): we compute $v(\rho_1 q_2^2) \geq v(h_1) + t_1 + v(h_2) = v(h) + t_1$ (with equalities when $t_2 > 0$); symmetrically, we get $v(\rho_2 q_1^2) \geq v(h) + t_2$ (with equalities when $t_1 > 0$); and finally, we get $v(\rho_1\rho_2) = v(h_1) + t_1 + v(h_2) + t_2 = v(h) + t_1 + t_2$.  In particular, each summand on the right-hand side of (\ref{eq product part-square}) has valuation $\geq v(h) + \min\{t_1, t_2\}$ and so we have $v(\rho) \geq v(h) + \min\{t_1, t_2\}$, and we get $t = v(\rho) - v(h) \geq \min\{t_1, t_2\}$, proving the first claim of the proposition.

    Suppose that $t_2 > t_1$.  Then in particular, we have $t_2 > 0$, from which the relations $v(\rho_1 q_2^2) = v(h) + t_1$ and $v(\rho_2 q_1^2), v(\rho_1 \rho_2) > v(h) + t_1$ now follow from our previous computations.  Thus, since by (\ref{eq product part-square}) the element $\rho$ is the sum of $\rho_1 q_2^2$ plus higher-valuation terms, we have $v(\rho) = v(h) + t_1$ and thus $t = t_1$, proving part (a) in this case.  Moreover, we see that any normalized reduction of $\rho$ is also a normalized reduction of $\rho_1 q_2^2$.  Since the decomposition $h_1 = \rho_1 + q_1^2$ is assumed to be good, by \Cref{prop good decomposition}(a), any normalized reduction of $\rho_1$ is not a square and thus include odd-degree terms; the same is then true of any normalized reduction of $\rho_1 q_2^2$ or of $\rho$ itself.  It follows from \Cref{prop good decomposition}(a) that the decomposition $h = \rho + q^2$ is good, and part (b) is proved in this case.  Symmetrically, the claims of part (a) are also proved when $t_1 > t_2$.

    Now suppose that $t_1 = t_2$; we set out to prove part (b) and thus assume the hypothesis that all roots of $h_1$ have positive valuation and all roots of $h_2$ have negative valuation.  This directly implies that any normalized reduction of $h_1$ (resp. of $h_2$) is a polynomial of the form $c_1 z^m \in k[z]$ (resp. $c_2 \in k[z]$) for some $c_1, c_2 \in k^\times$.  Since a normalized reduction of $h_2$ is a square, it follows from \Cref{prop good decomposition} combined with \Cref{rmk good decompositions} that we have $t_2 > 0$, and so we have $0 < t_1 = t_2 < 2v(2)$; moreover, this means that the polynomial $c_1 z^m$ must also be a square and so the power $m$ is even.  Meanwhile, in this situation, our above computations of the valuations of the right-hand terms in (\ref{eq product part-square}) give us $v(\rho_1 q_2^2) = v(\rho_2 q_1^2) = v(h) + t_1$ and $v(\rho_1 \rho_2) = v(h) + 2t_1 > v(h) + t_1$.  Choosing $\gamma, \gamma_1, \gamma_2 \in (K^\alg)^\times$ to be appropriate elements of valuation $t_1$, $v(h_1)$, and $v(h_2)$ respectively, we get normalized reductions $\overline{\gamma_1^{-1} h_1}(z) = c_1 z^m$ and $\overline{\gamma_2^{-1} h_2}(z) = c_2$ of $h_1$ and $h_2$ respectively, and we get that the reductions $\overline{r_i}$ of $r_i := \gamma^{-1} \gamma_1^{-1} \rho_i$ are normalized reductions of $\rho_i$ for $i = 1, 2$.  Multiplying both sides of the equation in (\ref{eq product part-square}) by $\gamma_1^{-1} \gamma_2^{-1} \gamma^{-1}$ and reducing, we get 
    \begin{equation}
        \overline{\gamma_1^{-1}\gamma_2^{-1}\gamma^{-1}\rho} =  (\overline{\gamma_2^{-1}h_2}) \overline{r_1} + (\overline{\gamma_1^{-1}h_1}) \overline{r_2} = c_2 \overline{r_1} + c_1 z^{2m} \overline{r_2}.
    \end{equation}
    We remark that the polynomial $\overline{r_1}(z) \in k[z]$ has degree $\deg(\overline{r}_1)\le 2m$; hence, if an odd-degree term of degree $s$ appears in $\overline{r_1}$ (resp.\ in $\overline{r_2}$), then an odd-degree term of degree $s$ (resp.\ $s+2m$) will also show up in $\overline{\gamma_1^{-1}\gamma_2^{-1}\gamma^{-1}\rho}$: roughly speaking, in the expression for $\overline{\gamma_1^{-1}\gamma_2^{-1}\gamma^{-1}\rho}$ no cancellation occurs between the odd-degree monomials of $\overline{r_1}$ and those of $\overline{r_2}$. We conclude that, since $\overline{r_1}$ and $\overline{r_2}$ are not squares by \Cref{prop good decomposition}, the reduced polynomial $\overline{\gamma_1^{-1}\gamma_2^{-1}\gamma^{-1}\rho}$ is not a square, so that we have $v(\rho)=t_1+v(h)$ (i.e.\ $t=t_1=t_2$), and the decomposition of $h$ is good by \Cref{prop good decomposition}(a).
    \end{proof}
\end{prop}

\subsection{Forming models of \texorpdfstring{$Y$}{Y} using part-square decompositions} \label{sec models hyperelliptic forming}
In this subsection, we compute the model of the hyperelliptic curve $Y: y^2=f(x)$ corresponding to any given smooth model of the projective line $X$, in the sense of \S\ref{sec semistable galois covers}.  In doing so, we recover \cite[Proposition 1]{lehr2001reduction} in the hyperelliptic case (which may directly be compared to our \Cref{prop normalization model} below) as the models we compute follow similar computations to those done by Lehr to prove his proposition.

More precisely, choose elements $\alpha \in K^\alg$ and $\beta \in (K^\alg)^\times$.  Given any polynomial $h(x) \in K^\alg[x]$, define the translated and scaled coordinate $x_{\alpha,\beta} = \beta^{-1}(x - \alpha)$ (as defined in \S\ref{sec models hyperelliptic line}), and let $h_{\alpha,\beta}$ be the polynomial such that $h_{\alpha,\beta}(x_{\alpha,\beta}) = h(x)$.  Let $D := D_{\alpha,b}$ be a disc in $K^\alg$ with $\alpha\in K^\alg$ and $b=v(\beta)$ for some $\beta\in (K^\alg)^\times$. To this disc we can attach (see \S\ref{sec models hyperelliptic line}) a smooth model $\XX_D$ of the line $X$, defined over some extension of $R$. We will show that, after possibly replacing this extension with a further extension $R'$, which in particular will be large enough so that $f_{\alpha,\beta}$ admits a good part-square decomposition $f_{\alpha,\beta}=q_{\alpha,\beta}^2+\rho_{\alpha,\beta}$ over the fraction field $K'$ of $\Frac(R')$, the model of $Y$ corresponding to $\XX_D/R'$ has reduced special fiber, and its equation can explicitly be written using $q_{\alpha,\beta}$ and $\rho_{\alpha,\beta}$; we will denote this model by $\YY_{D}$.

The strategy will be the following one: after a suitable change of the coordinate $y$, we will rewrite the equation $y^2=f_{\alpha,\beta}(x_{\alpha,\beta})$ of the hyperelliptic curve $Y$ in the form
\begin{equation}
    \label{equation front chart YD}
    y^2 + q_0(x_{\alpha,\beta})y - \rho_0(x_{\alpha,\beta})=0,\quad\text{with }\deg(q_0)\le g+1,\ \deg(\rho_0)\le 2g+2,
\end{equation}
such that the following conditions are satisfied:
\begin{enumerate}[(A)]
    \item $\rho_0$ and $q_0$ have integral coefficients (i.e., we have $\rho_0(x_{\alpha,\beta}), q_0(x_{\alpha,\beta}) \in R'[x_{\alpha,\beta}]$);
    \item the $k$-curve given by the reduction of the equation in (\ref{equation front chart YD}) is reduced.
\end{enumerate}

Then, the model $\YY_D$ is constructed as follows.  The equation in (\ref{equation front chart YD}) above defines a scheme $W$ over $R'$ whose generic fiber is isomorphic to the affine chart $x_{\alpha,\beta}\neq \infty$ of the hyperelliptic curve $Y$. The coordinate $x_{\alpha,\beta}$ defines a map $W\to \XX_D$, whose image is the affine chart $x_{\alpha,\beta}\neq \infty$ of $\XX_D$. Over the affine chart $x_{\alpha,\beta}\neq 0$ of $\XX_D$, we can correspondingly form the $R'$-scheme $W^\vee$ defined by the equation \begin{equation}
\begin{split}
    \label{equation rear chart YD}
    \check{y}^2+q_0^\vee(\check{x}_{\alpha,\beta})\check{y}-\rho_0^\vee(\check{x}_{\alpha,\beta})&=0, \qquad\text{where }\  \check{x}_{\alpha,\beta}=x_{\alpha,\beta}^{-1},\  \check{y}=x_{\alpha,\beta}^{-(g+1)}y, \\
    q_0^\vee(\check{x}_{\alpha,\beta})&={x_{\alpha,\beta}}^{-(g+1)} q_0(x_{\alpha,\beta}), \ \text{and} \ \rho_0^\vee(\check{x}_{\alpha,\beta})={x_{\alpha,\beta}}^{-(2g+2)} \rho_0(x_{\alpha,\beta}).
\end{split}
\end{equation}

We can now \emph{define} $\YY_D$ to be the scheme obtained by gluing the affine charts $W$ and $W^\vee$ together in the obvious way: it is endowed with a degree-$2$ covering map $\YY_D\to \XX_D$, and its generic fiber is identified with the hyperelliptic curve $Y\to X$.

\begin{prop}
    The scheme $\YY_D$ constructed above, which is defined over an appropriate extension $R'$ of $R$, coincides with the normalization of $\XX_D/R'$ in the function field of the hyperelliptic curve $Y$, and it is a model of $Y$ whose special fiber is reduced.
    \begin{proof}
        We have to show that the scheme $\YY_D$ we have constructed is normal. The $R'$-schemes $W$ and $W^\vee$ are complete intersections, and hence they are Cohen-Macaulay; as a consequence, to check that $\YY_D$ is normal, it is enough to prove that it is regular at its codimension-1 points. Since the generic fiber of $\YY_D$ coincides with $Y$, it is certainly regular; hence, all that is left is to check that $\YY_D$ is regular at the generic point $\eta_{V_i}$ of each irreducible component $V_i$ of the special fiber $\SF{\YY_D}$. Since we are assuming that the $k$-curve $W_s$ is reduced, the $k$-curve $\SF{\YY_D}$ is also clearly reduced, which implies that $\YY_D$ is certainly regular at the points $\eta_{V_i}$.  Thus, the scheme $\YY_D$ is actually normal, and it is clear that $\YY_D$ is the model of $Y$ obtained by normalizing $\XX_D/R'$ in the function field of $Y$.
    \end{proof}
\end{prop}

All we have to do now is determine a change of the coordinate $y$ such that the conditions (A) and (B) above are satisfied. To do this, suppose that we are given a good part-square decomposition $f_{\alpha, \beta} = q_{\alpha,\beta}^2 + \rho_{\alpha,\beta}$ (which certainly exists over some extension of $K$, thanks to \Cref{prop totally odd existence} and \Cref{cor totally odd is good}).  Let $\gamma\in (K^\alg)^\times$ be an element whose valuation is $v(\gamma)=\truncate{t}+v(f_{\alpha,\beta})$, where $t=t_{q_{\alpha,\beta},\rho_{\alpha,\beta}}=v(\rho_{\alpha,\beta})-v(f_{\alpha,\beta})$. We remark that we necessarily have $t\ge 0$, since the part-square decomposition is assumed to be good (see \Cref{rmk good decompositions}). The change of variable we perform is $y \mapsto \gamma^{1/2} y + q_{\alpha,\beta}(x_{\alpha,\beta})$, and it leads to an equation of the form (\ref{equation front chart YD}) with 
\begin{equation}
    \label{equation q_0 rho_0 formulas YD}
        q_0 = 2\gamma^{-1/2}q_{\alpha,\beta} \qquad \mathrm{and} \qquad
        \rho_0 = \gamma^{-1}\rho_{\alpha,\beta}.
\end{equation}

The valuations of $q_0$ and $\rho_0$ can be computed as follows.
\begin{enumerate}
    \item For $q_0$, we have $2v(q_0)= 2v(2)-\truncate{t}+2v(q_{\alpha,\beta})-v(f_{\alpha,\beta})$. Let us remark that, since $t\ge 0$, we have $2v(q_{\alpha,\beta})\ge v(f_{\alpha,\beta})$, and moreover equality holds whenever $t>0$. We deduce that:
    \begin{enumerate}[(i)]
        \item $v(q_0)\ge 2v(2)-\truncate{t}$ for all $t$, so that $q_0$ is consequently always integral;
        \item $v(q_0)=2v(2)-\truncate{t}$ whenever $t>0$;
        \item $v(q_0)>0$ if $0\le t< 2v(2)$; and 
        \item $v(q_0)=0$ if $t\ge 2v(2)$.
    \end{enumerate}
    \item For $\rho_0$, we have $v(\rho_0) = t - \truncate{t}$; in particular,
    \begin{enumerate}[(i)]
        \item $\rho_0$ is always integral;
        \item $v(\rho_0)=0$ if $0\le t\le 2v(2)$; and 
        \item $v(\rho_0)>0$ if $t>2v(2)$.
    \end{enumerate}
\end{enumerate}
These computations guarantee that the condition (A) above is satisfied. We now verify that also the condition (B) is satisfied.
\begin{lemma}
    In the context above, the condition (B) above is also satisfied, i.e.\ the reduction of equation (\ref{equation front chart YD}) defines a reduced $k$-curve. Moreover, this curve is a separable (resp.\ inseparable) quadratic cover of the $k$-line of coordinate $x_{\alpha,\beta}$ if and only if $t\ge 2v(2)$ (resp.\ $0\le t<2v(2)$).
    \begin{proof}
        Suppose by way of contradiction that the $k$-curve defined by the reduction of  (\ref{equation front chart YD}) is non-reduced. This is clearly equivalent to saying that the polynomial $g(x_{\alpha,\beta},y)\in k[x_{\alpha,\beta},y]$ given by the reduction of (\ref{equation front chart YD}) (i.e.\  $g(x_{\alpha,\beta},y):=y^2+\overline{q_0(x_{\alpha,\beta})}y - \overline{\rho_0(x_{\alpha,\beta})}$) is a square. If we treat $g(x_{\alpha,\beta},y)$ as a monic quadratic polynomial in the variable $y$, we can say that it is a square if and only if its constant term $\overline{\rho_0(x_{\alpha,\beta})}\in k[x_{\alpha,\beta}]$ is a square, and its discriminant $\Delta=\overline{{q}_0}^2+4\overline{\rho_0}=\overline{4\gamma^{-1} f_{\alpha,\beta}}\in k[x_{\alpha,\beta}]$ is zero. However, when $t\ge 2v(2)$, we have $v(\gamma)=v(4f_{\alpha,\beta})$ and therefore $\Delta\neq 0$; when $0\le t<2v(2)$, the reduced polynomial $\overline{\rho_0}$ is a normalized reduction of $\rho_{\alpha,\beta}$, which is not a square by \Cref{prop good decomposition}. We conclude that the $k$-curve $g(x_{\alpha,\beta},y)=0$ is always reduced.
        
        Now the coordinate $x_{\alpha,\beta}$ defines a quadratic cover from the $k$-curve $g(x_{\alpha,\beta},y)=0$ to the affine $k$-line, and it is immediate to realize that this cover is inseparable only when the linear term $\overline{q_0(x_{\alpha,\beta})}y$ vanishes, which happens if and only if $0<t\le 2v(2)$.
    \end{proof}
\end{lemma}

The following proposition summarizes the results we have obtained.
\begin{prop}
	\label{prop normalization model}
	
	Let $\XX_D$ be the smooth model of the line corresponding to the disc $D:=D_{\alpha,v(\beta)}$, with $\alpha \in K^\alg$ and $\beta \in (K^\alg)^\times$.  Then, after replacing $K$ with an appropriate finite extension, the normalization $\YY_D$ of $\XX_D$ in $K(Y)$ has reduced special fiber. Given a good part-square decomposition $f_{\alpha, \beta} = q_{\alpha,\beta}^2 + \rho_{\alpha,\beta}$, and letting $t=t_{q_{\alpha,\beta}, \rho_{\alpha,\beta}}$, the model $\YY_D$ falls under (exactly) one of the following two cases:
	\begin{enumerate}[(i)]
		\item $t\geq 2v(2)$; in this case, $\SF{\YY_D}$ is a separable degree-2 cover of $\SF{\XX_D}$; and 
		\item $0\le t < 2v(2)$; in this case, $\SF{\YY_D}$ is an inseparable degree-2 cover of $\SF{\XX_D}$.
	\end{enumerate}
	The equations describing the affine charts $x_{\alpha,\beta}\neq \infty$ and $x_{\alpha,\beta}\neq 0$ of the model $\YY_D$ have the form (\ref{equation front chart YD}) and (\ref{equation rear chart YD}) respectively, and they can be explicitly computed from $q_{\alpha,\beta}$ and $\rho_{\alpha,\beta}$ using the formulas in (\ref{equation q_0 rho_0 formulas YD}).
\end{prop}

\begin{rmk} \label{rmk inseparable is inessential}
One can show through an elementary but tedious combinatorial argument that, roughly speaking, the role of the inseparable components in $\SF{\Yrst}$ is inessential: they are just lines that only get added whenever it is necessary to create room between three or more separable components that would otherwise intersect at the same point and violate semistability.  Moreover, inseparable components cannot even occur when the genus of $Y$ is $1$ or $2$.  It is in light of this that we focus only on separable components in the next subsection and eventually define \emph{valid discs} (\Cref{dfn valid disc} below) to correspond only to separable components.
\end{rmk}

\subsection{The special fiber \texorpdfstring{$\SF{\YY_D}$}{(YD)s} in the separable case}
\label{sec models hyperelliptic separable}
We will now study the special fiber of the model $\YY_D$ associated to a given disc $D:=D_{\alpha,v(\beta)}$ (for some $a\in K^\alg$ and $\beta \in (K^\alg)^\times$) which was computed in the previous subsection.  This subsection will consider the case in which $\SF{\YY_D}\to \SF{\XX_D}$ is separable: this means that it is possible to find a part-square decomposition $f_{\alpha,\beta}=q_{\alpha,\beta}^2+\rho_{\alpha,\beta}$ satisfying $t:=t_{q_{\alpha,\beta}, \rho_{\alpha,\beta}} \geq 2v(2)$, and the equation of $\SF{\YY_D}$ has the form $y^2 + \overline{q_0(x_{\alpha,\beta})}y = \overline{\rho_0(x_{\alpha,\beta})}$ with $q_0=2\gamma^{-1/2}q_{\alpha,\beta}$ and $\rho_0=\gamma^{-1}\rho_{\alpha,\beta}$, where $\gamma \in (K^\alg)^{\times}$ is an element of valuation $v(f_{\alpha,\beta})+2v(2)$.

We remark that the separable quadratic cover $\SF{\YY_D}\to \SF{\XX_D}$ is branched precisely above the points $P_1, \ldots, P_N$ of $\SF{\XX_D}$ at which the branch locus $\Rinfty$ reduces and is étale elsewhere: this can be seen directly from the equation of $\SF{\YY_D}$, or can be deduced from the fact the branch locus of $\YY_D\to \XX_D$ has pure dimension 1 by the Zariski–Nagata purity theorem. In order to state and prove the results in this subsection, we partition the branch locus $R=\lbrace P_1, \ldots, P_N\rbrace\subseteq \SF{\XX_D}(k)$ in three subsets as $R=R_0\sqcup R_1\sqcup R_2$, in the following way.
\begin{equation*}
    \begin{split}
    R_0&=\lbrace P\in \SF{\XX_D}: \text{$\SF{\YY_D}$ exhibits a unique smooth point $Q$ above $P$}\rbrace;\\
    R_1&=\lbrace P\in \SF{\XX_D}: \text{$\SF{\YY_D}$ has a (unique) singular point $Q$ above $P$ and has one branch at $Q$}\rbrace;\\
    R_2&=\lbrace P\in \SF{\XX_D}: \text{$\SF{\YY_D}$ has a (unique) singular point $Q$ above $P$ and has two branches at $Q$}\rbrace.
    \end{split}
\end{equation*}
We denote the cardinality of each subset $R_i \subseteq R$ by $N_i$ for $i = 0, 1, 2$.
\begin{rmk} \label{rmk R_0 R_1 R_2}
    The following statements are clear from the definitions above.
    \begin{enumerate}[(a)]
        \item The set $R_0\cup R_1$ is precisely the branch locus of the quadratic cover $\NSF{\YY_D}\to \SF{\XX_D}$, where $\NSF{\YY_D}$ is the normalization of the $k$-curve $\SF{\YY_D}$.
        \item The curve $\SF{\YY_D}$ has exactly $N_1+N_2$ singular points, which lie over the $N_1+N_2$ points of $R_1\cup R_2$.
        \item The unique point $Q\in \SF{\YY_D}$ lying over some given $P\in R$ is fixed by the action of the hyperelliptic involution. If $P\in R_2$, the two branches of $\SF{\YY_D}$ passing through $Q$ get flipped by the hyperelliptic involution.
        \item The special fiber $\SF{\YY_D}$ either consists of two components flipped by the hyperelliptic involution, or it is irreducible. In the first case (which always occurs, for example, if $\overline{\rho_0(x_{\alpha,\beta})}$ is the zero polynomial, i.e.\ if $t>2v(2)$), the two components are necessarily two lines that trivially cover $\SF{\XX_D}$, while $\NSF{\YY_D}$ is their disjoint union, and we have $R_0\cup R_1 = \varnothing$.  If $\SF{\YY_D}$ is irreducible, however, the quadratic cover $\NSF{\YY_D}\to \SF{\XX_D}$ is necessarily ramified, because $\mathbb{P}^1_k$ does not have non-trivial finite étale connected covers: hence, and we have $R_0\cup R_1 \neq \varnothing$.
    \end{enumerate}
\end{rmk}

We want to better understand the ramification behaviour of $\NSF{\YY_D}\to \SF{\XX_D}$ above the points of $R_0\cup R_1$; to this aim, we can measure, above each point, the length of the module of relative K\"{a}hler differentials of the cover.
\begin{dfn}
    \label{dfn ell ramification index}
    Given $P\in \SF{\XX_D}(k)$, we set \begin{equation*}\ell(\XX_D,P) = \length_{\OO_{\SF{\XX_D},P}}\left(\Omega_{\NSF{\YY_D}/\SF{\XX_D}}\otimes \OO_{\SF{\XX_D},P}\right).\end{equation*}
\end{dfn}

\begin{rmk}
    \label{rmk ell}
    For any $P\in \SF{\XX_D}(k)$, the integer $\ell(\XX_D,P)$ satisfies the following properties.
    \begin{enumerate}[(a)]
    \item If $P\not\in R_0\cup R_1$, then $\NSF{\YY_D}\to \SF{\XX_D}$ is unramified over $P$, and we thus have $\ell(\XX_D,P)=0$.
    \item If $P\in R_0\cup R_1$, and we denote by $Q$ its unique preimage $\NSF{\YY_D}$, the ramification index of the cover $\NSF{\YY_D}\to \SF{\XX_D}$ at $Q$ is $e_Q=2$, and \cite[Proposition 7.4.13]{liu2002algebraic} ensures that $\ell(\XX_D,P)\ge e_Q-1$, with equality if and only if the cover is tame. As we are in a wild setting, this means that we have $\ell(\XX_D,P)\ge 2$.
    \end{enumerate}
\end{rmk}

The knowledge of $\ell(\XX_D,P)$ at the points $P \in \SF{\XX_D}$ gives us information about the abelian rank of $\SF{\YY_D}$.

\begin{prop}
    \label{prop riemann hurwitz}
    The genus of $\NSF{\YY_D}$ is given by
    \begin{equation}\label{equation rh}
        g\left(\NSF{\YY_D}\right)=-1+\frac{1}{2}\sum_{P\in \SF{\XX_D}(k)} \ell(\XX_D,P),
    \end{equation}
    with the convention that the genus of the disjoint union of two lines is $-1$.  In our setting, this formula implies the inequality
    \begin{equation}
        \label{equation rh p=2}
       g\left(\NSF{\YY_D}\right)\ge -1+(N_0+N_1).
    \end{equation}
    \begin{proof}
        The equation in (\ref{equation rh})  is just the Riemann-Hurwitz formula (see, for example, \cite[Theorem 7.4.16]{liu2002algebraic}), while (\ref{equation rh p=2}) follows from (\ref{equation rh}) via \Cref{rmk ell}.
    \end{proof}
\end{prop}

We now see how to compute $\ell(\XX_D,P)$ for a given point $P\in \SF{\XX_D}$ from the good part-square decomposition $f_{\alpha,\beta}=q_{\alpha,\beta}^2+\rho_{\alpha,\beta}$ given.

\begin{lemma}
    \label{lemma computation ell ramification}
    Choose $P\in \SF{\XX_D}(k)$. Let us denote by $n_q(P):=\ord_P(\overline{q_0})$ and $n_\rho(P):=\ord_P(\overline{\rho_0})$ the respective orders of vanishing at the point $P$ of the reductions of the polynomials $q_0$ and $\rho_0$ defined in (\ref{equation q_0 rho_0 formulas YD}), with the convention that the zero polynomial has vanishing order $\infty$, and that, if $P=\infty$, the vanishing orders of $\overline{\rho_0}$ and $\overline{q_0}$ at $\infty$ are respectively those of $\overline{\rho_0^\vee}$ and $\overline{q_0^\vee}$ at $0$, i.e.\  $n_q=g+1 - \deg(\overline{q_0})$ and  $n_\rho=2g+2 - \deg(\overline{\rho_0})$.
    \begin{enumerate}[(a)]
        \item If $2n_q(P)\le n_\rho(P)$, then we have $\ell(\XX_D,P)=0$.
        \item If $2n_q(P)>n_\rho(P)$ and $n_\rho(P)$ is odd, then we have $\ell(\XX_D,P)=2n_q(P)-n_\rho(P)+1$.
    \end{enumerate}
    \begin{proof}
        We lose no generality in assuming $P$ has coordinate $\overline{x_{\alpha,\beta}}=0$.  For brevity we write $z$ for the variable $x_{\alpha,\beta}$. We may write the equation of $\SF{\YY_D}$ as $y^2 + \overline{q_0}(z)y = \overline{\rho_0}(z)$, with $\overline{q_0}(z) = z^{n_q} q_1(z)$ and $\overline{\rho_0}(z) = z^{n_\rho} r_1(z)$, where $q_1(z), \rho_1(z) \in k[z]$ do not vanish at 0.  We proceed by desingularizing $\SF{\YY_D}$ above $P$ by means of a sequence of blowups.
        
        Assume that $2n_q \le n_\rho$. Then, after $n_q$ blowups at $(0,0)$, we obtain $y^2 + q_1(z)y = z^{n_\rho-2n_q} \rho_1(z)$. Since $q_1(0)\neq 0$, there are exactly $2$ solutions for $y$ at $z=0$, which means that the blown-up curve is étale above $P$, implying that $\ell(\XX_D,P) = 0$.  We have thus proved part (a).
        
        Assume that $2n_q > n_\rho$ and that $n_\rho$ is odd. Then, after $\frac{1}{2}(n_\rho-1)$ blowups at $(0,0)$, we obtain the equation 
        \begin{equation} \label{eq 2n_q > n_rho}
            y^2 + z^{n_q-(n_\rho-1)/2}q_1(z)y = z \rho_1(z).
        \end{equation}
        The curve given by (\ref{eq 2n_q > n_rho}) has a unique point $(0,0)$ above $z=0$ and it is non-singular at that point; this is enough to guarantee that $\ell(\XX_D,P) > 0$.  Let $B := k[z,y]_{(z)}/(\text{equation in }(\ref{eq 2n_q > n_rho}))$ be the local ring of functions on the blown-up curve at $(0,0)$, which is a free $k[z]_{(z)}$-algebra of rank 2. Then $\ell(\XX_D,P)$ equals the length of the $k[z]_{(z)}$-module $\Omega_{B/k[z]_{(z)}}$, or, equivalently, the dimension over $k$ of $\Omega_{B/k[z]_{(z)}}$. We have an isomorphism of $k[z]_{(z)}$-modules 
        \begin{equation}
            \Omega_{B/k[z]_{(z)}} = B dy / (z^{n_q-(n_\rho-1)/2}q_1(z) dy) \stackrel{\sim}{\to} (k[z]_{(z)})[y] / (y^2 - zr_1(z), z^{n_q-(n_\rho-1)/2}q_1(z)),
        \end{equation}
        where the isomorphism is given by sending $dy$ to 1.  The latter $k[z]_{(z)}$-module, however, is a free algebra of rank 2 over the ring
        $$k[z]_{(z)}/(z^{n_q-(n_\rho-1)/2}q_1(z)) \cong k[z]/(z^{n_q-(n_\rho-1)/2}),$$
        which clearly has dimension $n_q-\frac{1}{2}(n_\rho-1)$ over $k$.  From this, part (b) follows.
    \end{proof}
\end{lemma}
\begin{rmk} \label{rmk ramification index}
        \Cref{lemma computation ell ramification} allows us to calculate $\ell(\XX_D,P)$ from a given good part-square decomposition of $f_{\alpha,\beta}$ only in certain cases: in fact, when $2n_q(P)> n_\rho(P)$ and $n_\rho(P)$ is even, the lemma is inconclusive. At the same time, we remark that if we choose a \emph{totally odd} part-square decomposition for $f_{\alpha,\beta}$ (which can always be done by \Cref{prop totally odd existence}), the polynomial $\overline{\rho_0}$ will certainly have a zero of odd multiplicity at the points $0$ and $\infty$ of $\SF{\XX_D}$; hence, we will certainly be able to compute $\ell(\XX_D,P)$ where $P$ is the point $\overline{x_{\alpha,\beta}}=0$ or the point $\overline{x_{\alpha,\beta}}=\infty$ via the lemma. In other words, given a point $P\in \SF{\XX_D}$, by appropriately choosing the center $\alpha$ of the disc $D$ and constructing a totally odd decomposition for $f_{\alpha,\beta}$, \Cref{lemma computation ell ramification} allows us to compute $\ell(\XX_D,P)$ at the point, and the result it produces is a non-negative even integer.
\end{rmk}

We now give a criterion to determine whether $\XX_D\le \Xrst$ (which is equivalent to saying that $\YY_D\le \Yrst$).
\begin{thm}
    \label{thm part of rst separable}
    Assume that $D$ is a disc such that $\SF{\YY_D}\to \SF{\XX_D}$ is separable, and let $N$ denote the number of points of $\SF{\XX_D}$ to which the roots $\Rinfty$ reduce. We have $\XX_D\le \Xrst$ if and only if one of the following conditions holds:
    \begin{enumerate}[(i)]
            \item $N\ge 3$;
            \item $N \geq 2$ and $\SF{\YY_D}$ is irreducible; or 
            \item $N=1$ and $\SF{\YY_D}$ is irreducible of positive abelian rank.
    \end{enumerate}
    Moreover, whenever $\XX_D\le \Xrst$, the strict transform of the $k$-curve $\SF{\YY_D}$ in $\SF{\Yrst}$ is smooth, and it consequently coincides with its normalization $\NSF{\YY_D}$.
    \begin{proof}
        The result essentially follows from a combinatorial argument that directly makes use of the description we have given of $\SF{\YY_D}$ in this subsection, by applying the criterion we have presented in \Cref{prop part of rst}.  Let us write $N=N_0+N_1+N_2$ as we did at the beginning of this subsection.
        
        Suppose that $\SF{\YY_D}$ is not irreducible.  As we have seen in \Cref{rmk R_0 R_1 R_2}(d), this is equivalent to saying that $N_0=N_1=0$, and the curve $\SF{\YY_D}$ consists, in this case, of two lines $L_1$ and $L_2$ meeting each other above the $N=N_2$ points of $\SF{\XX_D}$; the number of singular points of $\SF{\YY_D}$ is $N$, and through each singular point one branch of $L_1$ and one branch of $L_2$ pass, flipped by the hyperlliptic involution. We have $m(L_i)=1$, $a(L_i)=0$, $w(L_i)=N$, and that the stabilizer of $L_i$ in the Galois group $G$ is trivial for $i = 1, 2$; hence, \Cref{prop part of rst} ensures that $\XX_D\le \Xrst$ if and only if $N\ge 3$.
        
        Suppose now that $\SF{\YY_D}$ is irreducible, which is to say that $N_0+N_1 \geq 1$, and let $V=\SF{\YY_D}$ denote the unique irreducible component of $\SF{\YY_D}$. We have the following: 
        \begin{itemize}
            \item $w(V)=N_1+2N_2$;
            \item if $N_2 \geq 1$, then there are $2$ branches of $V$ passing through a singular point which are flipped by $G$; and 
            \item $a(V)\ge -1+N_0+N_1$ by \Cref{prop riemann hurwitz}.
        \end{itemize}
        Suppose that $N = 1$.  If $a(V) \geq 1$ then by Proposition \ref{prop part of rst} we have $\XX_D \leq \Xrst$, while if $a(V) = 0$, then we have $w(V) \leq 1$ and so Proposition \ref{prop part of rst} says that we have $\XX_D \not\leq \Xrst$.
        
        Now suppose that $N \geq 2$.  If $N_2 \geq 1$, then we have $w(V) \geq 2$ and that $V$ has $2$ branches passing through a singular point which are flipped by $G$, and so we get $\XX_D \leq \Xrst$ by Proposition \ref{prop part of rst}.  If $N_2 = 0$, then we have $N_0 + N_1 \geq 2$ so that $a(V) \geq 1$, and so we again get $\XX_D \leq \Xrst$ by \Cref{prop part of rst}.
        
        The statement about the strict transform of $\SF{\YY_D}$ in $\SF{\Yrst}$ is an immediate consequence of \Cref{prop smoothness components rst}.
    \end{proof}
\end{thm}

\section{Clusters and valid discs} \label{sec clusters}

We begin this section by defining \emph{clusters (of roots)}, \emph{depths}, and \emph{relative depths} of clusters, and the \emph{cluster picture} associated to the odd-degree polynomial $f(x)$ defining the hyperelliptic curve $Y: y^2=f(x)$.  This notion of ``cluster" is equivalent to the one found in \cite[Definition 1.1]{dokchitser2022arithmetic}.  We then set up the notion of \emph{valid discs} associated to $f$, so that each one corresponds to a component of $\SF{\Xrst}$.

\begin{dfn}
    \label{dfn cluster}
    Given a non-empty subset $\mathfrak{s}\subseteq \RR$, we say that $\mathfrak{s}$ is a \emph{cluster} (of $\mathcal{R}$) if there exists a disc $D\subseteq K^\alg$ such that $D\cap \RR=\mathfrak{s}$.
    
    We define the \emph{depth} of a cluster $\mathfrak{s} \subseteq \RR$ to be the value $d_+(\mathfrak{s}) := \min_{\zeta, \zeta' \in \mathfrak{s}} v(\zeta - \zeta') \in \qqinfty$.
    
    The set of pairs $(\mathfrak{s},d_+(\mathfrak{s}))$ for all clusters $\mathfrak{s}$ of $\RR$, is called the \emph{cluster data} of $\RR$.

    For every cluster $\mathfrak{s}\subsetneq \RR$, the \emph{parent cluster} of $\mathfrak{s}$ is the smallest cluster $\mathfrak{s}'$ properly containing it; in this situation, we say that $\mathfrak{s}$ is a \emph{child cluster} of $\mathfrak{s}'$. Two distinct clusters having the same parent are said to be \emph{sibling clusters}.

    Given any cluster $\mathfrak{s} \subsetneq \mathcal{R}$, we define the value $d_-(\mathfrak{s})$ to be the depth of the parent cluster $\mathfrak{s}'$ of $\mathfrak{s}$, and we set $d_-(\mathcal{R}) = -\infty$.  We write $I(\mathfrak{s}) = [d_-(\mathfrak{s}), d_+(\mathfrak{s})]$.

    The \emph{relative depth} of a cluster $\mathfrak{s} \subseteq \RR$ is defined to be $\delta(\mathfrak{s}) := d_+(\mathfrak{s}) - d_-(\mathfrak{s}) \in \qqminusinfty$.
\end{dfn}

\begin{rmk} \label{rmk cluster}
    We note the following.
    \begin{enumerate}[(a)]
        \item We have that $\RR$ itself is always a cluster, with $d_-(\RR)=-\infty$, $d_+(\RR)$ finite, and $\delta(\RR)=+\infty$.
        \item For every $a \in \mathfrak{s}$, the singleton $\{a\}$ is always a cluster, with $d_-(\{a\})$ finite, $d_+(\{a\})=+\infty$ (as by convention the minimum taken over the empty set is $+\infty$), and $\delta(\{s\})=+\infty$.
    \end{enumerate}
\end{rmk}

We want now to introduce some further definitions for later use that relate the clusters to the discs that cut them out of $\mathcal{R}$. We begin with the following remark.

\begin{dfn}
    \label{dfn linked}
    Given a cluster $\mathfrak{s}\subseteq \RR$, we say that a disc $D\subseteq K^\alg$ is \emph{linked to} $\mathfrak{s}\subseteq \mathcal{R}$ if we have $D = D_{\mathfrak{s},b}$ for some depth $b\in I(\mathfrak{s})$.
\end{dfn}

\begin{rmk}
    A disc $D$ is linked to a cluster $\mathfrak{s}$ if and only if we have either $D\cap \RR=\mathfrak{s}$ or that $D$ is the minimal disc such that $D\cap \RR\supsetneq \mathfrak{s}$.  More precisely, if $D = D_{\mathfrak{s}, b}$ for some $b \in I(\mathfrak{s})$, we have $D\cap \RR=\mathfrak{s}$ whenever $b\in (d_-(\mathfrak{s}),d_+(\mathfrak{s})]$, whereas $D$ is the smallest disc such that $D\cap\RR\supsetneq \mathfrak{s}$ when $b=d_-(\mathfrak{s})$; moreover, in this case the subset $D\cap\RR \subseteq \RR$ is the parent cluster of $\mathfrak{s}$.
\end{rmk}

We observe that, given a disc $D$, the points of $\Rinfty$ reduce to distinct points $P_1, \ldots,\linebreak[0] P_{N-1},\linebreak[0] \infty \in \SF{\XX_D}(k)$ for some $N \geq 1$; we can accordingly write $\RR=\mathfrak{c}_1\sqcup \ldots\sqcup \mathfrak{c}_{N-1}\sqcup \mathfrak{c}_\infty$, where $\mathfrak{c}_i$ consists of the roots in $\RR$ reducing to $P_i$, and $\mathfrak{c}_\infty$ consists of the roots in $\RR$ reducing to $\infty$. We clearly have $D\cap \RR=\mathfrak{c}_1\sqcup \ldots \sqcup \mathfrak{c}_{N-1}=\RR\setminus \mathfrak{c}_\infty$. It is an easy combinatorial exercise to verify the following proposition.

\begin{prop}
    \label{lemma discs linked to clusters}
    With notation as above, we have the following.
    \begin{enumerate}[(a)]
        \item If $N=1$, which is equivalent to $D\cap \RR=\varnothing$, the disc $D$ is not linked to any cluster.
        \item If $N=2$, the disc $D$ is linked to exactly one cluster, namely $\mathfrak{c}_1=\mathcal{R}\setminus \mathfrak{c}_\infty=D\cap \mathcal{R}$.
        \item[(c)] If $N \ge 3$, the disc $D$ is linked to exactly $N$ clusters, namely $\mathfrak{c}_1, \ldots, \mathfrak{c}_{N-1}$, and $\mathfrak{s}:=D\cap \RR=\RR\setminus \mathfrak{c}_\infty = \mathfrak{c}_1\sqcup \ldots \sqcup \mathfrak{c}_{N-1}$.
    \end{enumerate}

    Moreover, case (c) occurs if and only if we have $D=D_{\mathfrak{s},d_+(\mathfrak{s})}$ for some non-singleton cluster $\mathfrak{s}\subseteq \mathcal{R}$, in which case the subsets $\mathfrak{c}_1, \ldots, \mathfrak{c}_{N-1}$ are precisely the children clusters of $\mathfrak{s}$, and we have $D=D_{\mathfrak{c}_i,d_-(\mathfrak{c}_i)}$ for $i=1,\ldots, N-1$.
\end{prop}

We now define a term which we will use throughout the rest of the paper in order to refer to components of the relatively stable model of a hyperelliptic curve.

\begin{dfn} \label{dfn valid disc}
     A disc $D \subseteq K^\alg$ is a \emph{valid disc} if it satisfies $\XX_D \leq \Xrst$ and if the quadratic cover $\SF{\YY_D}\to \SF{\XX_D}$ is separable.
\end{dfn}

We note that our notion of \emph{valid disc} differs from the one in \cite{dokchitser2022arithmetic}, although in both cases valid discs are used to build a particular semistable model of $Y$ with desired properties.

The cluster picture allows us to completely identify the valid discs in the tame setting.

\begin{rmk} \label{rmk cluster p odd}

If the residue characteristic is different from $2$, the cluster data alone fully characterizes the valid discs in a manner that is easy both to state and to prove.  In this setting, there is a one-to-one correspondence between non-singleton clusters and valid discs, which is given by $\mathfrak{s}\mapsto D_{\mathfrak{s},d_+(\mathfrak{s})}$. In other words, the valid discs are precisely those discs that minimally cut out the clusters in $\mathcal{R}$.  Moreover, the components of the relatively smooth model $\Yrst$ of the hyperelliptic curve $Y$ come \textit{entirely} from valid discs -- more precisely, all discs $D \subset K^\alg$ such that we have $\XX_D \leq \Xrst$ satisfy that the quadratic cover $\SF{\Yrst} \to \SF{\Xrst}$ is separable (so that the second condition in \Cref{dfn valid disc} is superfluous).  This is in fact what first led the authors to the definition of the relatively stable model.

\end{rmk}

We want some analog of \Cref{rmk cluster p odd} for working over residue characteristic $2$; however, in this setting, valid discs do not correspond in a one-on-one manner with clusters.  Instead, what we will show is that there are no valid discs linked to odd-cardinality clusters, while an even-cardinality cluster may have $0$, $1$, or $2$ valid discs linked to it, as in \Cref{thm summary depths valid discs} below.  In order to get this result, we need to set up a framework for considering the models $\YY_{D}$ corresponding to families of discs $D:=D_{\alpha,b}$ which all contain the same subset $\mathfrak{s}\subseteq \RR$ of roots.  This is the main goal of the next section.

\section{Finding valid discs with a given center} \label{sec depths}

In this section, given an even-cardinality cluster $\mathfrak{s} \subseteq \mathcal{R}$ and a center $\alpha\in \mathfrak{s}$, we investigate for which depths $b \in I(\mathfrak{s}) := [d_-(\mathfrak{s}), d_+(\mathfrak{s})]$ the disc $D_{\alpha,b}$ is valid for the hyperelliptic curve $Y: y^2 = f(x)$.  

The section is organized as follows. In \S\ref{sec depths piecewise-linear} we introduce the language of translated and scaled part-square decompositions, which will be useful for dealing with the problem. In \S\ref{sec depths construction valid discs} we identify the depths $b\in I$ for which $D_{\alpha,b}$ is a valid disc (if such depths exist) as the endpoints $b_{\pm}$ of a sub-interval $J(\mathfrak{s}) \subseteq I(\mathfrak{s})$ (see \Cref{thm summary depths valid discs}); in doing so, we prove \Cref{thm introduction main}(a),(b).  In \S\ref{sec depths separating}, we develop our strategy for determining $J(\mathfrak{s})$ provided that, for each of the two factors $f^{\mathfrak{s}}$ and $f^{\RR\setminus \mathfrak{s}}$ of $f$ corresponding to the roots lying in $\mathfrak{s}$ and $\RR\setminus\mathfrak{s}$ respectively, we know a part-square decomposition that is totally odd with respect to the center $\alpha$.  Then we apply these results in \S\ref{sec depths threshold} to describe conditions on the relative depth of an even-cardinality cluster which ensure that such a sub-interval $J(\mathfrak{s}) \subseteq I(\mathfrak{s})$ exists and prove \Cref{thm introduction main}(c),(d).

\subsection{Translated and scaled part-square decompositions} 
\label{sec depths piecewise-linear}

Given any polynomial $h(z) \in K^\alg[z]$ and any choice of elements $\alpha\in K^\alg, \beta\in(K^\alg)^\times$, we can compute the (Gauss) valuations of the translated and scaled polynomial $h_{\alpha,\beta}$ obtained from $h$ by translating by $\alpha$ and scaling by $\beta$ (as defined in \S\ref{sec models hyperelliptic forming}).  The following lemma will allow us to treat the Gauss valuation of a certain translation and scaling of $h$ as a function on discs.

\begin{lemma}
    \label{lemma vfun of disc}
    As we vary $\alpha \in K^\alg$ and $\beta \in (K^\alg)^\times$, the valuation $v(h_{\alpha,\beta})$ depends only on the disc $D:=D_{\alpha,v(\beta)}$.
    \begin{proof}
        This is simply the $p$-adic Maximum Modulus Principle, which says that 
        \begin{equation} \label{eq maximum modulus}
            v(h_{\alpha, \beta}) = \inf_{z_{\alpha,\beta} \in \mathcal{O}_{K^\alg}} v(h_{\alpha,\beta}(z_{\alpha,\beta})) = \inf_{z \in D} v(h(z)).
        \end{equation}
    \end{proof}
\end{lemma}

Given a disc $D$, we will consequently denote $\vfun_h(D)$ the valuation of $h_{\alpha,\beta}$ for any $\alpha$ and $\beta$ such that $D=D_{\alpha,v(\beta)}$; note from (\ref{eq maximum modulus}) that we have $\vfun_h(D) = \inf_{z \in D} v(h(z))$.  
When a center $\alpha\in K^\alg$ is fixed, we may consider the function $b \mapsto \vfun_h(D_{\alpha,b})\in\qqinfty$ defined for all $b \in \qq$.

\begin{lemma} \label{lemma mathfrakh}
    Suppose a center $\alpha\in K^\alg$ is fixed. With the above set-up, we have the following.
    \begin{enumerate}[(a)]
        \item The function $b \mapsto \vfun_h(D_{\alpha,b})\in\qqinfty$ satisfies the property of being a continuous, non-decreasing piecewise linear function with integer slopes and whose slopes decrease as the input increases.
        \item For any $b \in \mathbb{Q}$ and $\beta \in (K^\alg)^{\times}$ such that $v(\beta) = b$, the left (resp.\ right) derivative of the function $c \mapsto \vfun_h(D_{\alpha,c})\in\qqinfty$ at $c = b$ coincides with the highest (resp.\ lowest) degree of the variable $x_{\alpha,\beta}$ appearing in the normalized reduction of $h_{\alpha,\beta}$, i.e.\ with the number of roots $\zeta$ of $h$ in $K^\alg$ (counted with multiplicity) such that $v(\zeta-\alpha)\ge b$ (resp.\ $v(\zeta-\alpha)>b$).
    \end{enumerate}
\end{lemma}

\begin{proof}

Write $H_i$ for the $z^i$-coefficient of $h_{\alpha,1}$, and note that $\beta^i H_i$ is the $z^i$-coefficient of $h_{\alpha,\beta}$ for any scalar $\beta$.  Now given any $b \in \qq$ and $\beta \in (K^\alg)^{\times}$ with $v(\beta) = b$, by definition we have 
\begin{equation} \label{eq mathfrakv}
\vfun_h(D_{\alpha,b}) = \min_{0 \leq i \leq \deg(h)} \{v(\beta^i H_i)\} = \min_{0 \leq i \leq \deg(h)} \{v(H_i) + ib\}.
\end{equation}
All the properties of the function $b\mapsto \vfun_h(D_{\alpha,b})$ stated in the lemma immediately follow from the explicit expression given above.
\end{proof}

Given a part-square decomposition $h = q^2 + \rho$ of a nonzero polynomial $h$, by translating and scaling we can clearly form the part-square decompositions $h_{\alpha,\beta}=q_{\alpha,\beta}^2+\rho_{\alpha,\beta}$ for all $\alpha\in K^\alg, \beta\in (K^\alg)^\times$.

\begin{lemma}
    \label{prop translated part-square decompositions}
    Let $h=q^2+\rho$ be a part-square decomposition.
    \begin{enumerate}[(a)]
        \item The property of the induced part-square decomposition $h_{\alpha,\beta} = q_{\alpha,\beta}^2 + \rho_{\alpha,\beta}$ being good or not only depends on the disc $D_{\alpha,v(\beta)}$ and not on the particular choices of $\alpha$ and $\beta$.
        \item The property of the induced part-square decomposition $h_{\alpha,\beta} = q_{\alpha,\beta}^2 + \rho_{\alpha,\beta}$ being totally odd or not does not depend on $\beta$.
    \end{enumerate}
    
    \begin{proof}
        Part (a)  is an immediate consequence of \Cref{lemma vfun of disc}, while part (b) is immediate.
    \end{proof}
\end{lemma}

We can consequently define the following, which are the variants of those given in \Cref{dfn good totally odd} relative to the choice of a disc.
\begin{dfn}
    \label{dfn translated scaled good totally odd}
    Let $h=q^2+\rho$ be a part-square decomposition. We make the following definitions:
    \begin{enumerate}[(a)]
        \item the decomposition is \emph{good} at a disc $D$ whenever $h_{\alpha,\beta} = q_{\alpha,\beta}^2 + \rho_{\alpha,\beta}$ is a good part-square decomposition for some (any) $\alpha\in K^\alg$ and $\beta\in (K^\alg)^\times$ such that $D=D_{\alpha,v(\beta)}$; and 
        \item the decomposition is \emph{totally odd} with respect to a center $\alpha\in K^\alg$ if $h_{\alpha,\beta}=q_{\alpha,\beta}^2+\rho_{\alpha,\beta}$ is a totally odd part-square decomposition for some (any) $\beta\in (K^\alg)^\times$.
    \end{enumerate}
\end{dfn}

\begin{rmk}
    \label{rmk totally odd good}
    If $h=q^2+\rho$ is a totally odd part-square decomposition with respect to a center $\alpha$, then, by \Cref{cor totally odd is good} combined with \Cref{prop translated part-square decompositions}(b), it is good at the discs $D_{\alpha,b}$ for all $b\in \qq$. 
\end{rmk}

Recalling the number $t_{q,\rho} := v(\rho) - v(h)\in \qqinfty$ from \S\ref{sec models hyperelliptic part-square}, we define the related function
$$\tfun_{q, \rho} := \vfun_\rho - \vfun_f$$
so that $\tfun_{q,\rho}(D) = t_{q_{\alpha,\beta}, \rho_{\alpha,\beta}}$ for any $\alpha\in K^\alg$, $\beta \in (K^\alg)^\times$ such that $D=D_{\alpha,v(\beta)}$.  When a center $\alpha\in K^\alg$ is fixed, we can study the function $b \mapsto \tfun_{q,\rho}(D_{\alpha,b}) : \qq \to \qq \cup \{+\infty\}$, which is the difference between two continuous piecewise-linear functions and so is itself a continuous piecewise-linear function. Taking into account \Cref{rmk same t for good}, we can give the following definition.
\begin{dfn}
    \label{dfn t fun}
    Given a set of elements $\mathfrak{s}\subseteq K^\alg$ and a disc $D$, we define $\tbest{\mathfrak{s}}{D}\in \zerotwo$ to be $\truncate{\tfun_{q,\rho}(D)}$ for any part-square decomposition $h=q^2+\rho$ which is good at the disc $D$, where $h(z)\in K^\alg[z]$ is any polynomial whose set of roots is $\mathfrak{s}$.
\end{dfn}
\begin{rmk}
    \label{rmk t fun computing}
    Fix a center $\alpha$. If $h\in K^\alg[z]$ is a nonzero polynomial and $\mathfrak{s}$ is its set of roots, the knowledge of a part-square decomposition $h=q^2+\rho$ that is totally odd with respect to the center $\alpha$ makes it possible to compute $\tbest{\mathfrak{s}}{D_{\alpha,b}}\in \zerotwo$ for all depths $b\in \qq$: this follows immediately from \Cref{dfn t fun} together with \Cref{rmk totally odd good}.
\end{rmk}

\begin{prop}
    \label{prop mathfrak t minimum}
    Suppose that we have a disjoint union $\mathfrak{s}=\mathfrak{s}_1 \sqcup \cdots \sqcup \mathfrak{s}_N$ of clusters.  Then we have $\mathfrak{t}^{\mathfrak{s}}(D) \geq \min\{\mathfrak{t}^{\mathfrak{s}_1}(D), \dots, \mathfrak{t}^{\mathfrak{s}_N}(D)\}$.  Moreover, this is an equality in the following cases:
        \begin{enumerate}[(i)]
            \item whenever the minimum is attained by a unique $\tbest{\mathfrak{s}_i}{D}$; and 
            \item if $N=2$, $D\cap \mathfrak{s}_2=\varnothing$, and there exists a disc $D'\subsetneq D$ such that $\mathfrak{s}_1\subseteq D'$.
        \end{enumerate}
    
    \begin{proof}
        For $1 \leq i \leq N$, choose polynomials $h_i$ having $\mathfrak{s}_i$ as their sets of roots, and let $h_i=q_i^2+\rho_i$ be part-square decompositions that are good at the disc $D$. Then, by setting $q = \prod_i q_i$, we obtain a part-square decomposition for $h:=\prod_i h_i$ satisfying $\tfun_{q,\rho}(D)\ge \min_i \{\tfun_{q_i,\rho_i}(D)\}$ by \Cref{prop product part-square}. From this the first claim follows. Similarly, equality under the hypotheses of (i) and (ii) follows from parts (a) and (b) of \Cref{prop product part-square} respectively (one checks straightforwardly that the hypotheses in (ii) precisely correspond to the hypotheses of \Cref{prop product part-square}(b)).
    \end{proof}
\end{prop}

\subsection{Identifying the valid discs}
\label{sec depths construction valid discs}

Given a cluster $\mathfrak{s} \subseteq \RR$ and a choice of center $\alpha \in \mathfrak{s}$, our aim is to establish for which $b\in I(\mathfrak{s})$ the disc $D_{\alpha,b}$ is a valid disc.  To this end, we will introduce a (possibly empty) closed sub-interval $J(\mathfrak{s})$, whose endpoints (if they exist) will correspond to the depths $b\in I(\mathfrak{s})$ for which $D_{\alpha,b}$ is a valid disc; the precise statement is given in \Cref{thm summary depths valid discs}.

Let us begin by studying the function $I(\mathfrak{s}) \ni b \mapsto \tbest{\RR}{D_{\alpha,b}}\in \zerotwo$, which enjoys the following properties.

\begin{lemma} \label{lemma slopes of t}
    The function $ I(\mathfrak{s}) \ni b \mapsto \tbest{\RR}{D_{\alpha,b}}$ is a continuous piecewise-linear function with decreasing slopes. It is identically zero if $|\mathfrak{s}|$ is odd.  On the other hand, when $|\mathfrak{s}|$ is even, its slopes are odd integers ranging from $1-|\mathfrak{s}|$ to $2g+1-|\mathfrak{s}|$, except over the subset of $I(\mathfrak{s})$ where $\tbest{\RR}{D_{\alpha,b}}=2v(2)$, over which the slope is zero (if this subset contains an open interval).
    \begin{proof}
        Choose an interior point $b \in I(\mathfrak{s})$, i.e.\ $b\in (d_-(\mathfrak{s}),d_+(\mathfrak{s}))$, and choose $\beta\in (K^\alg)^\times$ such that $v(\beta)=b$. As $b$ lies in the open interval $(d_-(\mathfrak{s}),d_+(\mathfrak{s}))$, any normalized reduction of $f_{\alpha,\beta}$ is a scalar times $x_{\alpha,\beta}^{|\mathfrak{s}|}$. We deduce from \Cref{prop good decomposition} that, if $|\mathfrak{s}|$ is odd, the part-square decomposition $f_{\alpha,\beta}=0^2+f_{\alpha,\beta}$ is good and $\tbest{\RR}{D_{\alpha,b}}=0$, whereas, if $|\mathfrak{s}|$ is even, this decomposition is not good, and we therefore have $\tbest{\RR}{D_{\alpha,b}}>0$. In the latter case, let us take a part-square decomposition $f=q^2+\rho$ which is totally odd with respect to the center $\alpha$, so that $\tbest{\RR}{D_{\alpha,b}}=\truncate{\tfun_{q,\rho}(D_{\alpha,b})}$. Since $\deg(f)=2g+1$, by \Cref{dfn qrho} the odd-degree polynomial $\rho$ has degree at most $2g+1$. Now, $b \mapsto \tfun_{q,\rho}(D_{\alpha,b})$ is, by definition, the difference between the functions $b \mapsto \vfun_{\rho}(D_{\alpha,b})$ and $b \mapsto \vfun_{f}(D_{\alpha,b})$; by \Cref{lemma vfun of disc}, the former is a piecewise linear function with decreasing odd integer slopes between $1$ and $2g+1$, while the latter is linear with slope $|\mathfrak{s}|$ over $I(\mathfrak{s})$.
    \end{proof}
\end{lemma}

In light of the above lemma, either $b \mapsto \tbest{\RR}{D_{\alpha,b}}$ is always $<2v(2)$ over $I(\mathfrak{s})$, or else it attains the output $2v(2)$ over some closed sub-interval of $I(\mathfrak{s})$ and is $<2v(2)$ elsewhere. Let $J(\mathfrak{s})$ denote the sub-interval of $I(\mathfrak{s})=[d_-(\mathfrak{s}),d_+(\mathfrak{s})]$ over which the output of $b \mapsto \tbest{\RR}{D_{\alpha,b}}$ equals $2v(2)$; in the former case just mentioned, we have $J(\mathfrak{s})=\varnothing$, while in the latter case, the interval will have the form $J(\mathfrak{s})=[b_-(\mathfrak{s}),b_+(\mathfrak{s})]$ for some endpoints $b_\pm(\mathfrak{s})$.

\begin{rmk} \label{rmk structure of J}
    If the cluster $\mathfrak{s}$ has odd cardinality, then we have $\tbest{\RR}{D_{\alpha,b}} = 0$ for all $b\in I(\mathfrak{s})$, and therefore we have $J(\mathfrak{s})=\varnothing$.
\end{rmk}

By \Cref{prop normalization model}, it is clear that, given $D=D_{\alpha,b}$ with $b\in I(\mathfrak{s})$, the cover $\SF{\YY_D}\to \SF{\XX_D}$ is separable if and only if $b\in J(\mathfrak{s})$; in particular, for $b\in I(\mathfrak{s})$, the disc $D:=D_{\alpha,b}$ can only be valid if $b\in J(\mathfrak{s})$. To establish for which $b\in J(\mathfrak{s})$ the disc $D$ is valid, we need the following general lemma, which will allow us to compute the ramification of the cover $\SF{\YY_{D}}\to \SF{\XX_{D}}$ above $0$ and $\infty$.
\begin{lemma} \label{lemma ell and t function}
    Fix a center $\alpha\in K^\alg$, and choose $b\in \qq$ such that $\tbest{\RR}{D_{\alpha,b}}=2v(2)$, and consider the model $\YY_D$ corresponding to the disc $D:=D_{\alpha,b}$.  Let $\ell(\XX_D,P)$ be the integer defined in \Cref{dfn ell ramification index} for any point $P$ of $\SF{\XX_D}$.  Write $\partial^+\mathfrak{t}^\RR$ (resp.\ $\partial^-\mathfrak{t}^\RR$) for the right (resp.\ left) derivative of the function $c \mapsto \tbest{\RR}{D_{\alpha,c}}$.  Then we have the following.
    \begin{enumerate}[(a)]
        \item If $\partial^+\tbest{\RR}{b}\ge 0$, then we have $\ell(\XX_D,0)=0$.
        \item If $\partial^+\tbest{\RR}{b}$ is odd and negative, then we have $\ell(\XX_D,0)=1-\partial^+\tbest{\RR}{b}$.
        \item If $\partial^-\tbest{\RR}{b} \leq 0$, then we have $\ell(\XX_D,\infty)=0$.
        \item If $\partial^-\tbest{\RR}{b}$ is odd and positive, then we have $\ell(\XX_D,\infty)=1+\partial^-\tbest{\RR}{b}$.
    \end{enumerate}
     
    \begin{proof}
        This is just a rephrasing of  \Cref{lemma computation ell ramification} using the language introduced in \S\ref{sec depths piecewise-linear}.
        To see this, let us fix a part-square decomposition $f=q^2+\rho$ that is totally odd with respect to the center $\alpha$, and let ${q_0}$ and ${\rho_0}$ be the polynomials involved in the statement of \Cref{lemma computation ell ramification}: they are defined as appropriate scalings of $q_{\alpha,\beta}$ and $\rho_{\alpha,\beta}$ for some chosen $\beta\in (K^\alg)^\times$ such that $v(\beta)=b$.
        
        Now the polynomial $\overline{q_0}$ is a normalized reduction of $q_{\alpha,\beta}$, and either the polynomial $\overline{\rho_0}$ is $0$ (when $\tfun_{q,\rho}(D_{\alpha,\beta})>2v(2)$), or it is a normalized reduction of $\rho_{\alpha,\beta}$ (when $\tfun_{q,\rho}(D_{\alpha,\beta})=2v(2)$).
        Now we have $\mathfrak{t}(D_{\alpha,c}) = \truncate{\tfun(D_{\alpha,c})}$ for all $c \in \qq$ (see \Cref{rmk t fun computing}); moreover, whenever $\tfun_{q,\rho}(D_{\alpha,c})>0$ (and hence, in particular, for all $c$ in a neighborhood of $b$), we can write $\tfun_{q,\rho}(D_{\alpha,c})=\vfun_\rho(D_{\alpha,c})-2\vfun_q(D_{\alpha,c})$, where, in light of \Cref{lemma mathfrakh}, the first summand only has odd slopes, while the second summand only has even slopes.
        Let $n_\rho$ and $n_q$ denote the orders of vanishing of $\overline{\rho_0}$ and $\overline{q_0}$ at $x_{\alpha,\beta}=0$. The assumption in (a) means that either we have $\tfun_{q,\rho}(D_{\alpha,b})>2v(2)$, or we have $\tfun_{q,\rho}(D_{\alpha,b})=2v(2)$ with $\partial^+ \tfun_{q,\rho}(b)\ge 0$; thanks to \Cref{lemma mathfrakh}, this can be translated into saying that $n_\rho\ge 2n_q$, and it is now evident that the conclusion of part (a) follows from \Cref{lemma computation ell ramification}. A similar reasoning can be followed to prove parts (b)--(d).
    \end{proof}
\end{lemma}

As a first application of the lemma above, we will show that a necessary condition for $D_{\alpha,b}$ to be a valid disc when $b\in I(\mathfrak{s})$ is that $b$ is an endpoint of the sub-interval $J(\mathfrak{s})\subseteq I(\mathfrak{s})$.
\begin{lemma}
    \label{lemma not a valid disc}
    Given $b\in I(\mathfrak{s})$ and letting $D = D_{\alpha,b}$, we have the following.
    \begin{itemize}
        \item[(a)] If $b\not\in J(\mathfrak{s})$, the cover $\SF{\YY_D}\to \SF{\XX_D}$ is inseparable; hence, we have in particular that $D$ is not a valid disc.
        \item[(b)] If $b$ is an interior point of $J(\mathfrak{s})$, then we have $\XX_D\not\le \Xrst$, and so $D$ is not a valid disc.
    \end{itemize}
    
    \begin{proof}
        The statement of (a) is a direct result of \Cref{prop normalization model}, as we have already discussed.  We therefore set out to prove the statement of (b); we assume that $J(\mathfrak{s})\neq \varnothing$ and $b_-(\mathfrak{s})< b< b_+(\mathfrak{s})$ and let $D = D_{\alpha,b}$. The number $N$ of distinct points of $\SF{\XX_D}$ to which the roots in $\mathcal{R}\cup \{\infty\}$ reduce is $\leq 2$; this is because, since $b$ does not coincide with an endpoint of the interval $I(\mathfrak{s})$, we have that the $2g+2$ roots $\Rinfty$ each reduce either to $0$ or to $\infty$ in $\SF{\XX_D}$. Moreover, since $b$ is an interior point of $J(\mathfrak{s})$, we have $\tbest{\RR}{D}=2v(2)$ and that the left and right derivatives of $b'\mapsto\tbest{\RR}{D_{\alpha,b'}}$ at $b'=b$ are both equal to $0$; by \Cref{lemma ell and t function}, this implies that $\SF{\YY_D}$ has two branches above $0\in \SF{\XX_D}$ and two branches above $\infty\in \SF{\XX_D}$, and the special fiber $\SF{\YY_D}$ consequently consists of two rational components (see \Cref{rmk R_0 R_1 R_2}).
        
        Now \Cref{thm part of rst separable} implies that $\XX_D\not\le \Xrst$; this is because we know that $N\le 2$ and that the special fiber $\SF{\YY_D}$ is not irreducible.
    \end{proof}
\end{lemma}

Let us write $\lambda_-(\mathfrak{s}) = \partial^-\tbest{\RR}{b_-}$ and $\lambda_+(\mathfrak{s}) = -\partial^+\tbest{\RR}{b_+}$ (where $\partial^\pm\mathfrak{t}^\RR$ is defined as in \Cref{lemma ell and t function}).  The integer $\lambda_-(\mathfrak{s})$ (resp.\ $\lambda_+(\mathfrak{s})$) is only defined if $J(\mathfrak{s})\neq \varnothing$ and its endpoint $b_-(\mathfrak{s})$ (resp.\ $b_+(\mathfrak{s})$) does not coincide with $d_-(\mathfrak{s})$ (resp.\ $d_+(\mathfrak{s})$). In particular, $\lambda_+(\mathfrak{s})$ and $\lambda_-(\mathfrak{s})$ can only be defined if $\mathfrak{s}$ has even cardinality (by \Cref{rmk structure of J}).  When defined, both integers $\lambda_\pm(\mathfrak{s})$ are odd (by \Cref{lemma slopes of t}); more precisely, we have $\lambda_-(\mathfrak{s})\in \{1, 3, \ldots, 2g+1-|\mathfrak{s}|\}$ and $\lambda_+(\mathfrak{s})\in \{1, 3, \ldots|\mathfrak{s}|-1\}$.

\begin{lemma}
    \label{prop lambda plus minus and ell}
    With the above notation, suppose that $J(\mathfrak{s}) \neq \varnothing$, and write $D_\pm = D_{\alpha,b_\pm(\mathfrak{s})}$. Then we have the following.
    \begin{enumerate}
        \item[(a)] Assume that $b_-(\mathfrak{s})<b_+(\mathfrak{s})$ and $|\mathfrak{s}|$ is even. Then we have $\ell(\XX_{D_-},0)=\ell(\XX_{D_+};\infty)=0$.
        \item[(b)] Assume that $b_-(\mathfrak{s})<b_+(\mathfrak{s})$ and $|\mathfrak{s}|$ is odd. Then we have $\ell(\XX_{D_-},0)=\ell(\XX_{D_+};\infty)=1$.
        \item[(c)] Assume that $b_-(\mathfrak{s})>d_-(\mathfrak{s})$. Then we have $\ell(\XX_{D_-},\infty)=1+\lambda_-(\mathfrak{s})$.
        \item[(d)] Assume that $b_+(\mathfrak{s})<d_+(\mathfrak{s})$. Then we have $\ell(\XX_{D_+},0)=1+\lambda_+(\mathfrak{s})$.
    \end{enumerate}
    \begin{proof}
        This follows immediately from \Cref{lemma ell and t function}, taking into account the properties that we have already discussed of the piecewise-linear function $I(\mathfrak{s}) \ni b \mapsto\tbest{\RR}{D_{\alpha,b}}$ and of the linear function $I \ni b \mapsto \vfun_f(D_{\alpha,b})$ in our setting.
    \end{proof}
\end{lemma}

We are now ready to state a necessary and sufficient condition for $D_{\alpha,b}$ to be a valid disc when $b\in I(\mathfrak{s})$ in the following theorem, which immediately implies \Cref{thm introduction main}(b).

\begin{thm}
    \label{thm summary depths valid discs}
    Let $\mathfrak{s}$ be a cluster of $\RR$, and fix a center $\alpha \in \mathfrak{s}$.  Let $D = D_{\alpha,b}$ for some $b\in I(\mathfrak{s})=[d_-(\mathfrak{s}),d_+(\mathfrak{s})]$ (in other words, let $D$ be any disc linked to $\mathfrak{s}$).  Then the disc $D$ is valid if and only if $b$ is an endpoint of $J(\mathfrak{s})$. Hence, there exist two (possibly coinciding) valid discs $D_{\alpha,b_-(\mathfrak{s})}$ and $D_{\alpha,b_+(\mathfrak{s})}$ linked to $\mathfrak{s}$ when $J(\mathfrak{s}) \neq \varnothing$, and there does not exist a valid disc linked to $\mathfrak{s}$ when $J(\mathfrak{s})=\varnothing$.

    \begin{proof}
        The forward direction of the statement is already immediately implied by \Cref{lemma not a valid disc}.  We therefore assume that $J(\mathfrak{s})\neq\varnothing$ and that $b\in \{ b_-(\mathfrak{s}),b_+(\mathfrak{s})\}$ and set out to prove that the disc $D = D_{\alpha, b}$ is valid.  We remark that, since $\tbest{\RR}{D}=2v(2)$, the cover $\SF{\YY_D}\to \SF{\XX_D}$ is separable thanks to \Cref{prop normalization model}; to determine whether or not $D$ is a valid disc, we may therefore apply the criterion stated in \Cref{thm part of rst separable}. Let $N$ be the integer defined in that theorem.
        
        Assume that $b$ is an endpoint of $I(\mathfrak{s})$.  Then \Cref{lemma discs linked to clusters} implies that the roots $\Rinfty$ reduce to $\geq 3$ distinct points of $\SF{\XX_D}$ (i.e., $N\ge 3$), and $D$ is certainly a valid disc by \Cref{thm part of rst separable}.
        
        Now assume instead that the rational number $b$ is an interior point of $I(\mathfrak{s})$.  Then the roots of $\mathfrak{s}$ reduce to $0\in \SF{\XX_D}$, while those of $\mathcal{R}\setminus \mathfrak{s}$, together with $\infty$, reduce to $\infty\in \SF{\XX_D}$. Let us assume that $b=b_-(\mathfrak{s})$: the $b=b_+(\mathfrak{s})$ case is analogous and will thus be omitted. We know from \Cref{prop lambda plus minus and ell}(b) that $\ell(\XX_-,\infty)=1+\lambda_-(\mathfrak{s})\ge 2$; in particular, $\SF{\YY_D}$ has only one branch above $\infty\in \SF{\XX_D}$ and is consequently irreducible.  We have $N=2$, and thus the criterion stated in \Cref{thm part of rst separable} ensures that $\XX_D\le \Xrst$, as desired.

    \end{proof}
\end{thm}

As a corollary to the above theorem, we easily obtain \Cref{thm introduction main}(a) as well, using \Cref{rmk structure of J} (along with its implication that at least in the particular case that there are no even-cardinality clusters, every valid disc $D$ is linked to no cluster and thus satisfies $D \cap \mathcal{R} = \varnothing$).

\subsection{Separating the roots (for an even-cardinality cluster \texorpdfstring{$\mathfrak{s}$}{s})}
\label{sec depths separating}
Throughout this subsection, we let $\mathfrak{s}$ be an even-cardinality cluster of roots and fix an element $\alpha \in \mathfrak{s}$.

\subsubsection{Factoring $f$}
\label{sec depths separating factorizing}
We write the polynomial $f(x)$ as a product $f(x)=c f^{\mathfrak{s}}(x) f^{\RR\setminus \mathfrak{s}}(x)$, where $c$ is the leading coefficient of $f$, and we write 
\begin{equation} \label{eq factorization}
    f^{\mathfrak{s}}(x) = \prod_{a \in \mathfrak{s}} (x - a)
    \qquad \mathrm{and} \qquad
    f^{\RR\setminus\mathfrak{s}}(x) = \prod_{a \in \mathcal{R}\setminus\mathfrak{s}} (x- a).
\end{equation}

Now let us define $\mathfrak{t}_+^{\mathfrak{s}}$ and $\mathfrak{t}_-^{\mathfrak{s}}$ to be the functions on the domain $[0,+\infty)$ given by 
\begin{equation} \label{eq mathfrak t pm}
    \mathfrak{t}_+^{\mathfrak{s}}: b \mapsto \tbest{\mathfrak{s}}{D_{\alpha,d_+(\mathfrak{s})-b}}\qquad \text{and}\qquad
    \mathfrak{t}_-^{\mathfrak{s}}: b \mapsto \tbest{\RR\setminus\mathfrak{s}}{D_{\alpha,b+d_-(\mathfrak{s})}}.
\end{equation}

Essentially, the function $\mathfrak{t}_+^{\mathfrak{s}}$ is defined by evaluating $\tbestsimple{\mathfrak{s}}$ on discs that are enlargements of $D_{\alpha,d_+(\mathfrak{s})}$; all such discs contain $\mathfrak{s}$, and  $D_{\alpha,d_+(\mathfrak{s})}$ is the minimal disc centered at $\alpha$ with this property.  Symmetrically, the function $\mathfrak{t}_-^{\mathfrak{s}}$ is defined by evaluating $\tbestsimple{\RR\setminus\mathfrak{s}}$ at contractions of $D_{\alpha,d_-(\mathfrak{s})}$ around the center $\alpha$: all such discs are disjoint from $\RR\setminus\mathfrak{s}$, except the largest one (i.e.\ $D_{\alpha,d_-(\mathfrak{s})})$, which is the minimal disc centered at $\alpha$ that intersects $\RR\setminus\mathfrak{s}$.

\begin{rmk}
    \label{rmk tpm indepdendent of the center}
    The function $\mathfrak{t}_+^{\mathfrak{s}}$ does not depend on the choice of $\alpha\in \mathfrak{s}$ used to define it. On the other hand, the function  $\mathfrak{t}_-^{\mathfrak{s}}(b)$ is the same for all $\alpha \in \mathfrak{s}$ only when $b\in [0,\delta(\mathfrak{s})]\subseteq [0,+\infty)$ but may depend on the choice of $\alpha \in \mathfrak{s}$ at inputs beyond this interval.  We justify leaving $\alpha$ out of our notation for $\mathfrak{t}_-^{\mathfrak{s}}$ by the fact that they hypotheses in our results below (in particular, \Cref{prop formulas for b_pm}) will only depend on the values this function takes on the interval $[0, \delta(\mathfrak{s})]$.
\end{rmk}

\begin{prop}
    \label{prop properties t plus minus}
    Each of the functions $\mathfrak{t}_\pm^{\mathfrak{s}}$ is strictly increasing on its domains until it reaches $2v(2)$ and become constant.  Over the part of the domain where $\mathfrak{t}^{\mathfrak{s}}_+$ (resp.\ $\mathfrak{t}^{\mathfrak{s}}_-$) is not constant, its slopes are decreasing odd integers between 1 and $|\mathfrak{s}|-1$ (resp.\ between 1 and $2g+1-|\mathfrak{s}|$).
    \begin{proof}
        We will only prove the result for $\mathfrak{t}^{\mathfrak{s}}_{+}$, as the proof for $\mathfrak{t}^{\mathfrak{s}}_{-}$ is analogous. Choose a part-square decomposition for $f^{\mathfrak{s}}=(q^{\mathfrak{s}})^2+\rho^{\mathfrak{s}}$ that is totally odd with respect to the center $\alpha$, so that  $\mathfrak{t}^{\mathfrak{s}}_{+}(b)=\truncate{\tfun_{q^\mathfrak{s},\rho^\mathfrak{s}}(D_{\alpha,b_+(\mathfrak{s})-b})}$ for all $b\in [0,+\infty)$. Since $\mathfrak{s}\subset D_{\alpha,b_+(\mathfrak{s})-b}$, we deduce from \Cref{lemma mathfrakh} that $[0, +\infty) \ni b \mapsto \vfun_{f^\mathfrak{s}}(D_{\alpha,b_+(\mathfrak{s})-b})$ has slope 0; on the other hand, the function $[0, +\infty) \ni b \mapsto \vfun_{\rho^\mathfrak{s}}(D_{\alpha,b_+(\mathfrak{s})-b})$ has odd integer slopes between 1 and $|\mathfrak{s}|-1$. From this, recalling that $\tfun_{q^\mathfrak{s}, \rho^\mathfrak{s}}=\vfun_{\rho^\mathfrak{s}}-\vfun_{f^\mathfrak{s}}$ by definition, the proposition follows.
    \end{proof}
\end{prop}

We now claim that the function $I(\mathfrak{s}) \ni b \mapsto \tbest{\RR}{D_{\alpha,b}} \in \zerotwo$ we have studied in the previous subsection can be completely recovered from $\mathfrak{t}_\pm^{\mathfrak{s}}$. In fact, we have the following.
\begin{prop} \label{prop t^R is min of t^s and t^(R-s)}
    For all $b \in I(\mathfrak{s})$, we have 
    \begin{equation*}
        \tbest{\RR}{D_{\alpha,b}}=\min\{\tbest{\mathfrak{s}}{D_{\alpha,b}},\tbest{\RR\setminus\mathfrak{s}}{D_{\alpha,b}}\}=\min\{\mathfrak{t}_+^{\mathfrak{s}}(d_+(\mathfrak{s})-b),\mathfrak{t}_-^{\mathfrak{s}}(b-d_-(\mathfrak{s}))\}.
    \end{equation*}
    \begin{proof}
        It is clearly enough to prove the result for $b$ an interior point of $I(\mathfrak{s})$, which will extend by continuity to the endpoints of $I(\mathfrak{s})$. For such an input $b$, we note that the roots $s\in \mathfrak{s}$ satisfy $v(s-\alpha)<b$, while the roots $s\in \RR\setminus \mathfrak{s}$ satisfy $v(s-\alpha)>b$. As a consequence, the hypothesis in point (ii) of \Cref{prop mathfrak t minimum} holds and we get the desired equality.
    \end{proof}
\end{prop}

\subsubsection{A standard form for the two factors}
\label{sec depths separating std form}
Let us introduce the polynomials
\begin{equation} \label{eq standard form}
    f_+^{\mathfrak{s},\alpha}(z):=\prod_{a \in \mathfrak{s}} (1 - \beta_+^{-1}(a-\alpha) z)
    \qquad \mathrm{and} \qquad
    f_-^{\mathfrak{s},\alpha}(z):=\prod_{a \in \mathcal{R}\setminus\mathfrak{s}} (1 -  \beta_-(a-\alpha)^{-1}z),
\end{equation}
where the scalars $\beta_\pm \in (K^\alg)^{\times}$ are chosen to satisfy $v(\beta_\pm) = d_\pm(\mathfrak{s})$.  These are just transformed versions of $f^{\mathfrak{s}}$ and $f^{\RR\setminus\mathfrak{s}}$, normalized so that their constant terms are $1$ and all coefficients are integral.  More precisely, we have the conversion formulas 
\begin{equation*}
    f^{\mathfrak{s}} = \beta_+^{|\mathfrak{s}|}(f^{\mathfrak{s},\alpha}_+)^\vee(\beta_+^{-1} (z-\alpha)) \qquad f^{\mathcal{R} \smallsetminus \mathfrak{s}} = \big(\prod_{a \in \RR \smallsetminus \mathfrak{s}} (\alpha - a)\big) f^{\mathfrak{s},\alpha}_-(\beta_-^{-1} (z-\alpha)),
\end{equation*}
where $(f^{\mathfrak{s},\alpha}_+)^\vee(z) = z^{|\mathfrak{s}|}f^{\mathfrak{s},\alpha}_+(1/z)$. Given a part-square decomposition for $f^{\mathfrak{s},\alpha}_+$ and for $f^{\mathfrak{s},\alpha}_-$, there is an obvious way of producing one for $f^{\mathfrak{s}}$ and $f^{\mathcal{R} \smallsetminus \mathfrak{s}}$, which in turn induces one for $f = c f^{\mathfrak{s}} f^{\RR \smallsetminus \mathfrak{s}}$.  More precisely, given two part-square decompositions $f^{\mathfrak{s},\alpha}_+ = q_+^2 + \rho_+$ and $f^{\mathfrak{s},\alpha}_- = q_-^2 + \rho_-$, one obtains the decompositions
\begin{equation*}
    f^\mathfrak{s} = [q^{\mathfrak{s}}]^2 + \rho^{\mathfrak{s}}, \qquad
    f^{\RR \smallsetminus \mathfrak{s}} = [q^{\RR \smallsetminus \mathfrak{s}}]^2 + \rho^{\RR \smallsetminus \mathfrak{s}}, \qquad
    f = q^2 + \rho
\end{equation*}
by setting $q^{\mathfrak{s}} = \beta_+^{|\mathfrak{s}|/2} q_+^\vee(\beta_+^{-1}(z-\alpha))$, $q^{\RR \smallsetminus \mathfrak{s}} = \sqrt{\prod_{a \in \RR \smallsetminus \mathfrak{s}} (\alpha - a)}q_-(\beta_-^{-1}(z-\alpha))$, and $q = \sqrt{c} q^{\mathfrak{s}} q^{\RR \smallsetminus \mathfrak{s}}$ (after making appropriate choices of square roots), where $q_+^\vee(z) = z^{|\mathfrak{s}|/2} q_+(1/z)$.

\begin{rmk} \label{rmk tpm and fpm}
    We observe the following.
    \begin{enumerate}[(a)]
        \item By construction, we have 
        \begin{equation}
            \tfun_{q^{\mathfrak{s}},\rho^{\mathfrak{s}}}(D_{\alpha,b}) = \tfun_{q_+,\rho_+}(d_+(\mathfrak{s})-b) \ \ \mathrm{and} \ \ \tfun_{q^{\RR \smallsetminus \mathfrak{s}},\rho^{\RR \smallsetminus \mathfrak{s}}}(D_{\alpha,b})=\tfun_{q_-,\rho_-}(b-d_-(\mathfrak{s}))
        \end{equation}
        for all $b \in \mathbb{Q}$; moreover, the above decomposition of $f^{\mathfrak{s}}$ (resp.\ $f^{\RR \smallsetminus \mathfrak{s}}$) is good at $D_{\alpha,b}$ if and only if the above decomposition of $f^{\mathfrak{s}}_+$ (resp.\ $f^{\mathfrak{s}}_-$) is good at $d_+(\mathfrak{s})-b$ (resp.\ $b-d_-(\mathfrak{s})$).
        \item It follows from part (a) above that the introduction of $f^{\mathfrak{s}}_\pm$ allows us to reinterpret the function $\mathfrak{t}_\pm^{\mathfrak{s}}$ as $[0,+\infty) \ni b \mapsto \tbest{Z_\pm}{D_{0,b}}$, where $Z_\pm$ denotes the set of roots of $f_\pm^{\mathfrak{s}}$. We remark that the translation and homotheties that define $f_\pm^{\mathfrak{s}}$ are chosen so that all elements of $Z_\pm$ have valuation $\le 0$ and some element in each of $Z_+$ and $Z_-$ has valuation $0$.
        \item Part (b) above implies that the knowledge of a totally odd part-square decomposition for $f_\pm^{\mathfrak{s}}$ allows us to compute $\mathfrak{t}_\pm^{\mathfrak{s}}$: this is just \Cref{rmk t fun computing}. More precisely, if $f_\pm^{\mathfrak{s}}=q_\pm^2+\rho_\pm$ is a totally odd part-square decomposition, we have $\mathfrak{t}_\pm^{\mathfrak{s}}(b)=\truncate{\tfun_{q_\pm,\rho_\pm}(D_{0,b})}$ for all $b\in [0,+\infty)$.
    \end{enumerate}
\end{rmk}

\subsubsection{Reconstructing the invariants $b_\pm$} 
\label{sec depths separating reconstructing invariants}
Let $b_0(\mathfrak{t}_\pm^{\mathfrak{s}})$ be the least value of $b \in [0,+\infty)$ at which  $\mathfrak{t}_\pm^{\mathfrak{s}}: [0,+\infty)\to \zerotwo$ attains $2v(2)$, and let $\lambda(\mathfrak{t}_\pm^{\mathfrak{s}})$ denote the left derivative of $\mathfrak{t}_\pm^{\mathfrak{s}}$ at $b_0(\mathfrak{t}_\pm^{\mathfrak{s}})$, which is clearly only defined when $b_0(\mathfrak{t}_\pm^{\mathfrak{s}})>0$. These invariants are closely related to those introduced in the previous subsection; in particular, the following proposition can be used to provide alternate formulas for the invariants $b_\pm(\mathfrak{s})$.
\begin{prop} \label{prop formulas for b_pm}
     Suppose that we have the inequality $d_+(\mathfrak{s}) - b_0(\mathfrak{t}_+^{\mathfrak{s}}) < d_-(\mathfrak{s}) + b_0(\mathfrak{t}_-^{\mathfrak{s}})$ (or the equivalent inequality $\delta(\mathfrak{s}) < b_0(\mathfrak{t}_+^{\mathfrak{s}}) + b_0(\mathfrak{t}_-^{\mathfrak{s}})$).  Then we have $J(\mathfrak{s}) = \varnothing$.  Otherwise, we have $J(\mathfrak{s}) \neq \varnothing$ and the formulas 
\begin{equation}
    \label{eq bpm}
    b_\pm(\mathfrak{s})=d_\pm \mp b_0(\mathfrak{t}_\pm^{\mathfrak{s}})\qquad \text{and}\qquad \lambda_\pm(\mathfrak{s}) = \lambda(\mathfrak{t}_\pm^{\mathfrak{s}}).
\end{equation}
\end{prop}

\begin{proof}
    Taking into account that we have $\tbest{\RR}{D_{\alpha,b}}=\min\{\mathfrak{t}_+^{\mathfrak{s}}(d_+(\mathfrak{s}) - b), \mathfrak{t}_-^{\mathfrak{s}}(b - d_-(\mathfrak{s}))\}$ by \Cref{prop t^R is min of t^s and t^(R-s)} and that each of the functions $\mathfrak{t}_\pm^{\mathfrak{s}}$ is strictly increasing until it reaches $2v(2)$ by \Cref{prop properties t plus minus}, the proposition follows from an unwinding of the definitions of $b_0(\mathfrak{t}_\pm^{\mathfrak{s}})$ and $b_\pm(\mathfrak{s})$.
\end{proof}

\subsection{Estimating thresholds for depths of even-cardinality clusters}
\label{sec depths threshold}
The results that we have obtained in the above subsections show that given an even-cardinality cluster $\mathfrak{s}$ of roots associated to the hyperelliptic curve $Y : y^2 = f(x)$, there are $0$, $1$, or $2$ valid discs linked to it, and the results suggest how we may determine how many valid discs are linked to it via the knowledge of the rational numbers $b_\pm(\mathfrak{s})=d_{\pm}(\mathfrak{s})\mp b_0(\mathfrak{t}^{\mathfrak{s}}_\pm)$.  Roughly speaking, the results of \S\ref{sec depths construction valid discs} and \S\ref{sec depths separating} show that an even-cardinality cluster $\mathfrak{s}$ has $2$ (resp.\ $1$) valid discs linked to it if and only if its relative depth $\delta(\mathfrak{s}) = d_+(\mathfrak{s}) - d_-(\mathfrak{s})$ exceeds (resp.\ equals) some threshold depending on $\mathfrak{s}$, namely the rational number given by $b_0(\mathfrak{t}^{\mathfrak{s}}_+) + b_0(\mathfrak{t}^{\mathfrak{s}}_-)$.  The precise statement is the following rephrasing of \Cref{thm summary depths valid discs} combined with \Cref{prop formulas for b_pm}; in turn, together with \Cref{prop viable thicknesses} below, it proves \Cref{thm introduction main}(c).

\begin{prop} \label{prop depth threshold}
    Write $B_{f,\mathfrak{s}} = b_0(\mathfrak{t}^{\mathfrak{s}}_+) + b_0(\mathfrak{t}^{\mathfrak{s}}_-)$.  Given an even-cardinality cluster $\mathfrak{s} \subset \mathcal{R}$ of relative depth $\delta(\mathfrak{s})$, there are exactly $2$ (resp.\ $1$; resp.\ $0$) valid discs linked to $\mathfrak{s}$ if we have $\delta(\mathfrak{s}) > B_{f,\mathfrak{s}}$ (resp.\ $\delta(\mathfrak{s}) = B_{f,\mathfrak{s}}$; resp.\ $\delta(\mathfrak{s}) < B_{f,\mathfrak{s}}$).
\end{prop}

\begin{rmk} \label{rmk B independence}
    We note that the rational number $B_{f,\mathfrak{s}}$ given in the above corollary does not depend on the depth $\delta(\mathfrak{s})$ in the following sense.  Given a center $\alpha \in \mathfrak{s}$, let $\mathfrak{s}_{[\lambda]} = \{\lambda(a - \alpha) + \alpha \ | \ a \in \mathfrak{s}\}$ for some $\lambda \in (K^\alg)^{\times}$ such that $v(\lambda) > -\delta(\mathfrak{s})$, so that $\mathfrak{s}_{[\lambda]}$ is a scaled version of $\mathfrak{s}$ and is a cluster in $\RR_{[\lambda]} := \mathfrak{s}_{[\lambda]} \sqcup (\mathcal{R} \smallsetminus \mathfrak{s})$ with relative depth $\delta(\mathfrak{s}_{[\lambda]}) = \delta(\mathfrak{s}) + v(\lambda)$. 
    Then it follows easily from the constructions in \S\ref{sec depths separating} that we have $\mathfrak{t}_+^\mathfrak{s} = \mathfrak{t}_+^{\mathfrak{s}_{[\lambda]}}$
    and $\mathfrak{t}_-^\mathfrak{s} = \mathfrak{t}_-^{\mathfrak{s}_{[\lambda]}}$, from which it follows that $B_{f_{[\lambda]},\mathfrak{s}_{[\lambda]}} = B_{f,\mathfrak{s}}$.
    In this sense, loosely speaking, we may view $B_{f,\mathfrak{s}}$ as a sort of ``threshold'' for the depth of $\mathfrak{s}$ at which we obtain $1$ valid disc linked to $\mathfrak{s}$ and above which we obtain $2$ valid discs linked to $\mathfrak{s}$.
\end{rmk}

In the rest of this subsection, we work towards obtaining estimates and exact formulas for the ``threshold depth" $B_{f,\mathfrak{s}}$ defined in \Cref{prop depth threshold} under various conditions on $\mathfrak{s} \subset \mathcal{R}$.

\begin{prop} \label{prop deep ubereven cluster}
    Suppose that $\mathfrak{s}$ is an even-cardinality cluster which itself is the disjoint union of even-cardinality child clusters $\mathfrak{c}_1, \ldots , \mathfrak{c}_N$ for some $N\ge 2$.  The minimum of the set 
    \begin{equation*}
        \{\mathfrak{t}_+^{\mathfrak{c}_i}(\delta(\mathfrak{c}_i))\}_{1 \leq i \leq N} \cup \{\mathfrak{t}_+^{\mathfrak{s}}(0)\}
    \end{equation*}
    of rational numbers is attained by more than one element. In particular, if we have $\mathfrak{t}_+^{\mathfrak{c}_i}(\delta(\mathfrak{c}_i)) = 2v(2)$ for $1 \leq i \leq N$, then we have $\mathfrak{t}_+^{\mathfrak{s}}(0) = 2v(2)$ also.
    
    \begin{proof}
        If the minimum of the set $\{\mathfrak{t}_+^{\mathfrak{c}_i}(\delta(\mathfrak{c}_i)\}_{1 \leq i \leq N}$ is attained by more than one element, then we are done, so let us assume that this minimum is attained by a unique element.  Then, if we apply \Cref{prop mathfrak t minimum} to the disc $D:=D_{\mathfrak{s},d_+(\mathfrak{s})}$, we obtain $\mathfrak{t}_+^{\mathfrak{s}}(0) = \min_{1 \leq i \leq N} \mathfrak{t}_+^{\mathfrak{c}_i}(\delta(\mathfrak{c}_i))$, and the claim is proved.
    \end{proof}
\end{prop}

\begin{lemma} \label{lemma estimates of b_0 from slopes}
    Let $\mathfrak{s}$ be a cluster of even cardinality.  We have \begin{equation}
        b_0(\mathfrak{t}^{\mathfrak{s}}_+) \geq \frac{2v(2) - \mathfrak{t}^{\mathfrak{s}}_+(0)}{|\mathfrak{s}| - 1} \ \ \mathrm{and} \ \ b_0(\mathfrak{t}^{\mathfrak{s}}_-) \geq \frac{2v(2) - \mathfrak{t}^{\mathfrak{s}}_-(0)}{2g + 1 - |\mathfrak{s}|}.
    \end{equation}
    Moreover, if $\lambda_+(\mathfrak{s}) = |\mathfrak{s}| - 1$ (resp.\ $\lambda_-(\mathfrak{s}) = 2g + 1 - |\mathfrak{s}|$), then the first (resp.\ second) inequality above is an equality.
    \begin{proof}
        This follows immediately from the properties of $\mathfrak{t}^{\mathfrak{s}}_\pm$ presented in \Cref{prop properties t plus minus}.
    \end{proof}
\end{lemma}

\begin{lemma} \label{lemma tplus minus odd cardinality child}
    Let $\mathfrak{s}$ be a cluster of even cardinality. Then we have the following:
    \begin{enumerate}[(a)]
        \item if $\mathfrak{s}$ has an odd-cardinality child cluster, then we have $\mathfrak{t}^{\mathfrak{s}}_+(0)=0$; and 
        \item if $\mathfrak{s}$ has an odd-cardinality sibling cluster, then we have  $\mathfrak{t}^{\mathfrak{s}}_-(0)=0$.
    \end{enumerate}
    \begin{proof}
        It is immediate to see that, if $\mathfrak{s}$ has a child cluster $\mathfrak{c}$ of odd cardinality $2m+1$, then, letting $\alpha\in \mathfrak{c}$, any normalized reduction of $f^{\mathfrak{s}}_+$ has odd degree $2g+1-2m$ and thus, in particular, is not a square. This implies that $f^{\mathfrak{s}}_+=0^2+f^{\mathfrak{s}}_+$ is a good part-square decomposition (see \Cref{prop good decomposition}), and hence that $\mathfrak{t}^{\mathfrak{s}}_+(0)=0$. This proves (a); the proof of (b) is analogous.
    \end{proof}
\end{lemma}

The following proposition directly implies \Cref{thm introduction main}(d).

\begin{prop} \label{prop estimating threshold}
    Let $\mathfrak{s}$ be a cluster of even cardinality, and let $\mathfrak{s}'$ be its parent cluster.  The rational number $B_{f,\mathfrak{s}}$ given in \Cref{prop depth threshold} satisfies the below inequalities.
    \begin{enumerate}[(a)]
        \item We have $B_{f,\mathfrak{s}} \leq 4v(2)$.  In particular, if $\delta(\mathfrak{s}) \ge 4v(2)$, then there exists a valid disc linked to $\mathfrak{s}$, and if $\delta(\mathfrak{s}) > 4v(2)$, then it is guaranteed that there are exactly $2$ valid discs linked to $\mathfrak{s}$.
        \item If $\mathfrak{s}$ has a child cluster (resp.\ a sibling cluster) of odd cardinality, then we have $B_{f,\mathfrak{s}} \geq \frac{2v(2)}{|\mathfrak{s}| - 1}$ (resp.\ $B_{f,\mathfrak{s}} \geq \frac{2v(2)}{2g + 1 - |\mathfrak{s}|}$). 
        In particular, if $\mathfrak{s}$ is a cardinality-$2$ or a cardinality-$2g$ cluster, there cannot be a valid disc linked to $\mathfrak{s}$ if $\delta(\mathfrak{s})<2v(2)$.
        \item If $\mathfrak{s}$ has both a child and a sibling cluster of odd cardinality, then we have 
        $$B_{f,\mathfrak{s}} \geq \left( \frac{2}{|\mathfrak{s}| - 1} + \frac{2}{2g + 1 - |\mathfrak{s}|} \right)v(2).$$
    \end{enumerate}

    \begin{proof}
        As the continuous piecewise-linear functions $\mathfrak{t}^{\mathfrak{s}}_\pm$ have positive integer slopes until reaching an output of $2v(2)$ by \Cref{prop properties t plus minus}, we must have $b_0(\mathfrak{t}^{\mathfrak{s}}_\pm) \leq 2v(2)$.  This implies that $B_{f,\mathfrak{s}} = b_0(\mathfrak{t}^{\mathfrak{s}}_+) + b_0(\mathfrak{t}^{\mathfrak{s}}_-) \leq 4v(2)$, proving part (a).
    
        Now if $\mathfrak{s}$ has a child cluster (resp.\ a sibling cluster) of odd cardinality, then we have $\mathfrak{t}^{\mathfrak{s}}_+(0) = 0$ (resp.\  $\mathfrak{t}^{\mathfrak{s}}_-(0) = 0$) by \Cref{lemma tplus minus odd cardinality child}.  By \Cref{lemma estimates of b_0 from slopes}, we then have $b_0(\mathfrak{t}^{\mathfrak{s}}_+) \geq \frac{2v(2)}{|\mathfrak{s}| - 1}$ (resp.\ $b_0(\mathfrak{t}^{\mathfrak{s}}_-) \geq \frac{2v(2)}{2g + 1 - |\mathfrak{s}|}$).  This proves (b) and (c) if we recall that $B_{f,\mathfrak{s}} = b_0(\mathfrak{t}^{\mathfrak{s}}_+) + b_0(\mathfrak{t}^{\mathfrak{s}}_-)$.
    \end{proof}
    
\end{prop}

\begin{rmk} \label{rmk estimating threshold}

Recalling that the definition of \emph{cluster} includes singleton subsets of roots, we see that the hypothesis of (c) is really a very mild one: it is equivalent to the condition that $\mathfrak{s}$ is not a union of $\geq 2$ even-cardinality clusters and that neither is $\mathfrak{s}'$.

\end{rmk}

\section{The toric rank of the special fiber} \label{sec toric rank}

Our purpose in this section is to use the framework we developed in \S\ref{sec depths} to glean information about the components of the special fiber of the relatively stable model of the hyperelliptic curve $Y$, based on knowledge of its valid discs which contain clusters.  After introducing the notions of a \emph{(semi-)viable} cluster and an \emph{(semi-)\"{u}bereven} cluster, we will prove the following theorem that allows to compute the toric rank of $\SF{\Yrst}$, and hence the toric rank of the special fiber of any semistable model of $Y$ defined over any extension of $R$.
\begin{thm}
    \label{thm toric rank}
    The toric rank of $\SF{\Yrst}$ equals the number of viable clusters which are not \"{u}bereven.
\end{thm}

Let us begin by defining viable and semi-viable clusters.
\begin{dfn} \label{dfn viable}
    We say that a cluster $\mathfrak{s}$ is \emph{viable} (resp. \emph{semi-viable}) if there exist $2$ distinct valid discs (resp. $\geq 1$ valid disc) linked to $\mathfrak{s}$.
\end{dfn}

\begin{rmk} \label{rmk viable clusters are even}
    We observe from \Cref{rmk structure of J} that semi-viable (and thus also viable) clusters have even cardinality.
\end{rmk}

\begin{prop} \label{prop viable correspondence}
    Viable clusters are in one-to-one correspondence with the nodes of $\SF{\Xrst}$ over which the cover $\SF{\Yrst} \to \SF{\Xrst}$ is unramified (i.e.\ the nodes of  $\SF{\Xrst}$ that lie under $2$ distinct nodes in $\SF{\Yrst}$).  More precisely, the correspondence is given by sending a viable cluster $\mathfrak{s}$ (to which valid discs $D_+ \subsetneq D_-$ are linked) to the node of $\SF{\Xrst}$ coming from the point in $\SF{\XX_{D_+}}$ (resp. $\SF{\XX_{D_-}}$) at which the set $\Rinfty \smallsetminus \mathfrak{s}$ (resp. the cluster $\mathfrak{s}$) reduces.
    \begin{proof}
        Suppose that $\mathfrak{s}$ is a viable cluster, and let $D_+\subsetneq D_-$ be the $2$ valid discs linked to it. It follows from \Cref{lemma not a valid disc} that we have $\XX_D\not\leq \Xrst$ for all discs $D$ satisfying $D_+\subsetneq D\subsetneq D_-$; hence (as one can deduce by applying \Cref{prop relative position smooth models line}) the corresponding lines (coming from $\SF{\XX_{D_+}}$ and $\SF{\XX_{D_-}}$) of $\SF{\Xrst}$ intersect at a node $P\in \SF{\Xrst}$.  Using \Cref{prop relative position smooth models line}(a)(b) and \Cref{lemma discs linked to clusters}, it is easy to verify the claim about the subsets of $\Rinfty$ which reduce to the corresponding point of $\SF{\XX_{D_+}}$ and that of $\SF{\XX_{D_-}}$.  We know moreover by \Cref{prop lambda plus minus and ell}(a) that the map $\SF{\Yrst}\to \SF{\Xrst}$ is unramified above $P$.
    
        Let us now prove the converse implication. Let $P$ be a node of $\SF{\Xrst}$ above which $\SF{\Yrst}$ is unramified; let $L_-$ and $L_+$ the two lines of $\SF{\Xrst}$ passing through $P$; and let $D_\pm$ be the corresponding discs. Since the cover $\SF{\Yrst} \to \SF{\Xrst}$ is unramified above $P$, no element of $\Rinfty$ reduces to $P\in \SF{\Xrst}$, which is equivalent to saying that no element of $\Rinfty$ reduces to the unique node of $\SF{\XX_{\{D_+,D_-\}}}$. In particular, $\infty$ lies on one and only one of the two lines $L_+$ and $L_-$ comprising the special fiber $\SF{\XX_{\{D_+,D_-\}}}$, say $\infty\in L_-\setminus L_+$; this implies, in particular, that we have $D_+\subsetneq D_-$ by \Cref{prop relative position smooth models line}. We can now write the decomposition $\mathcal{R}=\mathfrak{s}\sqcup (\mathcal{R}\setminus \mathfrak{s})$, where $\mathfrak{s}$ (resp.\ $\RR \smallsetminus \mathfrak{s}$) consists of the roots whose reductions in $\SF{\XX_{\{D_+,D_-\}}}$ lie on $L_+\smallsetminus L_-$ (resp.\ $L_-\smallsetminus L_+$). It is now clear that $\mathfrak{s}=D_+\cap \mathcal{R}$ is a cluster, to which the distinct valid discs $D_-$ and $D_+$ are linked.
    \end{proof}
\end{prop}

\begin{prop} \label{prop viable thicknesses}
    If $\mathfrak{s}$ is a viable cluster corresponding to a node $P \in \SF{\Xrst}$ as in \Cref{prop viable correspondence}, then the thickness of each of the $2$ nodes lying above $P$ is equal to $(b_+(\mathfrak{s}) - b_-(\mathfrak{s})) / v(\pi)$ (with the notation of \S\ref{sec depths construction valid discs}).
\end{prop}

\begin{proof}
    This is straightforward from applying Propositions \ref{prop thickness} and \ref{prop vanishing persistent}(b) to \Cref{thm summary depths valid discs}.
\end{proof}

We now give the other main definition of this section.

\begin{dfn} \label{dfn ubereven}
    We say that a cluster $\mathfrak{s}$ is \emph{\"{u}bereven} (resp. \emph{semi-\"{u}bereven}) if it is viable (resp. semi-viable) and if all of its children clusters are also viable (resp. semi-viable).
\end{dfn}

\begin{rmk}
    We have observed in \Cref{rmk viable clusters are even} that viable clusters have even cardinality; in fact, in the setting of residue characteristic $\neq 2$, the analog of a viable cluster for satisfying \Cref{prop viable correspondence} (and \Cref{prop viable thicknesses}) is simply an even-cardinality cluster.  It follows that the analog of our definition of an \emph{\"{u}bereven cluster} in that setting is the used in \cite{dokchitser2022arithmetic}, which is simply an even-cardinality cluster whose children all have even cardinality.
\end{rmk}

\begin{lemma}
    \label{lemma all but one}
    Let $\mathfrak{s}$ be a cluster, and let $\mathfrak{c}_1, \ldots, \mathfrak{c}_N$ be its children.  If each child $\mathfrak{c}_i$ is semi-viable, then the cluster $\mathfrak{s}$ is semi-\"{u}bereven and satisfies  $b_+(\mathfrak{s})=d_+(\mathfrak{s})$ (with the notation of \S\ref{sec depths construction valid discs}).
    \begin{proof}
        Since $\mathfrak{c}_i$ is semi-viable, by \Cref{prop depth threshold} we have $\delta(\mathfrak{c}_i) \geq B_{f,\mathfrak{s}}\ge b_0(\mathfrak{t}_+^{\mathfrak{c}_i})$, which implies that $\mathfrak{t}_+^{\mathfrak{c}_i}(\delta(\mathfrak{c}_i)) = 2v(2)$.  Now \Cref{prop deep ubereven cluster} says that we have $\mathfrak{t}_+^{\mathfrak{s}}(0) = 2v(2)$, which directly implies that $b_+(\mathfrak{s}) = d_+(\mathfrak{s})$; the disc $D_{\mathfrak{s}, d_+(\mathfrak{s})}$ is therefore valid, and thus the cluster $\mathfrak{s}$ is semi-\"{u}bereven.
    \end{proof}
\end{lemma}

\begin{prop} \label{prop ubereven correspondence}
    The assignement $\mathfrak{s}\mapsto D_{\mathfrak{s},d_+(\mathfrak{s})}$ induces a one-to-one correspondence between the \"{u}bereven clusters and the valid discs $D$ such that the special fiber $\SF{\YY_D}$ is reducible (i.e.\ $\SF{\YY_D}$ consists of $2$ rational components).
    \begin{proof}
        Suppose first that $D$ is a valid disc such that $\SF{\YY_D}$ is reducible. Then, we know by \Cref{thm part of rst separable} that the elements of $\mathcal{R}\cup\lbrace \infty\rbrace$ reduce to $N \ge 3$ distinct points of $\SF{\XX_D}$; by \Cref{lemma discs linked to clusters} we consequently have $D=D_{\mathfrak{s}, d_+(\mathfrak{s})}$ for some cluster $\mathfrak{s}$ and that the $N$ points are the reductions $\overline{\mathfrak{c}_1}, \ldots, \overline{\mathfrak{c}_{N-1}}, \overline{\mathcal{R}\setminus \mathfrak{s}} = \infty$, where $\mathfrak{c}_1, \ldots, \mathfrak{c}_{N-1}$ are the child clusters of $\mathfrak{s}$. From the fact that $D$ is a valid disc we deduce that $d_+(\mathfrak{s})=d_-(\mathfrak{c}_i)$ is a common endpoint of the intervals $J(\mathfrak{s})$ and $J(\mathfrak{c}_i)$ for all $i$; in particular, these intervals are non-empty and we have $b_+(\mathfrak{s})=d_+(\mathfrak{s})=d_-(\mathfrak{c}_i)=b_-(\mathfrak{c}_i)$. The fact that $\SF{\YY_D}$ consists of $2$ components means that we have $\ell(\XX_D,\infty)=0$ and $\ell(\XX_D,\overline{\mathfrak{c}_i})=0$ for all $i$, but, according to \Cref{prop lambda plus minus and ell}(c),(d), this implies that $b_-(\mathfrak{c}_i)<b_+(\mathfrak{c}_i)$ for all $i$ as well as $b_-(\mathfrak{s})<b_+(\mathfrak{s})$.  We conclude that $\mathfrak{s}$ is a \"{u}bereven cluster.
        
        Let us now prove the converse implication. Assume that $\mathfrak{s}$ is a \"{u}bereven cluster, and let us denote by $\mathfrak{c}_1, \ldots, \mathfrak{c}_{N-1}$ its children (with $N\ge 3$); note that we have $d_+(\mathfrak{s})=d_-(\mathfrak{c}_i)$ for all $i$. Letting $D=D_{\mathfrak{s}, d_+({\mathfrak{s}})}$, we have that the elements of $\Rinfty$ reduce in $\SF{\XX_D}$ to the $N$ distinct points, $\overline{\mathfrak{c}_1}, \ldots, \overline{\mathfrak{c}_{N-1}}, \overline{\mathcal{R}\setminus \mathfrak{s}} = \infty$. Now, since all of the $\mathfrak{c}_i$'s are viable, we have $b_+(\mathfrak{s})=d_+(\mathfrak{s})$ by \Cref{lemma all but one}; meanwhile, since $\mathfrak{s}$ is also assumed to be viable, we have $b_-(\mathfrak{s})<b_+(\mathfrak{s})$. We conclude, in particular, that the disc $D$ is valid (see \Cref{thm summary depths valid discs}).
        
        As before, from the fact that $D=D_{\mathfrak{s},d_+(\mathfrak{s})}$ is a valid disc, we deduce that $b_+(\mathfrak{s})=d_+(\mathfrak{s})=b_-(\mathfrak{c}_i)=d_-(\mathfrak{c}_i)$. Since both $\mathfrak{s}$ and the $\mathfrak{c}_i$'s are viable, they have even cardinality and satisfy $b_-(\mathfrak{c}_i)<b_+(\mathfrak{c}_i)$ and $b_-(\mathfrak{s})<b_+(\mathfrak{s})$; hence, by \Cref{prop lambda plus minus and ell}(a) we have $\ell(\XX_D,\overline{\mathfrak{c}_i})=0$ for all $i$, and $\ell(\XX_D,\infty)=0$. But this means that $\NSF{\YY_D}\to \SF{\XX_D}$ is an étale double cover of the line, i.e.\ that $\SF{\YY_D}$ is reducible, as discussed in \S\ref{sec models hyperelliptic separable}.
    \end{proof}
\end{prop}

We are now able to prove the main result of this section.
\begin{proof}[Proof (of \Cref{thm toric rank})]
    The toric rank of a semistable $k$-curve is just the number of nodes (which we denote by $\Nnodes$) minus the number of irreducible components (which we denote by $\Nirreducible$) plus 1 (see \S\ref{sec semistable preliminaries semistable models}). Now, since we have $g(X)=0$, the toric rank of $\SF{\Xrst}$ is $0$ (see the discussion in \S\ref{sec models hyperelliptic line}). Hence, the toric rank of $\SF{\Yrst}$ can be computed as
    \begin{equation*}
       [\Nnodes(\SF{\Yrst})-\Nnodes(\SF{\Xrst})]-[\Nirreducible(\SF{\Yrst})-\Nirreducible(\SF{\Xrst})].
    \end{equation*}
    Now \Cref{prop viable correspondence} implies that the first difference equals the number of viable clusters; meanwhile, \Cref{prop ubereven correspondence} implies that the second difference equals the number of \"{u}bereven clusters.
\end{proof}

\section{The $2$-rank of the special fiber} \label{sec 2-rank}

We now further use the framework established in \S\ref{sec depths} to compute the $2$-rank of the special fiber of a semistable model of a hyperelliptic curve over $K$.  By the $2$-rank of the special fiber $\SF{\Yrst}$ of the relatively stable model $\Yrst$ of $Y$, we mean the rank of the $2$-torsion subgroup $\Pic^0(\SF{\Yrst})[2]$ of the (generalized) Jacobian $\Pic^0(\SF{\Yrst}) / k$, which is well known to be an elementary abelian group.  In fact, as discussed in \S\ref{sec semistable preliminaries semistable models}, the group variety $\Pic^0(\SF{\Yrst})$ is an extension of an abelian variety $A / k$ by a torus $T / k$; since a torus over the characteristic-$2$ field $k$ has trivial $2$-torsion, the rank of $\Pic^0(\SF{\Yrst})[2]$ as a $\zz/2\zz$-module must equal that of $A[2]$, which in turn is non-negative and at most the abelian rank (or $g$ minus the toric rank) of $\SF{\Yrst}$.  The purpose of this section is to prove the following result (whose statement is visibly analogous to that of \Cref{thm toric rank}).

\begin{thm} \label{thm 2-rank}

The $2$-rank of $\SF{\Yrst}$ equals the number of semi-viable clusters which are not semi-\"{u}bereven, minus the number of viable clusters which are not \"{u}bereven (or equivalently, minus the toric rank, thanks to \Cref{thm toric rank}).

\end{thm}

Write $D_1, \dots, D_r$ for the set of valid discs.  For $1 \leq i \leq r$, let $\XX_i$ denote the corresponding smooth model of $X$; let $\YY_i$ denote the normalization of $\XX_i$ as in \S\ref{sec models hyperelliptic forming}; and let $M_i$ be the number of branch points of the quadratic cover $\NSF{\YY_i} \to \SF{\XX_i}$ (with notation as in \S\ref{sec models hyperelliptic separable}).  We set out to prove \Cref{thm 2-rank}, beginning by establishing a couple of lemmas.

\begin{lemma} \label{lemma deuring-shafarevich}
    With notation as above, the $2$-rank of $\SF{\Yrst}$ is given by $\sum_{i = 1}^r \max\{M_i - 1, 0\}$.
    \begin{proof}
        We observe from \cite[Example 9.2.8]{bosch2012neron} that the Jacobian $\Pic^0(\SF{\Xrst})$ is realized as an extension of an abelian group $A$ by a torus $T$ in the short exact sequence 
        \begin{equation} \label{eq ses from bosch}
            1 \to T \to \Pic^0(\SF{\Yrst}) \to \prod_{i = 1}^r \Pic^0(\NSF{\YY_i}) \to 1,
        \end{equation}
        where the second map to the right is induced by the morphisms $\NSF{\YY_i} \to \SF{\Yrst}$ (see \Cref{thm part of rst separable}); in other words, we may identify $A$ with the product of the Jacobians of each curve $\NSF{\YY_i} / k$.  (To fully verify that this follows from what is given in the above citation, we recall firstly that if $D$ is a non-valid disc such that $\XX_D \leq \Xrst$, then the corresponding inseparable component of $\SF{\Yrst}$ is a line and has trivial Jacobian by \Cref{rmk inseparable is inessential}, and secondly that each curve $\NSF{\YY_i}$ is the normalization of an irreducible component of $\SF{\Yrst}$, except in the case that $\NSF{\YY_i}$ is the disjoint union of $2$ lines and thus again has trivial Jacobian.)  Now the torus $T / k$ has trivial $2$-torsion, and we may therefore identify the $2$-torsion subgroup of $\Pic^0(\SF{\Xrst})$ with that of the product in (\ref{eq ses from bosch}), which in turn may be identified with the direct sum of the $2$-torsion subgroups of each $\Pic^0(\NSF{\YY_i})$.  The $2$-rank of $\SF{\Yrst}$ is therefore equal to the sum of the $2$-ranks of the curves $\NSF{\YY_i}$.

        Fix an index $i$.  If $M_i = 0$, then the curve $\NSF{\Yrst}$ consists of $2$ disjoint lines and therefore contributes $0$ to the sum.  If $M_i \geq 1$, we may apply the Deuring-Shafarevich formula given as \cite[Theorem 4.1 bis]{subrao1975p} (see also \cite[Theorem 1.1]{shiomi2011deuring}) to the $2$-cover $\NSF{\YY_i} \to \SF{\XX_i}$ to get that the $2$-rank of $\NSF{\YY_i}$ equals $M_i - 1$.  The desired formula for the $2$-rank of $\SF{\Yrst}$ follows.
    \end{proof}
\end{lemma}

Now let $D$ be a valid disc, and let $\XX_D$ be the corresponding smooth model of $X$ as in \S\ref{sec models hyperelliptic line}.  Let $\mathfrak{c}_1, \dots, \mathfrak{c}_{N-1}, \mathfrak{c}_\infty \subseteq \Rinfty$ be the subsets defined as in \S\ref{sec clusters} whose reductions $\overline{\mathfrak{c}_1}, \dots, \overline{\mathfrak{c}_{N-1}}, \overline{\mathfrak{c}_\infty} = \infty$ coincide with the points of $\SF{\XX_D}$ at which the branch locus reduces.  According to \Cref{lemma discs linked to clusters}, the disc $D$ is linked to the clusters $\mathfrak{c}_1, \dots, \mathfrak{c}_{N-1}$ as well as to the cluster $\mathfrak{s} := \mathcal{R} \smallsetminus \mathfrak{c}_\infty = \mathfrak{c}_1 \sqcup \dots \sqcup \mathfrak{c}_{N-1}$.

\begin{lemma} \label{lemma branch points from non-viable}
    Assume the above notation.  The number of branch points of the map $\NSF{\YY_D} \to \SF{\XX_D}$ is computed as follows.
    \begin{enumerate}[(a)]
        \item If $N = 1$, then there is exactly $1$ branch point.
        \item If $N = 2$, then the number of branch points is $1$ (resp. $2$) if $\mathfrak{c}_1 = \mathfrak{s}$ is (resp. is not) viable.
        \item If $N \geq 3$, then the number of branch points equals the number of non-viable clusters among $\mathfrak{c}_1, \dots, \mathfrak{c}_{N-1}, \mathfrak{s}$. 
    \end{enumerate}

    \begin{proof}
        In the situation of part (a) or (b), \Cref{thm part of rst separable} implies that $\NSF{\YY_D}$ is irreducible so that there is at least $1$ branch point.  If $N = 1$, then all of the branch locus $\Rinfty$ reduces to the single point $P_1$ of $\SF{\XX_D}$, so this is the unique branch point and part (a) is proved.
        
        Now assume that $N \geq 2$.  Any branch point of the map $\NSF{\YY_D} \to \SF{\XX_D}$ must be a point at which one of the subsets $\mathfrak{c}_1, \dots, \mathfrak{c}_{N-1}, \mathfrak{c}_\infty$ reduces and thus lies in $\{\overline{\mathfrak{c}_1},\dots, \overline{\mathfrak{c}_{N-1}}, \overline{\mathfrak{c}_\infty} = \infty\}$.  Suppose that $N = 2$ and $\mathfrak{c}_1$ is viable.  Then there is another disc $D'$ linked to $\mathfrak{c}_1$.  If $D' \subsetneq D$ (resp. $D \subsetneq D'$), then \Cref{prop viable correspondence} says that the component $\SF{\XX_{D'}}$ of $\SF{\Xrst}$ meets $\SF{\XX_D}$ at the point $\overline{\mathfrak{c}_1}$ (resp. $\infty$) and that the map is unramified above that point, so the other point $\infty$ (resp. $\overline{\mathfrak{c}_1}$) must be a branch point.  We therefore get that there is exactly $1$ branch point.  Suppose instead that $N = 2$ and that $\mathfrak{c}_1$ is not viable.  Then \Cref{prop viable correspondence} implies that the map ramifies over both points $\overline{\mathfrak{c}_1}$ and $\infty$, and so there are exactly $2$ branch points.  This proves part (b).
        Meanwhile, if $N \geq 3$,  \Cref{prop viable correspondence} implies that if $\mathfrak{s}$ (resp. $\mathfrak{c}_i$ for some index $i$) is viable, then the map $\NSF{\Yrst} \to \SF{\Xrst}$ is not branched over $\infty$ (resp. $\overline{\mathfrak{c}_i}$), and that the converse holds.
        Part (c) follows.
    \end{proof}
\end{lemma}

\begin{proof}[Proof (of \Cref{thm 2-rank})]
For $1 \leq i \leq r$, we retain the above notation for the disc $D_i \subset K^\alg$ and the number $M_i$ of branch points; we also let $N_i$ be the number of points of $\SF{\XX_i}$ at which the branch locus $\Rinfty$ reduces.  \Cref{lemma deuring-shafarevich} provides a formula for the $2$-rank of $\NSF{\Yrst}$ in terms of the integers $M_i$, so in order to prove the theorem, we need to compute these values $M_i$ in terms of the viability properties of various clusters.
    
Fix an index $i$.  If $N_i = 1$, then $D_i$ is linked to no cluster; we have $M_i = 1$ by \Cref{lemma branch points from non-viable}(a); and the contribution to the above sum is $M_i - 1 = 0$.  If $N_i = 2$, then $D_i$ is linked to $1$ cluster; \Cref{lemma branch points from non-viable}(b) says that if that cluster is semi-viable but not viable, we get $M_i = 2$ and the contribution to the sum is $M_i - 1 = 1$, but that otherwise, we get $M_i = 1$ and the contribution to the sum is $M_i - 1 = 0$.  Suppose now that $N_i \geq 3$.  By \Cref{lemma discs linked to clusters}(c), the disc $D_i$ is linked to the clusters $\mathfrak{c}_1, \dots, \mathfrak{c}_{N-1}, \mathfrak{s}$; these clusters are all therefore semi-viable.  Then \Cref{lemma branch points from non-viable}(c) implies that $M_i$ equals the number of clusters among them which are semi-viable but not viable.  The contribution to the sum is $0$ if $M_i = 0$ and is otherwise $1$ less than the number of clusters linked to $D_i$ which are semi-viable but not viable.
    
Putting all of this together (and taking into account that a cluster which is semi-viable but not viable is linked to exactly $1$ valid disc and therefore will not be counted multiple times), we get that the $2$-rank is equal to the number of clusters which are semi-viable but not viable minus the number of indices $i$ such that $N_i \geq 3$ and $M_i \geq 1$.  It therefore suffices to show that the assignment $\mathfrak{s} \mapsto D_{\alpha, d_+(\mathfrak{s})}$, where $\alpha$ is chosen to be any root in $\mathfrak{s}$, induces a one-to-one correspondence between clusters which are semi-\"{u}bereven but not \"{u}bereven and indices $i$ satisfying $N_i \geq 3$ and $M_i \geq 1$.  In fact, by \Cref{lemma discs linked to clusters}, the condition that $N_i \geq 3$ implies that $D_i$ is linked to a cluster $\mathfrak{s}$ along with its children $\mathfrak{c}_1, \dots, \mathfrak{c}_{N_i-1}$, and by definition, the cluster $\mathfrak{s}$ is semi-\"{u}bereven.  Conversely, given a semi-\"{u}bereven cluster $\mathfrak{s}$, we have $b_+(\mathfrak{s}) = d_+(\mathfrak{s})$ by \Cref{lemma all but one}, and so the disc $D_{\alpha, d_+(\mathfrak{s})}$ is valid.  Meanwhile, the condition that $M_i \neq 0$ is equivalent to the associated cluster $\mathfrak{s}$ not being \"{u}bereven by \Cref{prop ubereven correspondence}.  The desired correspondence follows.
    \end{proof}

\section{Examples and special cases} \label{sec examples}

In this section, we show how our results may be used to find the toric rank and the $2$-rank (and in some cases the full structure) of the special fiber of the relatively stable model of some particular families of hyperelliptic curves.

\begin{ex} \label{example1}
Letting $K$ be the maximal unramified extension of $\qq_2$, consider the family of hyperelliptic curves defined over $K$ by the equation 
    \begin{equation} \label{eq example1}
        Y_\lambda : y^2 = f_\lambda(x) := x(x - \lambda)(x^3 + x^2 - 4x + 1),
    \end{equation}
where $\lambda \in K \smallsetminus \{z \in R \ | \ z(z^3 + z^2 - 4z + 1) = 0\}$ is a parameter.  One sees easily (for instance from verifying that the polynomial factor $x^3 + x^2 - 4x + 1$ is irreducible modulo $2$) that the only possible even-cardinality cluster associated to $f_{g,\lambda}$ is $\mathfrak{s} := \{0, \lambda\}$, and that this subset is a cluster if and only if we have $v(\lambda) > 0$.  Let us fix $\alpha := 0 \in \mathfrak{s}$ as the center we will use in our calculations.  Using techniques from \S\ref{sec depths separating},\ref{sec depths threshold}, we compute the ``threshold depth" $B_{f_\lambda, \mathfrak{s}}$ as follows.

We first note that we have $\mathfrak{t}_\pm^{\mathfrak{s}}(0) = 0$ by \Cref{lemma tplus minus odd cardinality child}.  Now \Cref{prop properties t plus minus} implies that the function $\mathfrak{t}_+^{\mathfrak{s}}$ is simply a linear function of slope $1$, and so in particular we have $\lambda_+(\mathfrak{s}) = 1$.  \Cref{lemma estimates of b_0 from slopes} now gives us $b_0(\mathfrak{t}_+^{\mathfrak{s}})) = \frac{2v(2) - 0}{2 - 1} = 2v(2)$.

Meanwhile, choosing the scalar $\beta_-$ in (\ref{eq standard form}) to be $1$, one readily computes the formula $f_-^{\mathfrak{s}, \alpha} = f^{\mathcal{R} \smallsetminus \mathfrak{s}}(x) = x^3 + x^2 - 4x + 1$.  Then we may explicitly describe the function $\mathfrak{t}_-$ using a totally odd decomposition of $f_-^{\mathfrak{s}, \alpha}$.  A totally odd decomposition in this case (as for any cubic polynomial) is easy to come by: we may take $f_-^{\mathfrak{s}} = [q_-]^2 + \rho_- = [x + 1]^2 + [x^3 - 6x]$.  From our expression for $\rho_-$, we see that $\mathfrak{t}_-^{\mathfrak{s}}$ acts as $b \mapsto 3b$ for $b \in [0, \frac{1}{2}v(2)]$ and as $b \mapsto b + v(2)$ for $b \in [\frac{1}{2}, \infty)$.  It follows that we have $b_0(\mathfrak{t}_-^{\mathfrak{s}}) = v(2)$.

We therefore get $B_{f_\lambda, \mathfrak{s}} = b_0(\mathfrak{t}_+^{\mathfrak{s}}) + b_0(\mathfrak{t}_-^{\mathfrak{s}}) = 3v(2)$.  It now follows from \Cref{thm introduction main}(c) that there are $2$ (resp. $1$; resp. $0$) valid discs linked to $\mathfrak{s}$ if we have $v(\lambda) > 3v(2)$ (resp. $v(\lambda) = 3v(2)$; resp. $v(\lambda) < 3v(2)$); that in the first case, since the cluster $\mathfrak{s}$ is not \"{u}bereven, the $2$ guaranteed valid discs each give rise to a component of $\SF{\Yrst}$; and that these $2$ components meet at $2$ nodes each with thickness equal to $(v(\lambda) - 3v(2)) / v(\pi)$, where $\pi$ is a uniformizer of the extension $K' / K$ over which relatively stable reduction is obtained.  Since there are no other even-cardinality clusters, it follows from \Cref{thm toric rank} that the toric rank of the special fiber $\SF{\Yrst}$ equals $1$ (resp. $0$) if we are in the first case (resp. otherwise).

Applying \Cref{thm 2-rank}, we get that the $2$-rank of $\SF{\Yrst}$ is $1$ if we have $v(\lambda) = 3v(2)$ and is $0$ otherwise.
    
Supposing now that we have $v(\lambda) > 3v(2)$, we get the following further results.  \Cref{thm summary depths valid discs} combined with \Cref{prop formulas for b_pm} tells us in this case that the $2$ valid discs linked to $\mathfrak{s}$ are given by $D_{0, v(\lambda) - 2v(2)}$ and $D_{0, v(2)}$.  A part-square decomposition for $f_\lambda$ that is good at these discs can be obtained using the set-up and statement of \Cref{prop product part-square}(b).  Using this, one can now compute, using \Cref{prop riemann hurwitz} and \Cref{lemma computation ell ramification}, that the abelian rank of the components of $\SF{\Yrst}$ corresponding to both discs is $0$.  We know from the toric rank that that the abelian rank of the special fiber $\SF{\Yrst}$ is $2 - 1 = 1$, so there must be a further component of $\SF{\Yrst}$.  It is not too difficult to see, via \Cref{dfn relatively stable} and \Cref{prop viable correspondence}, that in fact there is only $1$ further component, which has genus $1$.  In this case, the fiber $\SF{\Yrst}$ is as in \Cref{fig example1} above.  This falls under Case B6 of the classification in \cite[\S5.2]{gehrunger2025reduction}.
\end{ex}

\begin{figure}

\includegraphics[scale=.4]{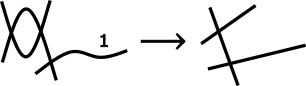}

\caption{The map $\Yrst \to \Xrst$ of special fibers in \Cref{example1}.}
\label{fig example1}
\end{figure}

\begin{ex} \label{example2}
For any $g \geq 1$ and over any field $K$ satisfying the hypotheses generally assumed in this paper, consider the family of genus-$g$ hyperelliptic curves over $K$ defined by the equation 
    \begin{equation}
        Y_{g,\lambda} : y^2 = f_{g,\lambda}(x) := x(x - \lambda)(x^{2g-1} - 1),
    \end{equation}
where $\lambda \in K \smallsetminus \{z \in R \ | \ z^{2g} - z = 0\}$ is a parameter.  (Note that in the special case of $g = 1$, this is just the Legendre family of elliptic curves.)  As the set of roots of $f_{g,\lambda}$ consists of $0$, $\lambda$, and all of the $(2g - 1)$th roots of unity, it is clear that, as in \Cref{example1}, the only possible even-cardinality cluster associated to $f_{g,\lambda}$ is $\mathfrak{s} := \{0, \lambda\}$, and that this subset is a cluster if and only if we have $v(\lambda) > 0$.  Let us again fix $\alpha := 0 \in \mathfrak{s}$ as the center we will use in our calculations.  Using techniques from \S\ref{sec depths separating},\ref{sec depths threshold}, we compute the ``threshold depth" $B_{f_{g,\lambda}, \mathfrak{s}}$ as follows.

We first of all have $b_0(\mathfrak{t}_+^{\mathfrak{s}})) = 2v(2)$ by the exact same argument as in \Cref{example1}.  Meanwhile, choosing the scalar $\beta_-$ in (\ref{eq standard form}) to be $1$, one readily computes the formula $f_-^{\mathfrak{s}, \alpha} = f^{\mathcal{R} \smallsetminus \mathfrak{s}}(x) = x^{2g-1} - 1$.  Then we may explicitly describe the function $\mathfrak{t}_-$ using a totally odd decomposition of $f_-^{\mathfrak{s}, \alpha}$.  A totally odd decomposition in this case is immediately visible: we may take $f_-^{\mathfrak{s}} = [q_-]^2 + \rho_- = [\sqrt{-1}]^2 + x^{2g-1}$.  Now it is clear that $\mathfrak{t}_-^{\mathfrak{s}}$ is simply the linear function $b \mapsto (2g - 1)b$.  It follows that we have $b_0(\mathfrak{t}_-^{\mathfrak{s}}) = \frac{2}{2g - 1}v(2)$.

We therefore get $B_{f_{g,\lambda}, \mathfrak{s}} = b_0(\mathfrak{t}_+^{\mathfrak{s}}) + b_0(\mathfrak{t}_-^{\mathfrak{s}}) = \frac{4g}{2g - 1}v(2)$; we have thus found a class of examples in which the ``threshold relative depth" is equal to the lower bound guaranteed by \Cref{thm introduction main}(d).  It now follows from \Cref{thm introduction main}(c) that there are $2$ (resp. $1$; resp. $0$) valid discs linked to $\mathfrak{s}$ if we have $v(\lambda) > \frac{4g}{2g - 1}v(2)$ (resp. $v(\lambda) = \frac{4g}{2g - 1}v(2)$; resp. $v(\lambda) < \frac{4g}{2g - 1}v(2)$); that in the first case, since the cluster $\mathfrak{s}$ is not \"{u}bereven, the $2$ guaranteed valid discs each give rise to a component of $\SF{\Yrst}$; and that these $2$ components meet at $2$ nodes each with thickness equal to $(v(\lambda) - \frac{4g}{2g - 1}v(2)) / v(\pi)$, where $\pi$ is a uniformizer of the extension $K' / K$ over which relatively stable reduction is obtained.  Since there are no other even-cardinality clusters, it follows from \Cref{thm toric rank} that the toric rank of the special fiber $\SF{\Yrst}$ equals $1$ (resp. $0$) if we are in the first case (resp. otherwise).

Applying \Cref{thm 2-rank}, we get that the $2$-rank of $\SF{\Yrst}$ is $1$ if we have $v(\lambda) = \frac{4g}{2g - 1}v(2)$ and is $0$ otherwise.

Supposing now that we have $v(\lambda) > \frac{4g}{2g - 1}v(2)$, we get the following further results.  \Cref{thm summary depths valid discs} combined with \Cref{prop formulas for b_pm} tells us in this case that the $2$ valid discs linked to $\mathfrak{s}$ are given by $D_+ := D_{0, v(\lambda) - 2v(2)}$ and $D_- := D_{0, \frac{4g}{2g - 1}v(2)}$.  A similar sort of computation as mentioned in \Cref{example1} shows that the abelian ranks of the components of $\SF{\Yrst}$ corresponding to the discs $D_+$ and $D_-$ are $0$ and $g - 1$ respectively.  We know from the toric rank that that the abelian rank of the special fiber $\SF{\Yrst}$ is $g - 1$, so any further components of $\SF{\Yrst}$ must be a line.  It is not too difficult to see, via \Cref{dfn relatively stable} and \Cref{prop viable correspondence}, that no such component can exist.  We have therefore fully constructed the relatively stable model $\Yrst$ and even find that it is defined over any extension of $K$ containing $\sqrt{-1}$ and an element of valuation equal to $\frac{1}{2g - 1}v(2)$.  In this case, the fiber $\SF{\Yrst}$ is as in \Cref{fig example2} above.  If $g = 2$, this falls under Case B7 of the classification in \cite[\S5.2]{gehrunger2025reduction}.
\end{ex}

\begin{figure}

\includegraphics[scale=.4]{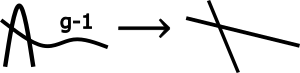}

\caption{The map $\Yrst \to \Xrst$ of special fibers in \Cref{example2}.}
\label{fig example2}
\end{figure}

\begin{ex} \label{example3}
The set-up of this example gives essentially the situation treated in \cite{dokchitser2023note}, and our computations in this example produce a variation on the results of that paper.
      
For any genus $g \geq 1$ and over any complete field $K$ satisfying the hypotheses generally assumed in this paper, let $Y : y^2 = f(x)$ be any hyperelliptic curve such that the $(2g + 1)$-element set of roots of $f$ is partitioned into cardinality-$2$ clusters $\mathfrak{s}_1, \dots, \mathfrak{s}_g$ and a singleton $\{a\}$, and such that there are no other even-cardinality clusters associated to $f$.  Choose centers $\alpha_i \in \mathfrak{s}$ for $1 \leq i \leq g$.  Using techniques from \S\ref{sec depths separating},\ref{sec depths threshold}, we compute the ``threshold depth" of each $B_{f, \mathfrak{s}_i}$ as follows.

For simplicity, suppose that $0 \in \mathfrak{s}_1$ and that we have chosen the center $0 =: \alpha_1$.  Exactly as in Examples \ref{example1} and \ref{example2}, we have $b_0(\mathfrak{t}_+^{\mathfrak{s}_1}) = 2v(2)$.  In order to compute $b_0(\mathfrak{t}_-^{\mathfrak{s}_1})$, choosing the scalar $\beta_-$ in (\ref{eq standard form}) to be $a$, we first observe that the polynomial $f_-^{\mathfrak{s}_1}(z)$ is the product of $1 - z$ and the square of a polynomial in $R[z]$.  Since the lowest-valuation terms of a square polynomial are its even-degree terms, one sees that the linear term of any part-square decomposition of $f_-^{\mathfrak{s}_1}(z)$ has valuation $0$.  It follows that $t_-^{\mathfrak{s}_1}$ is simply the linear function $b \mapsto b$, and therefore, we have $b_0(\mathfrak{t}_-^{\mathfrak{s}_1}) = 2v(2)$.

We therefore get $B_{f, \mathfrak{s}_1} = b_0(\mathfrak{t}_+^{\mathfrak{s}}) + b_0(\mathfrak{t}_-^{\mathfrak{s}}) = 4v(2)$ and indeed $B_{f, \mathfrak{s}_i} = 4v(2)$ for all $i$; we have thus found a class of examples in which the ``threshold relative depth" is equal to the upper bound guaranteed by \Cref{thm introduction main}(d).  It now follows from \Cref{thm introduction main}(c) that there are $2$ (resp. $1$; resp. $0$) valid discs linked to each $\mathfrak{s}_i$ if we have $\delta(\mathfrak{s}_i) > 4v(2)$ (resp. $v(\mathfrak{s}_i) = 4v(2)$; resp. $v(\mathfrak{s}_i) < 4v(2)$); that in the first case, since the cluster $\mathfrak{s}_i$ is not \"{u}bereven, the $2$ guaranteed valid discs each give rise to a component of $\SF{\Yrst}$; and that these $2$ components meet at $2$ nodes each with thickness equal to $(\delta(\mathfrak{s}_i) - 4v(2)) / v(\pi)$, where $\pi$ is a uniformizer of the extension $K' / K$ over which relatively stable reduction is obtained.

In the special case that we have $\delta(\mathfrak{s}_i) \geq 4v(2)$ for all $i$, all of the clusters $\mathfrak{s}_i$ are semi-viable, and \Cref{thm 2-rank} can be applied to show that the $2$-rank of $\SF{\Yrst}$ equals the number of indices $i$ such that $\delta(\mathfrak{s}_i) = 4v(2)$.  If we further assume that $\delta(\mathfrak{s}_i) > 4v(2)$ for all $i$, then all of the clusters $\mathfrak{s}_i$ are viable, so that the $2$-rank is $0$ but the special fiber $\SF{\Yrst}$ has the maximal possible toric rank $g$, and the $2$ valid discs linked to each $\mathfrak{s}_i$ are given by $D_{i,+} := D_{0, d_+(\mathfrak{s}_i) - 2v(2)}$ and $D_{i,-} := D_{0, d_-(\mathfrak{s}_i) + 2v(2)}$.  Since the toric rank is $g$, we know that the abelian rank is $0$, and so the abelian rank of each component of $\SF{\Yrst}$ is $0$.  In fact, one can show that if $g \geq 3$, the structure of $\SF{\Yrst}$ consists of $1$ inseparable component (of abelian rank $0$) intersecting the $g$ lines corresponding to the valid discs $D_{i,+}$, each of which intersects at $2$ points the line corresponding to the valid disc $D_{i,-}$, as in \Cref{fig example3} above.  If $g \in \{1, 2\}$, on the other hand, the only components of $\SF{\Yrst}$ are those corresopnding to the valid discs $D_{i,\pm}$, and when $g = 2$, the components corresponding to $D_{1,-}$ and $D_{2,-}$ intersect at $1$ point and this falls under Case C6 of the classification in \cite[\S5.2]{gehrunger2025reduction}.
\end{ex}

\begin{figure}

\begin{subfigure}[b]{.4\textwidth}
\centering
\includegraphics[scale=.4]{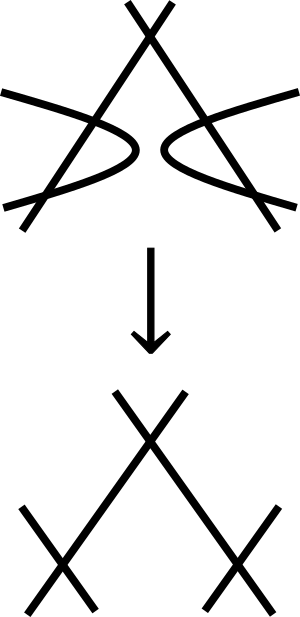}
\end{subfigure}
~
\begin{subfigure}[b]{.4\textwidth}
\centering
\includegraphics[scale=.4]{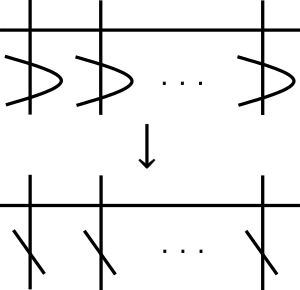}
\end{subfigure}

\caption{The map $\Yrst \to \Xrst$ of special fibers in \Cref{example3} for $g = 2$ (on the left) and for $g \geq 3$ (on the right).}
\label{fig example3}
\end{figure}

\bibliographystyle{plain}
\bibliography{bibfile}

\begin{thebibliography}{10}

\bibitem{artin1971degenerate}
Michael Artin and Gayn Winters.
\newblock Degenerate fibres and stable reduction of curves.
\newblock {\em Topology}, 10(4):373--383, 1971.

\bibitem{arzdorf2012another}
Kai Arzdorf and Stefan Wewers.
\newblock Another proof of the semistable reduction theorem.
\newblock {\em arXiv preprint arXiv:1211.4624}, 2012.

\bibitem{bosch2012neron}
Siegfried Bosch, Werner L{\"u}tkebohmert, and Michel Raynaud.
\newblock {\em N{\'e}ron Models}, volume~21.
\newblock Springer Science \& Business Media, 2012.

\bibitem{bouw2015semistable}
Irene~I. Bouw and Stefan Wewers.
\newblock Semistable reduction of curves and computation of bad {E}uler factors
  of {L}-functions.
\newblock {\em \em {I}{C}{E}{R}{M} course notes}, 2015.

\bibitem{bouw2017computing}
Irene~I. Bouw and Stefan Wewers.
\newblock Computing {L}-functions and semistable reduction of superelliptic
  curves.
\newblock {\em Glasgow Mathematical Journal}, 59(1):77--108, 2017.

\bibitem{coleman1987computing}
Robert Coleman.
\newblock Computing stable reductions.
\newblock In {\em S{\'e}minaire de Th{\'e}orie des Nombres, Paris 1985--86},
  pages 1--18. Springer, 1987.

\bibitem{deligne1969irreducibility}
Pierre Deligne and David Mumford.
\newblock The irreducibility of the space of curves of given genus.
\newblock {\em Publications Math{\'e}matiques de l'IHES}, 36:75--109, 1969.

\bibitem{dokchitser2022arithmetic}
Tim Dokchitser, Vladimir Dokchitser, C{\'e}line Maistret, and Adam Morgan.
\newblock Arithmetic of hyperelliptic curves over local fields.
\newblock {\em Mathematische Annalen}, pages 1--110, 2022.

\bibitem{dokchitser2023note}
Vladimir Dokchitser and Adam Morgan.
\newblock A note on hyperelliptic curves with ordinary reduction over 2-adic
  fields.
\newblock {\em Journal of Number Theory}, 244:264--278, 2023.

\bibitem{gehrunger2021reduction}
Tim Gehrunger and Richard Pink.
\newblock Reduction of hyperelliptic curves in characteristic $\neq 2$.
\newblock {\em arXiv preprint arXiv:2112.05550}, 2021.

\bibitem{gehrunger2025reduction}
Tim Gehrunger and Richard Pink.
\newblock Reduction of hyperelliptic curves in residue characteristic 2.
\newblock {\em Journal of Number Theory}, 2025.

\bibitem{lehr2001reduction}
Claus Lehr.
\newblock Reduction of $p$-cyclic covers of the projective line.
\newblock {\em Manuscripta Mathematica}, 106(2):151--175, 2001.

\bibitem{lehr2006wild}
Claus Lehr, Michel Matignon, et~al.
\newblock Wild monodromy and automorphisms of curves.
\newblock {\em Duke Mathematical Journal}, 135(3):569--586, 2006.

\bibitem{liu2002algebraic}
Qing Liu.
\newblock {\em Algebraic geometry and arithmetic curves}, volume~6.
\newblock Oxford University Press on Demand, 2002.

\bibitem{matignon2003vers}
Michel Matignon.
\newblock Vers un algorithme pour la r{\'e}duction stable des rev{\^e}tements
  $p$-cycliques de la droite projective sur un corps $p$-adique.
\newblock {\em Mathematische Annalen}, 325(2):323--354, 2003.

\bibitem{shiomi2011deuring}
Daisuke Shiomi.
\newblock On the {D}euring-{S}hafarevich formula.
\newblock {\em Tokyo Journal of Mathematics}, 34(2):313--318, 2011.

\bibitem{silverman2009arithmetic}
{J}oseph~{H}. {S}ilverman.
\newblock \textit{The arithmetic of elliptic curves}.
\newblock {\em Graduate Texts in Mathematics}, 106, 2009.

\bibitem{subrao1975p}
Dor{\'e} Subrao.
\newblock The $p$-rank of {A}rtin-{S}chreier curves.
\newblock {\em Manuscripta Mathematica}, 16(2):169--193, 1975.

\bibitem{yelton2021semistable}
Jeffrey Yelton.
\newblock Semistable models of elliptic curves over residue characteristic 2.
\newblock {\em Canadian Mathematical Bulletin}, 64(1):154--162, 2021.

\end{thebibliography}

\end{document}